\PassOptionsToPackage{english}{babel}
\documentclass[11pt]{article}
\usepackage{amsmath, amsfonts, amssymb, amsthm,  graphicx, mathtools, enumerate, dsfont}

\usepackage[blocks, affil-it]{authblk}

\allowdisplaybreaks

\usepackage[square]{natbib}
\usepackage[CJKbookmarks=true,
            bookmarksnumbered=true,
			bookmarksopen=true,
			colorlinks=true,
			citecolor=red,
			linkcolor=blue,
			anchorcolor=red,
			urlcolor=blue]{hyperref}
\usepackage[usenames]{color}

\usepackage[letterpaper, left=1.2truein, right=1.2truein, top = 1.2truein, bottom = 1.2truein]{geometry}
\usepackage[ruled, vlined, lined, commentsnumbered]{algorithm2e}
\usepackage[font=small,labelfont=it]{caption}

\usepackage{prettyref,soul}
\usepackage[ngerman]{babel}
\usepackage{bbm}
\usepackage{multirow}
\usepackage{diagbox}
\usepackage{hyperref}
    \hypersetup{ colorlinks = true, linkcolor = blue, citecolor = blue}
\usepackage{makecell}

\newtheorem{lemma}{Lemma}
\newtheorem{proposition}{Proposition}
\newtheorem{thm}{Theorem}

\newtheorem{corollary}{Corollary}
\newtheorem{assumption}{Assumption}
\newtheorem{remark}{Remark}

\makeatletter
\newcommand{\neutralize}[1]{\expandafter\let\csname c@#1\endcsname\count@}
\makeatother

\newrefformat{eq}{(\ref{#1})}
\newrefformat{chap}{Chapter~\ref{#1}}
\newrefformat{sec}{Section~\ref{#1}}
\newrefformat{algo}{Algorithm~\ref{#1}}
\newrefformat{fig}{Fig.~\ref{#1}}
\newrefformat{tab}{Table~\ref{#1}}
\newrefformat{rmk}{Remark~\ref{#1}}
\newrefformat{clm}{Claim~\ref{#1}}
\newrefformat{def}{Definition~\ref{#1}}
\newrefformat{cor}{Corollary~\ref{#1}}
\newrefformat{lmm}{Lemma~\ref{#1}}
\newrefformat{lemma}{Lemma~\ref{#1}}
\newrefformat{prop}{Proposition~\ref{#1}}
\newrefformat{app}{Appendix~\ref{#1}}
\newrefformat{ex}{Example~\ref{#1}}
\newrefformat{exer}{Exercise~\ref{#1}}
\newrefformat{soln}{Solution~\ref{#1}}
\newrefformat{cond}{Condition~\ref{#1}}

\def\P{{\mathbb P}}

\def\text#1{\mbox{\rm #1}}

\newcommand{\floor}[1]{{\left\lfloor {#1}\right\rfloor}}

\newcommand{\RR}{\mathbb{R}}
\newcommand{\wh}{\widehat}

\newcommand{\supp}{{\rm supp}}

\newcommand{\TV}{{\sf TV}}

\newcommand{\G}{\mathbb{G}}

\newcommand{\E}{\mathbb{E}}
\renewcommand{\P}{\mathbb{P}}
\def\Q{\mathbb{Q}}

\newcommand{\XX}{\widetilde{X}_{ij}}
\newcommand{\bXX}{\mbf{\widetilde{X}}_{ij}}

\newcommand{\cov}{\text{Cov}}
\newcommand{\var}{\text{Var}}

\newcommand{\tr}{\text{Tr}}

\renewcommand{\exp}{\textup{exp}}
\renewcommand{\cal}{\mathcal}
\newcommand{\define}{:=}

\newcommand{\ms}{\quad~}

\newcommand{\td}{\widetilde}
\newcommand{\under}{\underline}
\newcommand{\mbf}{\boldsymbol}
\newcommand{\LP}{\text{LP}}
\newcommand{\indc}{\mathbbm{1}_{\Omega_n}}
\newcommand{\indcc}{\mathbbm{1}_{\Omega_n^c}}
\DeclarePairedDelimiter{\abs}{\lvert}{\rvert}
\DeclarePairedDelimiter{\bbrace}{\lbrace}{\rbrace}
\DeclarePairedDelimiter{\parr}{(}{)}
\DeclarePairedDelimiter{\fence}{[}{]}
\DeclarePairedDelimiter{\ceil}{\lceil}{\rceil}
\DeclarePairedDelimiter{\nm}{\|}{\|}

\let\hat\widehat
\let\bar\overline

\makeatletter
\newcommand{\leqnomode}{\tagsleft@true}
\newcommand{\reqnomode}{\tagsleft@false}
\makeatother

\begin{document}
\selectlanguage{english}

\setlength{\abovedisplayskip}{5pt}
\setlength{\belowdisplayskip}{5pt}
\setlength{\abovedisplayshortskip}{5pt}
\setlength{\belowdisplayshortskip}{5pt}

\title{Optimal estimation of variance in nonparametric regression with random design} 

\author[1]{Yandi Shen}
\author[2]{Chao Gao}
\author[3]{Daniela Witten}
\author[4]{Fang Han}
\affil[1]{
University of Washington
 \\
ydshen@uw.edu
}
\affil[2]{
University of Chicago
 \\
chaogao@galton.uchicago.edu
}
\affil[3]{
University of Washington
 \\
dwitten@uw.edu
}
\affil[4]{
University of Washington
 \\
fanghan@uw.edu
}

\date{}

\maketitle

\begin{abstract}
Consider the heteroscedastic nonparametric regression model with random design
\begin{align*}
Y_i = f(X_i) + V^{1/2}(X_i)\varepsilon_i, \quad i=1,2,\ldots,n,
\end{align*}
with $f(\cdot)$ and $V(\cdot)$ $\alpha$- and $\beta$-H\"older smooth, respectively. 
We show that the minimax rate of estimating $V(\cdot)$ under both local and global squared risks is of the order
\begin{align*}
n^{-\frac{8\alpha\beta}{4\alpha\beta + 2\alpha + \beta}} \vee n^{-\frac{2\beta}{2\beta+1}},
\end{align*}
where $a\vee b\define \max\{a,b\}$ for any two real numbers $a,b$. This result extends the fixed design rate $n^{-4\alpha} \vee n^{-2\beta/(2\beta+1)}$ derived in \cite{wang2008effect} in a non-trivial manner, as indicated by the appearances of both $\alpha$ and $\beta$ in the first term. 
In the special case of constant variance, we show that the minimax rate is $n^{-8\alpha/(4\alpha+1)}\vee n^{-1}$ for variance estimation, which further implies the same rate for quadratic functional estimation and thus unifies the minimax rate under the nonparametric regression model with those under the density model and the white noise model. To achieve the minimax rate, we develop a U-statistic-based local polynomial estimator and a lower bound that is constructed over a specified distribution family of randomness designed for both $\varepsilon_i$ and $X_i$.
\end{abstract}

{\bf Keywords:} variance estimation, nonparametric regression, random design, minimax rate, U-statistics.

\section{Introduction}
\label{sec:intro}

Consider the model 
\begin{align}
\label{model:var_function}
Y_i = f(X_i) + V^{1/2}(X_i)\varepsilon_i,\quad i=1,2,\ldots,n,
\end{align}
where $\{X_i\}_{i=1}^n$ are independent and identically distributed (i.i.d.) univariate random design points, and $\{\varepsilon_i\}_{i=1}^n$ are i.i.d. with zero mean, unit variance, and are independent of $\{X_i\}_{i=1}^n$. In this paper, we study the optimal estimation of $V(\cdot)$ under both local and global squared risks. 
 Variance estimation is a fundamental statistical problem \citep{von1941distribution,von1942further,rice1984bandwidth,hall1990asymptotically} with wide applications. It is useful in, for example, construction of confidence bands for the mean function, estimation of the signal-to-noise ratio \citep{verzelen2018adaptive}, and selection of the optimal kernel bandwidth \citep{fan1992design}.
 
When $\{X_i\}_{i=1}^n$ are fixed, estimation of $V(\cdot)$ in \eqref{model:var_function} has been studied extensively in the literature via residual-based methods \citep{hall1989variance,ruppert1997local,hardle1997local,fan1998efficient} and difference-based methods \citep{muller1987estimation,muller2003estimating,brown2007variance,wang2008effect}. One important heuristic from previous studies is that, compared to residual-based methods, difference-based methods are able to achiever a smaller bias and subsequently a smaller mean squared error by avoiding direct estimation of the mean function. More precisely, when $X_i = i/n$, $i=1,\ldots,n$ and $f(\cdot)$ and $V(\cdot)$ in \eqref{model:var_function} are $\alpha$- and $\beta$-H\"older smooth, respectively, \cite{wang2008effect} proposed a difference estimator which achieved the optimal rate of the order $n^{-4\alpha}\vee n^{-\frac{2\beta}{2\beta + 1}}$ under both local and global squared risks. 

In contrast, our study focuses on the case where $\{X_i\}_{i=1}^n$ are i.i.d. random design points on the real line. For this, we show that when $f(\cdot)$ and $V(\cdot)$ in \eqref{model:var_function} are $\alpha$- and $\beta$-H\"older smooth, respectively, the minimax rate of estimating $V(\cdot)$ is of the order $n^{-\frac{8\alpha\beta}{4\alpha\beta + 2\alpha + \beta}}\vee n^{-\frac{2\beta}{2\beta+1}}$ under both local and global squared risks. This result has several noteworthy implications:
\begin{itemize}
\item The minimax rates in random and fixed design settings share a common component, $n^{-\frac{2\beta}{2\beta+1}}$, as well as the same transition boundary $\alpha = \beta/(4\beta + 2)$.
\item For $\alpha < \beta/(4\beta + 2)$, a faster rate is achievable with a random design.
\item Unlike the fixed design setting, for $\alpha < \beta/(4\beta + 2)$, $\alpha$ and $\beta$ are now both present in the first term of the minimax rate in the random design case.
\end{itemize}

We now discuss in more detail this minimax rate. 
The upper bound of the minimax rate is achieved by smoothing pairwise differences via local polynomial regression, the former of which is formulated via U-statistics. Our analysis of this estimator hence relies on the four-term Bernstein inequality in \cite{gine2000exponential}, and unlike classic kernel methods, requires no smoothness assumption on the design density. 

For the lower bound, due to the appearances of both $\alpha$ and $\beta$ in the non-trivial $n^{-\frac{8\alpha\beta}{4\alpha\beta + 2\alpha + \beta}}$ part of the minimax rate and the additional randomness of $\{X_i\}_{i=1}^n$, the derivation is much more involved than its counterpart in the fixed design setting. We tackle the first difficulty of entangled $\alpha$ and $\beta$ via a proper localization technique in the construction of the mean function $f(\cdot)$, depicted in Figure \ref{fig:v_point_lower} in Section \ref{sec:lb3}. The second difficulty caused by the randomness of $\{X_i\}_{i=1}^n$ is resolved with a new trapezoid-shaped construction of the mean $f(\cdot)$, aided by a result due to \cite{kolchin1978random} on the sparse multinomial distribution. This result helps characterize the asymptotic behavior of the locations of $\{X_i\}_{i=1}^n$ and plays a key role in our proof, but to our knowledge has not been well used in the nonparametric statistics literature. 

In the special case of constant variance, \eqref{model:var_function} is reduced to
\begin{align}
\label{model:nonpar}
Y_i = f(X_i) + \sigma\varepsilon_i, \quad i=1,2,\ldots,n,
\end{align}
and the goal becomes estimation of $\sigma^2$. 
In this case, the problem is linked to estimation of a quadratic functional, which has been studied in depth in the other two benchmark nonparametric models, the density model \citep{bickel1988estimating, laurent1996efficient, gine2008simple} and the white noise model \citep{donoho1990minimax, fan1991estimation, laurent2000adaptive}. In the density model, one observes an i.i.d. univariate sequence $\{X_i\}_{i=1}^n$ from some unknown density $f(\cdot)$, and the goal is to estimate $\int f^2(x)dx$. In the white noise model, one observes a continuous-time process from $dY_t = f(t)dt + n^{-1/2} dW_t$ for $t\in[0,1]$ with $W_t$  a standard Wiener process. The goal is to estimate $\int_0^1 f^2(t)dt$. Under an $\alpha$-smoothness condition on $f(\cdot)$, the minimax rate in both of the aforementioned two cases is $n^{-8\alpha/(4\alpha+1)}\vee n^{-1}$ (cf. Theorem 1(ii) and 2(ii) in \cite{bickel1988estimating}, Theorem 4 in \cite{fan1991estimation}). 

Following \cite{doksum1995nonparametric}, a quadratic functional of interest under \eqref{model:nonpar} with random design is
\begin{align}
\label{eq:Q}
Q \define \int f^2(x)p_X(x)w(x)dx,
\end{align}
where $p_X(\cdot)$ is the unknown design density and $w(\cdot)\geq 0$ is some known weight function. Assuming in \eqref{model:nonpar} that $f$ is $\alpha$-H\"older smooth, we show that 
the minimax rate of estimating $\sigma^2$ and $Q$ (when $\sigma^2$ is unknown) is $n^{-8\alpha/(4\alpha+1)}\vee n^{-1}$, thereby unifying the minimax rate of quadratic functional estimation in all three benchmark nonparametric models.




In this paper, we also provide extensions of \eqref{model:nonpar} to multivariate cases, with a focus on the multivariate nonparametric regression model
\begin{align}
\label{model:multivariate_nonpar}
Y_i &= f(\mbf{X}_i) + \sigma\varepsilon_i, \quad i=1,2,\ldots, n,
\end{align}
and the nonparametric additive model
\begin{align}
\label{model:additive}
Y_i &= \sum_{k=1}^d f_k(X_{i,k}) + \sigma\varepsilon_i, \quad i=1,2,\ldots,n,
\end{align}
in both fixed and random designs. Here, $\mbf{X}_i \define (X_{i,1},\ldots,X_{i,d})^\top$, $i=1,\ldots,n$, for some fixed positive integer $d$. Regarding the fixed design, we consider two types, namely, the grid design (GD) and the diagonal design (DD). With a total of $n$ design points, the former places them on a regular grid in the $d$-dimensional cube $[0,1]^d$ while the latter only places design points on the diagonal. Details are given in Sections \ref{subsec:multivariate_nonpar} and \ref{subsec:additive}.

\begin{table}[t!]
\centering
\caption{Summary of minimax rates in \eqref{model:var_function}, \eqref{model:nonpar}, \eqref{model:multivariate_nonpar} and \eqref{model:additive}. The two types of fixed design considered, (GD) and (DD), are defined in \eqref{eq:GD} and \eqref{eq:DD}, respectively. For a $d$-dimensional smoothness index $\mbf{\alpha} = (\alpha_1,\ldots,\alpha_d)^\top$, $\under{\alpha}\define d/(\sum_{k=1}^d 1/\alpha_k)$, $\alpha_{\min} \define \min_{1\leq k\leq d}\alpha_k$, and $\alpha_{\max} \define \max_{1\leq k\leq d}\alpha_k$. The respective sections contain the definition of the distribution class of $\{(X_i,\varepsilon_i)\}_{i=1}^n$ in the random design setting and distribution class of $\{\varepsilon_i\}_{i=1}^n$ in the fixed design setting. Our results include all of the random design rates and fixed design rates in \eqref{model:multivariate_nonpar} and \eqref{model:additive}. Note results for \eqref{model:multivariate_nonpar} and \eqref{model:additive} have additional requirements; see Sections \ref{subsec:multivariate_nonpar} and \ref{subsec:additive} for details.}
\begin{tabular}{| l | c | c | c |} 
 \hline
   & stated in & minimax rate & boundary\\ [0.5ex] 
 \hline
 \eqref{model:var_function}, fixed & \cite{wang2008effect} & $n^{-4\alpha}\vee n^{-2\beta/(2\beta+1)}$ & \multirow{2}{*}{$\alpha = \beta/(4\beta + 2)$} \\
 \eqref{model:var_function}, random & Theorems \ref{thm:v_function_upper}, \ref{thm:v_lower_point}, \ref{thm:v_lower_int} & $n^{-\frac{8\alpha\beta}{4\alpha\beta + \beta + 2\alpha}} \vee n^{-\frac{2\beta}{2\beta+1}}$ & \\[1ex] 
 \hline
 \eqref{model:nonpar}, fixed  & \cite{wang2008effect} & $n^{-4\alpha}\vee n^{-1}$ & \multirow{2}{*}{$\alpha = 1/4$} \\ 
 \eqref{model:nonpar}, random & Theorems \ref{thm:nonpar_upper}, \ref{thm:nonpar_lower} & $n^{-8\alpha/(4\alpha+1)}\vee n^{-1}$ & \\
 \hline
 \eqref{model:multivariate_nonpar}, fixed (GD) & Proposition \ref{prop:multivariate_fixed_GD} & $n^{-4\alpha_{\max}/d}\vee n^{-1}$ & $\alpha_{\max} = d/4$\\
 \eqref{model:multivariate_nonpar}, fixed (DD) & Proposition \ref{prop:multivariate_fixed_DD} & $n^{-4\alpha_{\min}}\vee n^{-1}$ & $\alpha_{\min} = 1/4$ \\  
 \eqref{model:multivariate_nonpar}, random & Propositions \ref{prop:multivariate_upper}, \ref{prop:multivariate_lower} & $n^{-8\under{\alpha}/(4\under{\alpha} + d)}\vee n^{-1}$ & $\under{\alpha} = d/4$ \\
 \hline
 \eqref{model:additive}, fixed (GD) & Proposition \ref{prop:additive_fixed_GD} & $n^{-1}$ &\diagbox[height = 0.6cm, width=4cm]{}{} \\\cline{4-4}
 \eqref{model:additive}, fixed (DD) & Proposition \ref{prop:additive_fixed_DD} & $n^{-4\alpha_{\min}} \vee n^{-1}$ & \multirow{2}{*}{$\alpha_{\min} = 1/4$} \\
 \eqref{model:additive}, random & Propositions \ref{prop:additive_random_lower}, \ref{prop:additive_random_independent} & $n^{-8\alpha_{\min}/(4\alpha_{\min}+1)} \vee n^{-1}$ & \\
 \hline
\end{tabular}
\label{table:1}
\end{table}

We summarize \!the \!minimax \!rates in all of the \!aforementioned \!models in \!Table \!\ref{table:1}. 

The rest of the paper is organized as follows. Section \ref{sec:nonpar} presents the simple model \eqref{model:nonpar} with constant variance. Section \ref{sec:v_function} discusses its heteroscedastic extension   \eqref{model:var_function}. Section \ref{sec:discussion} discusses the multivariate nonparametric regression model \eqref{model:multivariate_nonpar}, the additive model \eqref{model:additive}, and several other extensions of our main results. The essential lower bound proof of the minimax rate $n^{-8\alpha/(4\alpha+1)}\vee n^{-1}$ under model \eqref{model:nonpar} is presented in Section \ref{sec:proof}, with the rest of the proofs given in a supplement.

The notation used throughout the paper is as follows. For any positive integer $n$, $[n]$ denotes the set $\{1,2,\ldots,n\}$. For any real number $a$, we use $\ceil{a}$ to denote the smallest integer greater than or equal to $a$, and $\floor{a}$ the largest integer strictly smaller than $a$. For any positive integer $d$, $\mbf{0}_d$ denotes the zero vector of dimension $d$ and $\mathbf{I}_d$ denotes the identity matrix of dimension $d$. For a real vector $x$, $\|x\|$ and $\|x\|_\infty$ denote its Euclidean and infinity norms, respectively. For a real matrix $\mathbf{A}$, we use $\|\mathbf{A}\|$, $\|\mathbf{A}\|_F$, and $|\mathbf{A}|$ to denote its spectral norm, Frobenius norm, and determinant, respectively. For an $m$-times differentiable function $f:\RR\rightarrow\RR$ with some positive integer $m$, we use $f^{(k)}$ to denote its $k$th derivative for $k=1,2,\ldots,m$. For identically distributed random variables $X_i$ and $X_j$, we use $\P_{X_i}(\cdot)$ and $p_{X_i}(\cdot)$ to denote the distribution and density of $X_i$, $\XX$ to denote $X_i - X_j$, and $p_{\XX}(\cdot)$ to denote the density of $X_i - X_j$. Similar notation $\P_{\mbf{X}_i}(\cdot), p_{\mbf{X}_i}(\cdot), \bXX, p_{\bXX}(\cdot)$ applies to identically distributed random vectors $\mbf{X}_i$ and $\mbf{X}_j$. For a positive integer $d$ and $\mbf{\mu}\in\RR^d,\mbf{\Sigma}\in\RR^{d\times d}$, $\cal{N}_d(\mbf{\mu},\mbf{\Sigma})$ stands for the $d$-dimensional normal distribution with mean $\mbf{\mu}$ and covariance $\mbf{\Sigma}$. We will drop the subscript $d$ for simplicity when $d=1$. $\Phi(\cdot)$ and $\varphi(\cdot)$ represent the standard normal distribution and density. More generally, we will write $\varphi_{\mu,\sigma^2}(\cdot)$ as the density for the normal distribution with mean $\mu$ and variance $\sigma^2$. For two probability measures $\P,\Q$ defined on a common space $(\Omega,\cal{A})$, $\TV(\P,\Q)$ denotes their total variation distance, that is, $\TV(\P,\Q) \define \sup_{A\in\cal{A}}\abs*{\P(A)-\Q(A)}$. For two real sequences $\{a_n\}$ and $\{b_n\}$, $a_n\lesssim b_n$ if $|a_n|\leq C|b_n|$ for some positive absolute constant $C$. We say $a_n\asymp b_n$ if $a_n\lesssim b_n$ and $b_n\lesssim a_n$. 


%

\section{Homoscedastic case}
\label{sec:nonpar}

To illustrate some of the main ideas developed in this paper, we begin with a discussion of the elementary univariate homoscedastic nonparametric regression model \eqref{model:nonpar}:
\begin{align*}
Y_i = f(X_i) + \sigma\varepsilon_i, \quad i=1,2,\ldots,n.
\end{align*}
Here, $\{X_i\}_{i=1}^n$ are i.i.d. copies of a univariate random variable $X$, $f(\cdot)$ belongs to an $\alpha$-H\"older class that will be specified soon, and $\{\varepsilon_i\}_{i=1}^n$ are i.i.d. copies of a variable $\varepsilon$ with zero mean and unit variance and are independent of $\{X_i\}_{i=1}^n$. Both the mean function $f(\cdot)$ and the distribution of $\{X_i\}_{i=1}^n$ are assumed unknown. 

Model \eqref{model:nonpar} has been extensively studied using residual-based and difference-based methods; see, among many others, \cite{von1941distribution}, \cite{von1942further}, \cite{rice1984bandwidth}, \cite{gasser1986residual}, \cite{hall1990asymptotically}, \cite{hall1990variance}, \cite{thompson1991noise}, \cite{muller2003estimating}, \cite{wang2008effect}. A related functional estimation problem has also been studied in semiparametric models \citep{robins2008higher,robins2009semiparametric}. Most of the previous studies focus on the case of fixed design, especially the equidistant design with $X_i=i/n$, $i\in[n]$, for which the minimax rate of estimating $\sigma^2$ under an $\alpha$-H\"older smoothness constraint on $f(\cdot)$ is known to be $n^{-4\alpha}\vee n^{-1}$ (cf. Theorems 1 and 2 in \cite{wang2008effect}).



In detail, let $I$ be a fixed (possibly infinite) interval on the real line. Define the H\"older class $\Lambda_{\alpha,I}(C_\cal{F})$ on $I$ as follows:
\begin{equation}
\begin{aligned}
\label{eq:holder}
\Lambda_{\alpha,I}(C_\cal{F}) \define \big\{f: &\text{ for all } x, y\in I\text{ and } k = 0,\ldots, \floor{\alpha},\\
& \abs*{f^{(k)}(x)}\leq C_\cal{F} \text{ and } \abs*{f^{(\floor{\alpha})}(x) - f^{(\floor{\alpha})}(y)}\leq C_\cal{F}|x-y|^{\alpha^\prime}\big\},
\end{aligned}
\end{equation}
where $\alpha^\prime \define \alpha - \floor{\alpha}$. Denote the support of $X$ as $\supp(X)$.

Define 
the joint distribution class $\cal{P}_{\tiny{\text{cv}},(X,\varepsilon)}$ (where ``cv" stands for ``constant variance") with the following conditions:
\begin{itemize}
\item[(a)] $X$ satisfies $\supp(X)\subset I$.
\item[(b)] $X$ has density $p_X(\cdot)$ and there exists a fixed positive constant $C_0$ such that 
\begin{align*}
\sup_{x\in\RR} p_{X}(x) \leq C_0.
\end{align*}
\item[(c)]There exist two fixed constants $\delta_0 > 0$ and $c_0 > 0$ such that for any  $0 < \delta < \delta_0$, there exists a set $\cal{U}_\delta\subset[-1,1]$ such that
\begin{align*}
\lambda(\cal{U}_\delta) \geq c_0 \quad \text{ and } \quad\inf_{u\in\cal{U}_\delta}p_{\XX}(u\delta) \geq c_0,
\end{align*}
where $\lambda(\cdot)$ represents the Lebesgue measure on the real line, and $\XX = X_i - X_j$.
\item[(d)]$\E\varepsilon^4 \leq C_\varepsilon$ for some fixed positive constant $C_\varepsilon$.
\end{itemize}

Note that no smoothness condition is placed on the density of $X$. Condition (c) essentially requires the density $p_{\XX}$ to be ``dense" around $0$, and is strictly weaker than a uniform lower bound of $p_{\XX}$ over a fixed neighborhood of $0$. It also follows from the following sufficient condition on the marginal density $p_X(\cdot)$ (see Lemma \ref{lemma:marginal_condition} in the supplement for the justification):
\begin{itemize}
\item[(c$^\prime$)] $X$ is compactly supported (taken to be $[0,1]$ without loss of generality). There exists some positive constant $c_0$ and subset $S\subset[-1,1]$ with Lebesgue measure $\lambda(S)\geq 3/4$ such that $p_X(t)\geq c_0$ uniformly over $t\in S$.
\end{itemize}
In particular, (c$^\prime$) covers the uniform distribution on $[0,1]$ and the distribution of $X$ in the lower bound construction in the proof of Theorem \ref{thm:nonpar_lower}.

The rest of the section is devoted to proving, for any fixed positive constants $C_\cal{F}$ and $C_\sigma$, the following minimax rate: 
\begin{align}
\label{eq:nonpar_minimax}
\inf_{\td{\sigma}^2}\sup_{f\in\Lambda_{\alpha,I}(C_\cal{F})}\sup_{\sigma^2 \leq C_\sigma}\sup_{\P_{(X,\varepsilon)}\in\cal{P}_{\tiny{\text{cv}},(X,\varepsilon)}} \E\parr*{\td{\sigma}^2 - \sigma^2}^2 \asymp n^{-8\alpha/(4\alpha+1)}\vee n^{-1},
\end{align}
where $\P_{(X,\varepsilon)}$ denotes the joint distribution of $(X,\varepsilon)$, and $\td{\sigma}^2$ ranges over all estimators of $\sigma^2$.

\subsection{Upper bound}
\label{subsec:nonpar_upper}
The upper bound is achieved by a difference estimator based on U-statistics (with convention $0/ 0 =0$):
\begin{align}
\label{eq:nonpar_estimator}
\hat{\sigma}^2 \define \frac{{n\choose 2}^{-1}\sum_{i<j}K_h(X_i - X_j)(Y_i-Y_j)^2/2}{{n\choose 2}^{-1}\sum_{i<j}K_h(X_i-X_j)}.
\end{align}
Here, $K_h(\cdot) \define K(\cdot/h)/h$, where $h = h_n$ is a bandwidth parameter satisfying $h_n\downarrow 0$ as $n\rightarrow \infty$, and $K(\cdot)$ is a symmetric density kernel supported on $[-1,1]$ that satisfies 
\begin{align}
\label{eq:kernel}
\under{M}_K \leq \inf_{|u|\leq 1}K(u) \leq \sup_{|u|\leq 1} K(u) \leq \overline{M}_K
\end{align}
for two fixed constants $\overline{M}_K$ and $\under{M}_K$; one example is the box kernel $K(u) = \mathbbm{1}\bbrace*{|u|\leq 1}/2$ which satisfies \eqref{eq:kernel} with $\overline{M}_K = \under{M}_K = 1/2$. 

The following error bound is derived via the exponential inequality for degenerate U-statistics due to \cite{gine2000exponential}.

\begin{thm}
\label{thm:nonpar_upper}
Suppose the kernel $K(\cdot)$ in $\widehat{\sigma}^2$ is chosen such that \eqref{eq:kernel} is satisfied with constants $\overline{M}_K$ and $\under{M}_K$, and the bandwidth $h_n$ is chosen as 
\begin{align}
\label{eq:h_opt_nonpar}
h_n\asymp
\begin{cases}
n^{-2/(4\alpha+1)}, & 0 < \alpha < 1/4,\\
n^{-1}, & \alpha \geq 1/4.
\end{cases}
\end{align}
Then, under \eqref{model:nonpar} with random design, it holds that
\begin{align*}
\sup_{f\in\Lambda_{\alpha,I}(C_\cal{F})}\sup_{\sigma^2 \leq C_\sigma}\sup_{\P_{(X,\varepsilon)}\in\cal{P}_{\tiny{\text{cv}},(X,\varepsilon)}}\E\parr*{\hat{\sigma}^2 - \sigma^2}^2\leq C\parr*{n^{-8\alpha/(4\alpha + 1)}\vee n^{-1}},
\end{align*}
where $C$ is some fixed positive constant that only depends on $\overline{M}_K, \under{M}_K$, $\alpha,C_\cal{F},C_\sigma$ and $C_0,c_0,C_\varepsilon$ in $\cal{P}_{\tiny{\text{cv}},(X,\varepsilon)}$.
\end{thm}



\begin{remark}
\label{remark:bias_var_nonpar}
The error rate in Theorem \ref{thm:nonpar_upper} is achieved by choosing the optimal bandwidth $h_n$ to balance the ``bias-variance" decomposition:
\begin{align}
\label{eq:bias_var_nonpar}
\bbrace*{\E\parr*{\hat{\sigma}^2 - \sigma^2}^2}^{1/2} \lesssim h_n^{2(\alpha\wedge 1)} + \frac{1}{nh_n^{1/2}},
\end{align} 
where $a\wedge b\define \min\{a,b\}$ for any two real numbers $a,b$. The bias term $h_n^{2(\alpha\wedge 1)}$ reflects the second-order effect of the unknown mean on variance estimation, which has been noted by \cite{hall1989variance} and \cite{wang2008effect}. The variance part follows from the fact that there is an average number of $n^2h_n$ pairs of $(i,j)$ such that $|X_i - X_j|\leq h_n$. We note that the same ``bias-variance" decomposition has appeared in quadratic functional estimation in the density model and Gaussian sequence model \citep{bickel1988estimating, fan1991estimation, gine2008simple}. See Section \ref{subsec:quad_est} for a more detailed discussion.     
\end{remark}

\begin{remark}
\label{remark:compare_muller}
While most of the previous works are in the context of fixed design, \cite{muller2003estimating} considered constant variance estimation with random design, and their estimator (formula (1.4) therein) is almost identical to our $\widehat{\sigma}^2$. Under certain assumptions (Assumptions 1 and 2 and (2.4) - (2.7) therein), they show that their estimator is root-n consistent and asymptotically normal. However, as commented in the first paragraph on p. 184 of their paper, their condition (2.7) is only satisfied when the mean function smoothness $\alpha$ is strictly larger than $1/4$, and no analysis is provided below this threshold. Our minimax rate $n^{-8\alpha/(4\alpha+1)}\vee n^{-1}$ therefore confirms that $\alpha\geq 1/4$ is indeed the minimal requirement for any variance estimator to be root-n consistent and we also demonstrate the optimality of $\widehat{\sigma}^2$ for $0 < \alpha < 1/4$.
\end{remark}


Finally, in \eqref{model:nonpar}, we have assumed that the smoothness index $\alpha$ is known. If it is unknown, then the variance can be estimated adaptively via Lepski-type methods \citep{lepskii1991problem,lepskii1992asymptotically}. This is discussed in more detail in Section \ref{subsec:adaptive}.

\subsection{Lower bound}
\label{subsec:nonpar_lower}
The derivation of the lower bound in \eqref{eq:nonpar_minimax} is much more involved. In particular, the construction in the fixed design setting (cf. Theorem 2 in \cite{wang2008effect}) cannot be extended to the random design case, since the spike-type construction of $f(\cdot)$ located at each deterministic design point leads to a sub-optimal rate in the random design setting. To achieve a sharp rate, we have to exploit the randomness of $\{X_i\}_{i=1}^n$; this requires us to handle a highly convoluted alternative hypothesis  that no longer leads to a product measure of $\{Y_i\}_{i=1}^n$ given each realization of $\{X_i\}_{i=1}^n$ in LeCam's two-point method. This calls for a careful analysis of the locations of $\{X_i\}_{i=1}^n$.

We now sketch a proof of the $n^{-8\alpha/(4\alpha+1)}$ component in \eqref{eq:nonpar_minimax} for $0 < \alpha < 1/4$, with a particular emphasis on where the difference arises with the fixed design setting. 
The proof can be roughly divided into two steps. In the first step, we construct a two-point testing problem with the null being a Gaussian ($H_0$) and the alternative a Gaussian location mixture ($\td{H}_1$). In the second step, we approximate the Gaussian location mixture ($\td{H}_1$) by a location mixture with compact support ($H_1$), which, unlike the alternative in the first step, belongs to the considered model class. 

We start by introducing the construction of $f(\cdot)$, $\sigma^2$, $\varepsilon$, and $X$ under the null $H_0$ and the alternative $\td{H}_1$ in the first step. For each $n$, let
\begin{align*}
h_n \asymp n^{-2/(4\alpha + 1)}, \quad \theta_n^2 \asymp h_n^{2\alpha},\quad \text{ and }\quad N\define 1/(6h_n),
\end{align*}
and divide the unit interval $[0,1]$ into $N$ intervals of length $6h_n$, with $n$ large enough and $h_n$ chosen such that $N$ is a positive integer. 
\begin{itemize}
\item[]\emph{Choice of $f(\cdot)$}: Under $H_0$, let $f\equiv 0$. Under $\td{H}_1$, let $f(\cdot)$ be a piecewise trapezoidal function on the $N$ intervals. That is, for each $i\in[N]$, $f$ takes on a value of $h_n^\alpha\td{r}_i$ on the intervals $[(6i-5)h_n, (6i-1)h_n]$ and then linearly decreases to zero on the two endpoints $6(i-1)h_n$ and $6ih_n$, with $\{\td{r}_i\}_{i=1}^N$ i.i.d. standard normal variables.
\item[]\emph{Choice of $\sigma^2$}: Under $H_0$, let $\sigma^2 = 1+\theta_n^2$. Under $\td{H}_1$, let $\sigma^2 = 1$.
\item[]\emph{Choice of $\varepsilon$}: Under both $H_0$ and $\td{H}_1$, let $\varepsilon\sim\cal{N}(0,1)$.
\item[]\emph{Choice of $X$}: Under both $H_0$ and $\td{H}_1$, let $\{X_i\}_{i=1}^n$ be uniformly distributed over the union of the upper bases of the trapezoids, that is, over $\bigcup_{i=1}^N[(6i-5)h_n, (6i-1)h_n]$.
\end{itemize}
See Figure \ref{fig:nonpar_lower} for an illustration of the construction.

\begin{figure}[t!]
\centering
\includegraphics[width = 0.7\textwidth]{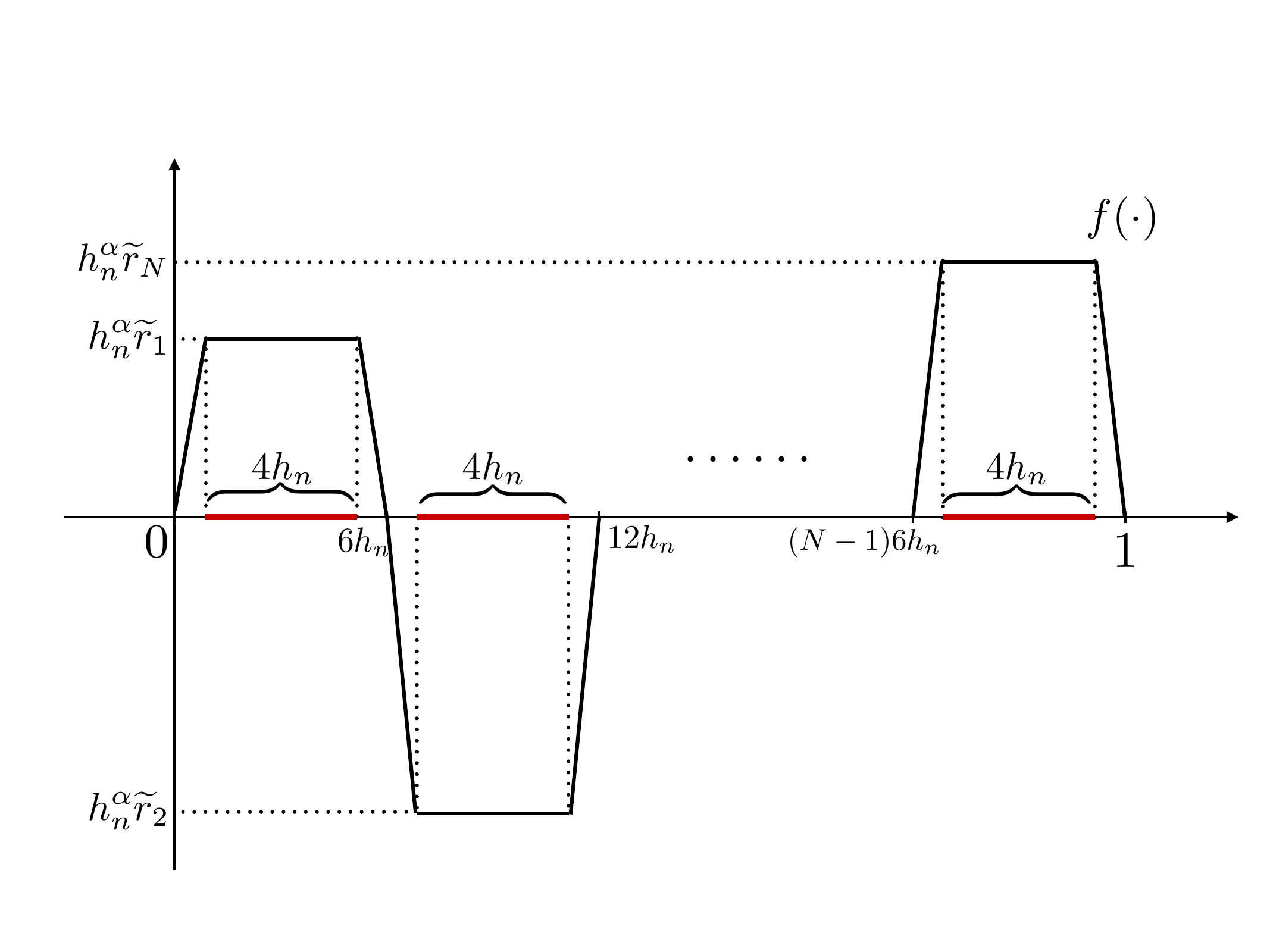}
\caption{The black solid line represents the construction of $f(\cdot)$ under the alternative hypothesis $\td{H}_1$. The thick red segments indicate the support of $X$ under both $H_0$ and $\td{H}_1$, on which $X$ is uniformly distributed. Here, $h_n\asymp n^{-2/(4\alpha + 1)}$ and is chosen such that $N\define 1/(6h_n)$ is a positive integer. $\{\td{r}_i\}_{i=1}^N$ are $N$ i.i.d. standard normal variables.}
\label{fig:nonpar_lower}
\end{figure}

 In contrast to the spike-type construction of $f(\cdot)$ in the fixed design setting, our construction is trapezoid-shaped, 
which guarantees a maximal variation in the mean to compensate for the difference in the variance under the null and alternative. 
This is unnecessary in the fixed design setting since the point of \!maximal variation in \!the mean (center of each spike) can be directly placed at each fixed $X_i = i/n$, resulting in $n$ \!evenly spaced spikes in $f(\cdot)$. 

Denote the joint distribution of $\{(X_i, Y_i)\}_{i=1}^n$ under $H_0$ and $\td{H}_1$ by $\P_0$ and $\td{\P}_1$ with respective density $p_0$ and $\td{p}_1$.  Under the above construction, conditional on $\{X_i\}_{i=1}^n$, $\{Y_i\}_{i=1}^n$ are distributed as 
\begin{align*}
H_0: p_0(\{Y_i\}_{i=1}^n\mid\{X_i\}_{i=1}^n) = \prod_{i=1}^n\varphi_{0,1+\theta_n^2}(Y_i)
\end{align*}
and 
\begin{align*}
\quad \td{H}_1: \td{p}_1(\{Y_i\}_{i=1}^n\mid \{X_i\}_{i=1}^n) = \prod_{j=1}^N \int \parr*{\prod_{\{i:b_i = j\}}\varphi_{h_n^\alpha v, 1}(Y_i)}\varphi(v)dv,
\end{align*}
where $\{b_i\}_{i=1}^n$ is the location index sequence of $\{X_i\}_{i=1}^n$ defined as 
\[
b_i := j~~~{\rm if}~~X_i\in[(6j-5)h_n, (6j-1)h_n], 
\]
which characterizes which trapezoid each $X_i$ falls into. Using Lemma \ref{lemma:gaussian_tv} that will be stated in Section \ref{sec:proof}, one can then upper bound 
\[
\TV(\P_0, \td{\P}_1)  = \E\TV(\P_0(\{Y_i\}_{i=1}^n\mid\{X_i\}_{i=1}^n), \td{\P}_1(\{Y_i\}_{i=1}^n\mid \{X_i\}_{i=1}^n)) \lesssim \theta_n^2nh_n^{1/2},
\]
which can be made smaller than a sufficiently small constant $c$ by choosing $h_n$ sufficiently small.

The second step of the proof aims to find a sequence of bounded random variables $\{r_i\}_{i=1}^N$ to replace the standard normal sequence $\{\td{r}_i\}_{i=1}^N$ in $\td{\P}_1$, so that for each realization of $\{r_i\}_{i=1}^N$, the corresponding $f(\cdot)$ in the alternative is $\alpha$-H\"older smooth with a fixed constant. Then, denoting the distribution of $\{r_i\}_{i=1}^N$ as $\G$, one wishes to approximate the conditional distribution $\td{\P}_1(\{Y_i\}_{i=1}^n\mid \{X_i\}_{i=1}^n)$ in $\td{H}_1$ by $\P_1(\{Y_i\}_{i=1}^n\mid \{X_i\}_{i=1}^n)$ with density 
\begin{align*}
p_1(\{Y_i\}_{i=1}^n\mid \{X_i\}_{i=1}^n) = \prod_{j=1}^N \int \parr*{\prod_{\{i:b_i = j\}}\varphi_{h_n^\alpha v, 1}(Y_i)}\G(dv)
\end{align*}
in $H_1$. 
Even with the aid of moment matching techniques already established in the literature, upper bounding $\TV(\P_1,\td{\P}_1)$ is still nontrivial. 
Specifically, unlike in the fixed design setting, now with high probability the conditional distribution of $\{Y_i\}_{i=1}^n$ given $\{X_i\}_{i=1}^n$ is no longer a product measure. This is because multiple $X_i$'s could fall into the same trapezoid in the construction of $f(\cdot)$. This can be handled relatively easily in the first step since there we only have to analyze the pairwise correlation of $Y_i\mid X_i$ and $Y_j \mid X_j$ depending on whether $X_i$ and $X_j$ fall into the same trapezoid, but it is much less tractable in the second step. More specifically, in order to match moments, we now have to divide the $X_i$'s into groups based on their memberships among the trapezoids, which naturally requires us to
monitor the locations of $\{X_i\}_{i=1}^n$, and in particular the number of $X_i$'s that fall into the same trapezoid. This is possible by observing that the memberships of $\{X_i\}_{i=1}^n$ now follow a sparse multinomial distribution ($n^{2/(4\alpha + 1)}$ bins, $n$ balls) so that a result in \cite{kolchin1978random} can be applied.
This allows us to show that with high probability the maximum number of $X_i$'s in each trapezoid is bounded by a fixed constant,
which, along with Lemma \ref{lemma:moment_matching} in Section \ref{sec:proof}, allows us to calculate 
\[
\TV(\P_1,\td{\P}_1)\lesssim n\theta_n^{2p}
\]
for $p \define 1 + \ceil{1/4\alpha}$. This indicates that $\TV(\P_1,\td{\P}_1)$ is smaller than some sufficiently small constant $c$. Then, by the triangle inequality,
\[
\TV(\P_0, \P_1) \leq \TV(\P_0, \td{\P}_1) + \TV(\P_1, \td{\P}_1) \leq 2c.
\]

Details of the above derivation will be given in Section \ref{sec:proof}. The resulting lower bound is as follows.

\begin{thm}
\label{thm:nonpar_lower}
Under \eqref{model:nonpar} with random design, it holds that
\begin{align*}
\inf_{\td{\sigma}^2}\sup_{f\in\Lambda_{\alpha,I}(C_\cal{F})}\sup_{\sigma^2 \leq C_\sigma}\sup_{\P_{(X,\varepsilon)}\in\cal{P}_{\tiny{\text{cv}},(X,\varepsilon)}}\E\parr*{\td{\sigma}^2 - \sigma^2}^2 \geq c\parr*{n^{-8\alpha/(4\alpha+1)}\vee n^{-1}},
\end{align*}
where $c$ is some fixed positive constant that only depends on $\alpha,C_\cal{F}, C_\sigma$ and $C_0, c_0,C_\varepsilon$ in $\cal{P}_{\tiny{\text{cv}},(X,\varepsilon)}$, and $\td{\sigma}^2$ ranges over all estimators of $\sigma^2$.
\end{thm}

\begin{remark}
\label{remark:uniform_design}
It remains an open problem to prove a lower bound rate that is strictly slower than $n^{-1}$ over the sub-class of $\cal{P}_{\tiny{\text{cv}},(X,\varepsilon)}$ with more regular designs, which includes in particular the uniform design on $[0,1]$. We conjecture that in this case, $n^{-8\alpha/(4\alpha+1)}\vee n^{-1}$ is still the minimax rate in view of analogous results in quadratic functional estimation \citep{bickel1988estimating, fan1991estimation}.
\end{remark}

\section{Heteroscedastic case}
\label{sec:v_function}

We now study the heteroscedastic model \eqref{model:var_function},
\begin{align*}
Y_i = f(X_i)  + V^{1/2}(X_i)\varepsilon_i ,\quad i=1,2,\ldots,n,
\end{align*}
where $\{X_i\}_{i=1}^n$ are i.i.d. copies of $X$ on the real line, $f(\cdot)$ and $V(\cdot)$ are $\alpha$- and $\beta$-H\"older smooth on the fixed (possibly infinite) interval $I$, respectively, and $\{\varepsilon_i\}_{i=1}^n$ are i.i.d. copies of $\varepsilon$ with zero mean and unit variance and are independent of $\{X_i\}_{i=1}^n$. As in Section \ref{sec:nonpar}, smoothness indices $\alpha$ and $\beta$ are assumed known, while $f(\cdot),V(\cdot)$, and the distribution of $X$ are unknown. 
For any estimator $\td{V}(\cdot)$, the estimation accuracy is measured both locally via 
\begin{align}
\label{eq:point_loss}
R_1(\td{V}, V;x^*) \define \parr*{\td{V}(x^*) - V(x^*)}^2
\end{align}
at a point $x^*$ in the support of $X$, $\supp(X)$, and globally via 
\begin{align}
\label{eq:int_loss}
R_2(\td{V}, V) \define \int \parr*{\td{V}(x) - V(x)}^2\P_X(dx)
\end{align}
with $\P_X$ the distribution of $X$. 

Model \eqref{model:var_function} has been studied in, for example, \cite{muller1987estimation}, \cite{hall1989variance}, \cite{ruppert1997local}, \cite{hardle1997local}, \cite{fan1998efficient}, \cite{munk2002minimax}, \cite{brown2007variance}, \cite{wang2008effect}, with a focus mainly on the fixed design case. An exception is \cite{munk2002minimax}, with which we draw a detailed comparison in Remark \ref{remark:comparison_munk} below. Theorems 1 and 2 in \cite{wang2008effect} established a minimax rate of the order $n^{-4\alpha}\vee n^{-2\beta/(2\beta+1)}$ under equidistance design $X_i = i/n$, $i\in[n]$ when $f(\cdot)$ and $V(\cdot)$ are $\alpha$- and $\beta$-H\"older smooth on [0,1]. 

Define $\cal{P}_{\tiny{\text{vf}},(X,\varepsilon)}$ (where ``vf" stands for ``variance function") as follows:
\begin{itemize}
\item[(a)] $X$ satisfies $\supp(X)\subset I$.
\item[(b)] $X$ has density $p_X(\cdot)$, and there exists a fixed positive constant $C_0$ such that 
\begin{align*}
\sup_{x\in\RR} p_X(x) \leq C_0.
\end{align*}
\item[(c)] There exist fixed positive constants $c_0$ and $\delta_0$ such that
\begin{align*}
&\inf_{x^*\in\supp(X)}p_X(x^*) \geq c_0 \quad \text{ and } \\
&\inf_{0<\delta<\delta_0}\inf_{x^*\in\supp(X)}\lambda\parr*{\bbrace*{u\in[-1,1]: x^* + \delta u\in\supp(X)}} \geq c_0,
\end{align*}
where $\lambda(\cdot)$ is the Lebesgue measure on the real line.
\item[(d)]$\E\varepsilon^4 \leq C_\varepsilon$ for some fixed positive constant $C_\varepsilon$.
\end{itemize}
One can readily verify that $\cal{P}_{\tiny{\text{vf}},(X,\varepsilon)}\subset \cal{P}_{\tiny{\text{cv}},(X,\varepsilon)}$, with the latter defined in the beginning of Section \ref{sec:nonpar}. Compared to $\cal{P}_{\tiny{\text{cv}},(X,\varepsilon)}$, Condition (c) in $\cal{P}_{\tiny{\text{vf}},(X,\varepsilon)}$ is posed on the marginal density and support of $X$, 
since in the variance function case we require a sufficient number of close pairs $(X_i, X_j)$ around each target $x^*$. We also note that, as in $\cal{P}_{\tiny{\text{cv}},(X,\varepsilon)}$, no smoothness assumption is posed on the design density in $\cal{P}_{\tiny{\text{vf}},(X,\varepsilon)}$. 

The rest of the section is devoted to proving, for any fixed positive constants $C_\cal{F}$ and $C_\cal{V}$, the following minimax rates
\begin{equation}
\begin{aligned}
\label{eq:intro_m1}
\!\inf_{\td{V}}\!\!\sup_{f\in\Lambda_{\alpha,I}(C_\cal{F})}\!\sup_{V\in\Lambda_{\beta,I}(C_\cal{V})}\!\sup_{\P_{(X,\varepsilon)}\in\cal{P}_{\tiny{\text{vf}},(X,\varepsilon)}}\sup_{x^*\in\supp(X)}\!\!\E R_1(\td{V},V;x^*) &\!\asymp\! n^{-\frac{8\alpha\beta}{4\alpha\beta + 2\alpha + \beta}}\!\vee \!n^{-\frac{2\beta}{2\beta+1}},\\
\inf_{\td{V}}\sup_{f\in\Lambda_{\alpha,I}(C_\cal{F})}\sup_{V\in\Lambda_{\beta,I}(C_\cal{V})}\sup_{\P_{(X,\varepsilon)}\in\cal{P}_{\tiny{\text{vf}},(X,\varepsilon)}}\E R_2(\td{V},V) &\!\asymp\! n^{-\frac{8\alpha\beta}{4\alpha\beta + 2\alpha + \beta}}\!\vee\! n^{-\frac{2\beta}{2\beta+1}},
\end{aligned}
\end{equation}
where $\P_{(X,\varepsilon)}$ denotes the joint distribution of $(X,\varepsilon)$, and $\td{V}(\cdot)$ ranges over all estimators of $V(\cdot)$.

\subsection{Upper bound}
\label{subsec:v_upper}
We now propose an estimator of $V(x^*)$ for some fixed $x^*\in\supp(X)$ by combining pairwise differences with local polynomial regression. We first introduce some notation. Let $\ell$ be the largest integer strictly smaller than $\beta$ and 
\[
\mbf{q}(u) \define (1,u,u^2/2!,\ldots, u^\ell/\ell!)^\top.
\] 
For any $1 \leq i < j\leq n$, define 
\[
D_{ij}\define (Y_i-Y_j)^2/2, ~X_{ij}\define (X_i +X_j)/2,~ {\rm and}~ K_{ij}\define K_{h_1}(X_i - X_j)K_{h_2}(X_{ij} - x^*), 
\]
where $h_1,h_2$ are two bandwidths. Define an $(\ell+1)\times(\ell+1)$ matrix
\begin{align*}
\mathbf{B}_n \define {n\choose 2}^{-1}\sum_{i<j} \mbf{q}\parr*{\frac{X_{ij} - x^*}{h_2}}\mbf{q}^\top\parr*{\frac{X_{ij} - x^*}{h_2}}K_{ij}
\end{align*}
and $\mathbf{B}_n^*$ as its adjugate such that $\mathbf{B}_n\mathbf{B}_n^* = \mathbf{B}_n^*\mathbf{B}_n = |\mathbf{B}_n|\mathbf{I}_{\ell+1}$. For example, when $\ell = 1$, we have
\begin{align*}
\mathbf{B}_n = 
\begin{bmatrix}
s_0 & s_1\\
s_1 & s_2
\end{bmatrix}
,\quad
\mathbf{B}_n^* =
\begin{bmatrix}
s_2 & -s_1\\
-s_1 & s_0
\end{bmatrix}, \quad \text{ and }\quad |\mathbf{B}_n| = s_0s_2 - s_1^2,
\end{align*}
where
\begin{align*}
s_k \define {n\choose 2}^{-1}\sum_{i<j} \parr*{\frac{X_{ij} - x^*}{h_2}}^kK_{ij}, \qquad k=0,1,2.
\end{align*}


Following \cite{fan1993local}, we propose a robust local polynomial estimator:
\begin{align}
\label{eq:v_estimator}
\widehat{V}_{\tiny{\LP}}(x^*) \define {n\choose 2}^{-1}\sum_{i<j} D_{ij}(|\mathbf{B}_n| + \tau_n)^{-1}\mbf{q}^\top(0)\mathbf{B}_n^*\mbf{q}\parr*{\frac{X_{ij} - x^*}{h_2}}K_{ij},
\end{align}
where $\tau_n$ is some sufficiently small positive constant that decays to $0$ polynomially with $n$. 
Let 
\[
w_{ij} \define {n\choose 2}^{-1}\mbf{q}^\top(0)\mathbf{B}_n^*\mbf{q}\parr*{\frac{X_{ij} - x^*}{h_2}}K_{ij} \quad {\rm and}\quad \td{w}_{ij}\define w_{ij}/(|\mathbf{B}_n| + \tau_n). 
\]
Then, it holds that $\widehat{V}_{\tiny{\LP}}(x^*) = \sum_{i<j} \td{w}_{ij}D_{ij}$, $\sum_{i<j}w_{ij} = |\mathbf{B}_n|$, and
\begin{align}
\label{eq:rep_lp}
\sum_{i<j}w_{ij}(X_{ij} - x^*)^k = \sum_{i<j}\td{w}_{ij}(X_{ij} - x^*)^k = 0, \qquad k=1,2,\ldots,\ell.
\end{align}
The last property \eqref{eq:rep_lp} is referred to as the \emph{reproducing property} of local polynomial estimators (cf. Proposition 1.12 in \cite{tsybakov2009introduction}).

\begin{thm}
\label{thm:v_function_upper}
Suppose the kernel $K(\cdot)$ in $\widehat{V}_{\tiny{\LP}}$ is chosen such that \eqref{eq:kernel} holds with constants $\overline{M}_K$ and $\under{M}_K$, $\tau_n \asymp n^{-\kappa}$ for some fixed constant $
\kappa \geq 1$, and the bandwidths $h_1,h_2$ are chosen as
\begin{align}
\label{eq:v_h_choice}
(h_1,h_2) \asymp 
\begin{cases}
\parr*{n^{-\frac{2\beta}{4\alpha\beta + \beta + 2\alpha}}, n^{-\frac{4\alpha}{4\alpha\beta + \beta + 2\alpha}}}, & 0< \alpha < \frac{\beta}{4\beta +2},\\
\parr*{n^{-1}, n^{-\frac{1}{2\beta + 1}}}, & \alpha \geq \frac{\beta}{4\beta +2}.
\end{cases}
\end{align}
Then, under \eqref{model:var_function} with random design, it holds that
\begin{align*}
\sup_{f\in\Lambda_{\alpha,I}(C_\cal{F})}\!\sup_{V\in\Lambda_{\beta,I}(C_\cal{V})}\!\sup_{\P_{(X,\varepsilon)}\in\cal{P}_{\tiny{\text{vf}},(X,\varepsilon)}}\!\sup_{x^*\in\supp(X)} \!\!\!\E R_1(\widehat{V}_{\tiny{\LP}}, V;x^*) \!\leq \!C \parr*{n^{-\frac{8\alpha\beta}{4\alpha\beta + \beta + 2\alpha}} \!\vee\! n^{-\frac{2\beta}{2\beta +1}}}
\end{align*}
and
\begin{align*}
\sup_{f\in\Lambda_{\alpha,I}(C_\cal{F})}\sup_{V\in\Lambda_{\beta,I}(C_\cal{V})}\sup_{\P_{(X,\varepsilon)}\in\cal{P}_{\tiny{\text{vf}},(X,\varepsilon)}} \E R_2(\widehat{V}_{\tiny{\LP}}, V) \leq C\parr*{n^{-\frac{8\alpha\beta}{4\alpha\beta + \beta + 2\alpha}} \vee n^{-\frac{2\beta}{2\beta +1}}},
\end{align*}
where $C$ is some fixed positive constant that only depends on $\overline{M}_K, \under{M}_K,\alpha,\beta,C_\cal{F},C_\cal{V}$ and $C_0,c_0,C_\varepsilon$ in $\cal{P}_{\tiny{\text{vf}},(X,\varepsilon)}$.
\end{thm}


\begin{remark}
\label{remark:comparison_v}
Variance function estimation in \eqref{model:var_function} with fixed design $X_i = i/n$, $i\in[n]$, has been studied in \cite{wang2008effect}. There the minimax rate is 
\begin{align*}
\inf_{\td{V}}\sup_{f\in\Lambda_{\alpha,[0,1]}(C_\cal{F})}\sup_{V\in\Lambda_{\beta,[0,1]}(C_\cal{V})}\sup_{\E\varepsilon^4 \leq C_\varepsilon}\sup_{x^*\in[0,1]} \E R_1(\td{V}, V;x^*) \asymp n^{-4\alpha}\vee n^{-2\beta/(2\beta + 1)},\\
\inf_{\td{V}}\sup_{f\in\Lambda_{\alpha,[0,1]}(C_\cal{F})}\sup_{V\in\Lambda_{\beta,[0,1]}(C_\cal{V})}\sup_{\E\varepsilon^4 \leq C_\varepsilon} \E R_2(\td{V}, V) \asymp n^{-4\alpha}\vee n^{-2\beta/(2\beta + 1)},
\end{align*}
with the integral in $R_2$ under the Lebesgue measure on $[0,1]$. Comparing the above result with the error rate in Theorem \ref{thm:v_function_upper}, we see that the transition boundary in both the fixed and random design settings is $\alpha = \beta/(4\beta+2)$. When $\alpha \geq \beta/(4\beta + 2)$, $V(\cdot)$ under both $R_1$ and $R_2$ can be estimated at the classic nonparametric rate $n^{-2\beta/(2\beta + 1)}$ as if the mean function $f(\cdot)$ were known. When $\alpha < \beta/(4\beta + 2)$, a faster rate can be achieved in the random design case. This can be intuitively understood by the fact that, by constrast to the fixed design case, a significant portion of pairs have distance smaller than $1/n$ in the random design setting.
\end{remark}

\begin{remark}
\label{remark:compare_preest}
As has been noted in \cite{wang2008effect}, in the fixed design setting, estimating the variance (function) by smoothing the squared residuals obtained from pre-estimation of the mean function $f(\cdot)$ is sub-optimal. The same conclusion also applies to the random design setting. Since the design being fixed or random has no first-order effect on the estimation of the mean, the above method only achieves the rates $n^{-4\alpha/(2\alpha+1)}\vee n^{-1}$ in variance estimation and $n^{-4\alpha/(2\alpha+1)}\vee n^{-2\beta/(2\beta+1)}$ in variance function estimation, neither of which is minimax optimal. 
\end{remark}

\begin{remark}
Unlike in the fixed design case, once below the threshold $\alpha = \beta/(4\beta + 2)$, $\alpha$ and $\beta$ are now both present in the minimax rate in the random design case, suggesting that the smoothness of $V(\cdot)$ always has an effect on its estimation. This is because variance function estimation in the random design setting is essentially a ``two-dimensional" problem, where we have to jointly choose two optimal neighborhood sizes to characterize the closeness between (i) each $X_i$ and $X_j$;  and (ii) every pair $(X_i,X_j)$ and each target point $x^*$. By contrast, in the fixed design setting, the distance between $X_i$ and $X_j$ is constrained to be no smaller than $1/n$, and thus cannot be jointly optimized with the distance between $(X_i,X_j)$ and $x^*$.
\end{remark}

\begin{remark}
One might wonder whether the following Nadaraya-Watson type estimator can be used to establish the upper bound in Theorem \ref{thm:v_function_upper}:
\begin{align}
\label{eq:v_nw}
\widehat{V}_{\tiny{\text{NW}}}(x^*) \define \frac{\sum_{i<j} K_{h_1}(X_i - X_j)K_{h_2}(X_{ij} - x^*)D_{ij}}{\sum_{i<j} K_{h_1}(X_i - X_j)K_{h_2}(X_{ij} - x^*)},
\end{align}
where $K(\cdot)$ is now chosen to be a higher-order kernel to further reduce bias when $\beta > 1$. It turns out that the analysis of $\widehat{V}_{\tiny{\text{NW}}}$ requires an extra assumption on the smoothness of the density $p_X(\cdot)$ which can be completely avoided with $\widehat{V}_{\tiny{\LP}}$. Moreover, it is well-known that local polynomial estimators have good finite sample properties and boundary performances when $X$ is compactly supported \citep{fan2018local}.
\end{remark}

\begin{remark}
\label{remark:comparison_munk}
\cite{munk2002minimax} considered minimax estimation of the variance function (and more generally, its derivatives) in the context of nonparametric regression with random design. We focus on the comparison of their results on variance function estimation with ours. Their lower bound (Theorem 1 therein) is proved independent of the smoothness level of the mean function and upper bound (Theorem 4 therein) is proved under sufficient smoothness on the mean function. Therefore their minimax rate is only comparable to the $n^{-2\beta/(2\beta + 1)}$ component in ours.  In this case, their lower bound of the order $n^{-(2\beta - 1)/(2\beta)}$ is proved over the following class of variance function:
\begin{align*}
\cal{S}_\beta \define \bbrace*{1 + \sum_{k=1}^\infty \delta_ke_k: |\delta_k|\lesssim k^{-\beta}}
\end{align*}
for any $\beta > 1$, where $\{e_k\}_{k=1}^\infty$ is an arbitrary basis on $L^2([-\pi,\pi])$. Moreover, continuous differentiability of the error density is required in their paper. In contrast, we pose no smoothness conditions on the error density, and neither $\cal{S}_\beta$ nor $\cal{S}_{\beta+1/2}$ can be embedded in the $\beta$-H\"older class $\Lambda_{\beta}$ considered in our setting (e.g., $f(x)=|x|$ with domain $[-\pi,\pi]$ belongs to $\cal{S}_{2}$ but is not $1.5$- or 2-H\"older smooth since it is not differentiable at the origin). In summary, the results in \cite{munk2002minimax} neither imply nor contradict the $n^{-2\beta/(2\beta+1)}$ part in our minimax rate, and our results are more refined since they characterize the exact elbow $\alpha = \beta/(4\beta+2)$ and also the minimax rate below this threshold.
\end{remark}



\subsection{Lower bound}\label{sec:lb3}

The following are matching lower bounds to Theorem \ref{thm:v_function_upper}.

\begin{thm}
\label{thm:v_lower_point}
Under \eqref{model:var_function} with random design, for any $x^*\in\supp(X)$, 
\begin{align*}
\inf_{\td{V}}\sup_{f\in\Lambda_{\alpha,I}(C_\cal{F})}\sup_{V\in\Lambda_{\beta,I}(C_\cal{V})}\sup_{\P_{(X,\varepsilon)}\in\cal{P}_{\tiny{\text{vf}},(X,\varepsilon)}}\E R_1(\td{V}, V;x^*) \geq c \parr*{n^{-\frac{8\alpha\beta}{4\alpha\beta + \beta + 2\alpha}} \vee n^{-\frac{2\beta}{2\beta +1}}},
\end{align*}
where $c$ is some fixed positive constant that only depends on $\alpha,\beta, C_\cal{F}, C_\cal{V}$ and $C_0,c_0,C_\varepsilon$ in $\cal{P}_{\tiny{\text{vf}},(X,\varepsilon)}$, and $\td{V}$ ranges over all estimators of $V$.
\end{thm}

\begin{thm}
\label{thm:v_lower_int}
Under \eqref{model:var_function} with random design, 
\begin{align*}
\inf_{\td{V}}\sup_{f\in\Lambda_{\alpha,I}(C_\cal{F})}\sup_{V\in\Lambda_{\beta,I}(C_\cal{V})}\sup_{\P_{(X,\varepsilon)}\in\cal{P}_{\tiny{\text{vf}},(X,\varepsilon)}} \E R_2(\td{V}, V) \geq c\parr*{n^{-\frac{8\alpha\beta}{4\alpha\beta + \beta + 2\alpha}} \vee n^{-\frac{2\beta}{2\beta +1}}},
\end{align*}
where $c$ is some fixed positive constant that only depends on $\alpha,\beta, C_\cal{F}, C_\cal{V}$ and $C_0,c_0,C_\varepsilon$ in $\cal{P}_{\tiny{\text{vf}},(X,\varepsilon)}$, and $\td{V}$ ranges over all estimators of $V$.
\end{thm}


\begin{figure}[t!]
\centering
\includegraphics[width = 0.7\textwidth]{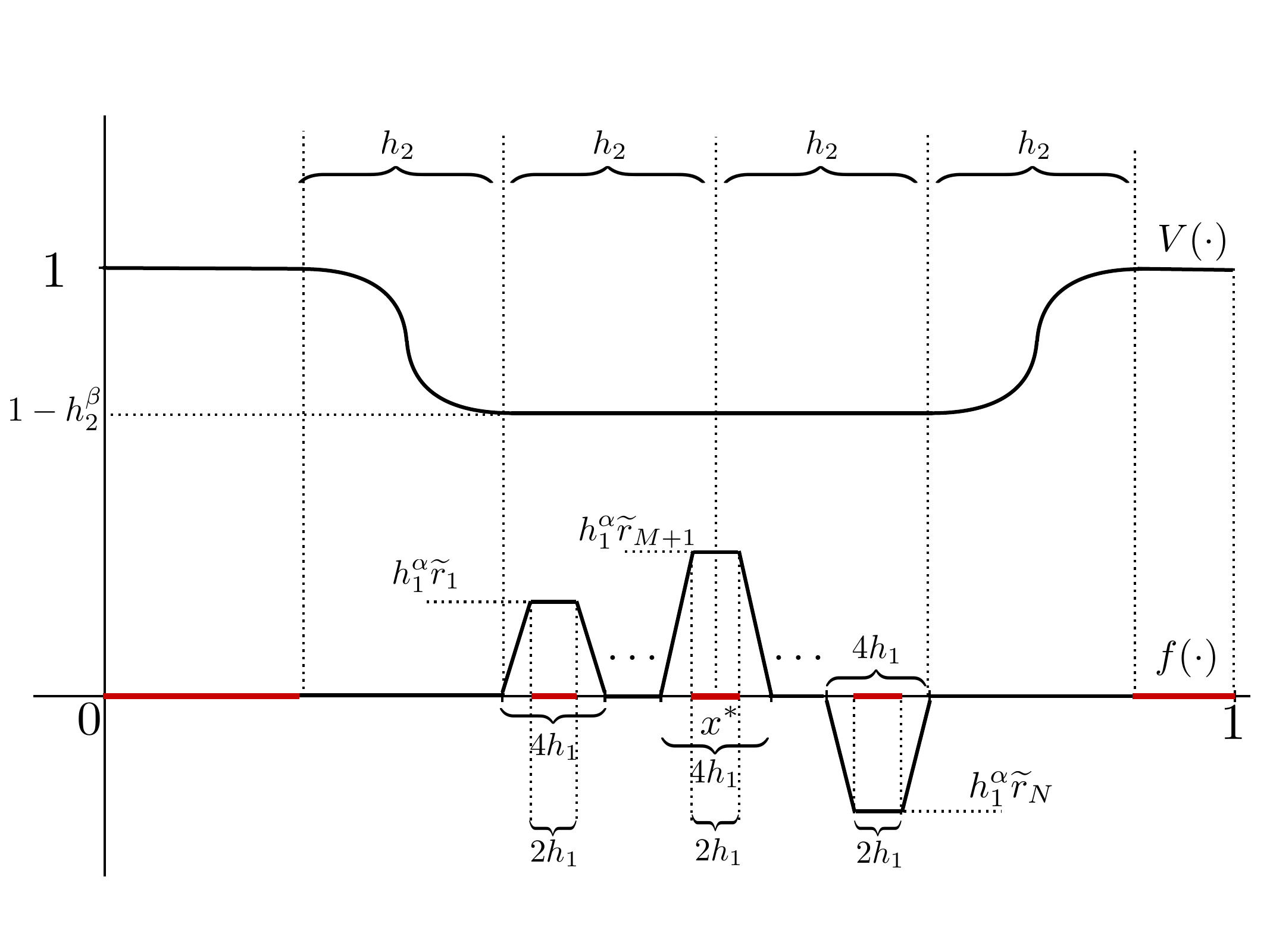}
\caption{The black solid line on the top represents the variance function $V(\cdot)$ in the alternative $\td{H}_1$, and the black solid line on the bottom represents the mean function $f(\cdot)$. The thick red segments mark the support of $X$ under both $H_0$ and $\td{H}_1$. Here, $h_1 \asymp n^{-\frac{2\beta}{4\alpha\beta + \beta + 2\alpha}}$, $h_2 \asymp n^{-\frac{4\alpha}{4\alpha\beta + \beta + 2\alpha}}$, and are chosen such that both $M\define h_2/(4h_1) - 1/2$ and $N \define 2M + 1$ are positive integers. $\{\td{r}_i\}_{i=1}^N$ are $N$ i.i.d. standard normal variables.}
\label{fig:v_point_lower}
\end{figure}

Due to the appearances of both $\alpha$ and $\beta$ in the nontrivial $n^{-\frac{8\alpha\beta}{4\alpha\beta + \beta + 2\alpha}}$ part of the minimax rate, proving the above two results is more involved than proving Theorem \ref{thm:nonpar_lower}. In particular, it takes an extra step of localization in the construction of the mean function $f(\cdot)$ as well as $V(\cdot)$. More precisely, for the lower bound at a target point $x^*$ in Theorem \ref{thm:v_lower_point}, our construction of both $f(\cdot)$ and $V(\cdot)$ only has variation within a small neighborhood of $x^*$. Such localized construction is not necessary in the fixed design setting, since when proving the $n^{-4\alpha}$ component therein (see Remark \ref{remark:comparison_v}), the variance function can simply be taken as a constant. 

In what follows, we give a proof sketch of the nontrivial $n^{-8\alpha\beta/(4\alpha\beta + \beta + 2\alpha)}$ component of the lower bound in Theorem \ref{thm:v_lower_point} for $\alpha < \beta/(4\beta + 2)$; the proof of Theorem \ref{thm:v_lower_int} can be seen as an extension of Theorem \ref{thm:v_lower_point} via a standard construction of multiple hypotheses. We assume the support of $X$ is contained in $I = [0,1]$, and for clarity of illustration, here we present the construction for an interior point $x^*\in (0,1)\bigcap\supp(X)$. The proof works for boundary points as well.

We continue to adopt the two-step approach introduced in the proof sketch of Theorem \ref{thm:nonpar_lower} in Section \ref{subsec:nonpar_lower}. The second step is very similar with the help of Lemmas \ref{lemma:moment_matching} and \ref{lemma:multinomial}, so we will focus on the construction under the null $H_0$ and alternative $\td{H}_1$ in the first step. Choose the parameters
\begin{align*}
h_1 \asymp n^{-\frac{2\beta}{4\alpha\beta + \beta + 2\alpha}}, \quad h_2 \asymp n^{-\frac{4\alpha}{4\alpha\beta + \beta + 2\alpha}}, \quad \text{ and } \quad \theta_n^2 =  h_1^{2\alpha} = h_2^\beta
\end{align*}
so that $h_2/h_1\rightarrow\infty$ as $n\rightarrow\infty$.
\begin{itemize}
\item[]\emph{Choice of $V(\cdot)$}: Under $H_0$ let $V\equiv 1$. Under $\td{H}_1$, let $V(\cdot)$ be one minus a smooth bump function around $x^*$ with width $h_2$ and height $h_2^\beta$ so that $V(x^*) = 1-\theta_n^2$.
\item[]\emph{Choice of $f(\cdot)$}: Under $H_0$ let $f\equiv 0$. Under $\td{H}_1$, let $f(\cdot)$ be a ``local" version of the design in Theorem \ref{thm:nonpar_lower}. That is, $f$ takes on a value of 0 outside of $[x^* - h_2, x^* + h_2]$, and inside that $h_2$-neighborhood of $x^*$, $f$ is piecewise trapezoidal with upper base length $2h_1$, lower base length $4h_1$ and height $\{h_1^\alpha \td{r}_i\}_{i=1}^N$ for a standard normal sequence $\{\td{r}_i\}_{i=1}^N$ with $N\define h_2/(2h_1)$ a positive integer. 
\item[]\emph{Choice of $\varepsilon$}: Under both $H_0$ and $\td{H}_1$, let $\varepsilon \sim \cal{N}(0,1)$. 
\item[]\emph{Choice of $X$}: Under both $H_0$ and $\td{H}_1$, let $X$ be uniformly distributed on the union of $[0,1]\backslash [x^*-h_2,x^* + h_2]$ and the upper bases of all the trapezoids inside $[x^* - h_2, x^* + h_2]$. 

\end{itemize}
See Figure \ref{fig:v_point_lower} for an illustration of $\td{H}_1$. 

Under the above construction, the squared distance between the null and alternative hypotheses $(1 - (1 - \theta_n^2))^2 = \theta_n^4 \asymp n^{-\frac{8\alpha\beta}{4\alpha\beta + \beta + 2\alpha}}$ is the desired minimax rate. Using Lemma \ref{lemma:gaussian_tv}, we can show that
\begin{align*}
\TV(\P_0, \td{\P}_1) \lesssim \theta_n^2nh_1^{1/2}h_2^{1/2} \leq c
\end{align*}
for some sufficiently small $c$, where $\P_0$ and $\td{\P}_1$ represent the joint distribution of $\{(X_i, Y_i)\}_{i=1}^n$ under $H_0$ and $\td{H}_1$, respectively. The detailed proof is presented in the supplement.

\section{Discussion}
\label{sec:discussion}
The two univariate models \eqref{model:var_function} and \eqref{model:nonpar} discussed in the previous two sections raise natural questions about possible extensions to the multivariate setting. In what follows, we first present some partial results in this direction in the sense of \eqref{model:multivariate_nonpar} and \eqref{model:additive}. We then establish some connections between our study and quadratic functional estimation and variance estimation in the linear model. Lastly, we discuss two more extensions of \eqref{model:nonpar} in the direction of adaptive estimation and mean function with inhomogeneous smoothness. Throughout, consider $C_\cal{F}, C_\sigma, C_0,c_0,C_\varepsilon$ to be fixed positive constants.

\subsection{Multivariate nonparametric regression}
\label{subsec:multivariate_nonpar}
Consider the following multivariate version of \eqref{model:nonpar}:
\begin{align*}
Y_i = f(\mbf{X}_i) + \sigma\varepsilon_i, \quad i=1,2,\ldots, n,
\end{align*}
where $\{\mbf{X}_i\}_{i=1}^n = \{(X_{i,1},\ldots,X_{i,d})^\top\}_{i=1}^n$ are i.i.d. copies of $\mbf{X} = (X_1,\ldots,X_d)^\top$ in $\RR^d$ for some fixed positive integer $d$, $\{\varepsilon_i\}_{i=1}^n$ are i.i.d. copies of $\varepsilon$ with zero mean and unit variance and are independent of $\{\mbf{X}_i\}_{i=1}^n$, and $f:\RR^d\rightarrow\RR$ belongs to a $d$-dimensional anisotropic H\"older class with smoothness index $\mbf{\alpha} = (\alpha_1,\ldots,\alpha_d)^\top$ defined below. The goal is to estimate $\sigma^2$ with $f(\cdot)$ and the distribution of $\mbf{X}$ as nuisance parameters. This problem has been studied in \cite{spokoiny2002variance}, \cite{munk2005difference}, \cite{cai2009variance}, to name a few, again with a focus on the fixed design setting. 

Let $I_1,\ldots,I_d$ be $d$ fixed (possibly infinite) intervals on $\RR$ and let $\mbf{I}$ be their Cartesian product $I_1\times\ldots\times I_d\subset \RR^d$. Following \cite{barron1999risk} and \cite{bhattacharya2014anisotropic},  we define an anisotropic H\"older class $\Lambda_{\mbf{\alpha},\mbf{I}}(C_\cal{F})$ on $\mbf{I}$ as follows. For any $\mbf{x}\in\mbf{I}$ and $k\in[d]$, let $f_k(\cdot\mid \mbf{x}_{-k})$ denote the univariate function $y\mapsto f(x_1,\ldots,x_{k-1},y,x_{k+1},\ldots,x_d)$, with $\mbf{x}_{-k}$ defined as $\mbf{x}$ without the $k$th component. Then, $\Lambda_{\mbf{\alpha},\mbf{I}}(C_\cal{F})$ is defined as all $f:\mbf{I}\mapsto \RR$ such that
\begin{align*}
\max_{1\leq k\leq d}\max_{0\leq j\leq \floor{\alpha_k}}\sup_{\mbf{x}\in\mbf{I}} \nm*{f_k^{(j)}(\cdot\mid \mbf{x}_{-k})}_{\infty} \leq C_\cal{F} 
\end{align*}
and
\begin{align*}
\max_{1\leq k\leq d}\sup_{x\in\mbf{I}}\sup_{y_1,y_2\in I_k}\frac{\abs*{f_k^{(\floor{\alpha_k})}(y_1\mid \mbf{x}_{-k}) - f_k^{(\floor{\alpha_k})}(y_2\mid \mbf{x}_{-k})}}{|y_1 - y_2|^{\alpha_k^\prime}} \leq C_\cal{F},
\end{align*}
where again $\floor{\alpha_k}$ is the largest integer strictly smaller than $\alpha_k$ and $\alpha_k^\prime \define \alpha_k - \floor{\alpha_k}$. Let $\supp(\mbf{X})$ be the support of $\mbf{X}$.

Define $\cal{P}_{\text{\tiny{mcv}},(\mbf{X},\varepsilon)}$ (where ``mcv" stands for ``multivariate constant variance") as the multivariate counterpart of $\cal{P}_{\tiny{\text{cv}},(X,\varepsilon)}$:
\begin{itemize}
\item[(a)] $\mbf{X}$ satisfies $\supp(\mbf{X})\subset \mbf{I}$.
\item[(b)] $\mbf{X}$ has density $p_{\mbf{X}}(\cdot)$ and there exists a fixed positive constant $C_0$ such that  
\begin{align*}
\sup_{\mbf{u}\in\RR^d} p_{\mbf{X}}(\mbf{u}) \leq C_0 .
\end{align*}
\item[(c)]There exist two fixed constants $\delta_0 > 0$ and $c_0>0$ such that for any $\mbf{\delta}\in\RR^{d}$ that satisfies $\|\mbf{\delta}\|_\infty < \delta_0$, there exists a set $\cal{U}\define\cal{U}_{\mbf{\delta}}\subset[-1,1]^d$ such that
\begin{align*}
&\mbf{\lambda}(\cal{U}_{\mbf{\delta}})\geq c_0 \quad \text{ and } \quad \inf_{\mbf{u}\in\cal{U}_{\mbf{\delta}}}p_{\bXX}(u_1\delta_1,\ldots,u_d\delta_d)\geq c_0,
\end{align*}
where $\mbf{\lambda}(\cdot)$ represents the Lebesgue measure on $\RR^d$. 
\item[(d)]$\E\varepsilon^4 \leq C_\varepsilon$ for some fixed positive constant $C_\varepsilon$.
\end{itemize}

For an upper bound on the minimax risk, we propose the following multivariate extension of \eqref{eq:nonpar_estimator} via a product kernel (again with convention $0/0 = 0$):
\begin{align}
\label{eq:est_d}
\widehat{\sigma}_d^2 \define \frac{{n\choose 2}^{-1}\sum_{i<j} \parr*{\prod_{k=1}^d K_{h_k}(X_{i,k} - X_{j,k})}(Y_i - Y_j)^2/2}{{n\choose 2}^{-1}\sum_{i<j} \parr*{\prod_{k=1}^d K_{h_k}(X_{i,k}-X_{j,k})}},
\end{align}
where $K(\cdot)$ is a kernel chosen to satisfy \eqref{eq:kernel}, and $\{h_k\}_{k=1}^d$ is a kernel bandwidth sequence.


%


%

In the following results, we will use $\under{\alpha}$ to denote the harmonic mean of the $d$-dimensional smoothness index $\mbf{\alpha}$, i.e. $\under{\alpha}\define d/(\sum_{k=1}^d 1/\alpha_k)$. This quantity is known as the \emph{effective smoothness} in classical problems such as anisotropic density estimation \citep{ibragimov1981more, birge1986estimating} and anisotropic function estimation \citep{nussbaum1986on,hoffman2002random}.  

\begin{proposition}
\label{prop:multivariate_upper}
Suppose $0< \alpha_k \leq 1$, $k\in[d]$. Suppose the kernel $K(\cdot)$ in $\widehat{\sigma}_d^2$ is chosen such that \eqref{eq:kernel} is satisfied with constants $\overline{M}_K$ and $\under{M}_K$, and the bandwidth sequence is chosen as $h_k \asymp n^{-2\under{\alpha}/(\alpha_k(4\under{\alpha}+d))}$ for all $k\in[d]$. Then, under \eqref{model:multivariate_nonpar} with random design, it holds that
\begin{align*}
\sup_{f\in\Lambda_{\mbf{\alpha},\mbf{I}}(C_\cal{F})}\sup_{\sigma^2\leq C_\sigma}\sup_{\P_{(\mbf{X},\varepsilon)}\in\cal{P}_{\text{\tiny{mcv}},(\mbf{X},\varepsilon)}}\E\parr*{\widehat{\sigma}_d^2 - \sigma^2}^2 \leq C\parr*{n^{-8\under{\alpha}/(4\under{\alpha} + d)}\vee n^{-1}},
\end{align*} 
where $C$ is some fixed positive constant that only depends on $\overline{M}_K, \under{M}_K,\mbf{\alpha},C_\cal{F},C_\sigma$ and $C_0,c_0,C_\varepsilon$ in $\cal{P}_{\text{\tiny{mcv}},(\mbf{X},\varepsilon)}$. 
\end{proposition}

\begin{proposition}
\label{prop:multivariate_lower}
Under \eqref{model:multivariate_nonpar} with random design, it holds that
\begin{align*}
\inf_{\td{\sigma}^2}\sup_{f\in\Lambda_{\mbf{\alpha},\mbf{I}}(C_\cal{F})}\sup_{\sigma^2\leq C_\sigma}\sup_{\P_{(\mbf{X},\varepsilon)}\in\cal{P}_{\text{\tiny{mcv}},(\mbf{X},\varepsilon)}}\E\parr*{\td{\sigma}^2 - \sigma^2}^2 \geq c\parr*{n^{-8\under{\alpha}/(4\under{\alpha} + d)}\vee n^{-1}},
\end{align*} 
where $c$ is some fixed positive constant that only depends on $\mbf{\alpha},C_\cal{F}, C_\sigma$ and $C_0,c_0,C_\varepsilon$ in $\cal{P}_{\text{\tiny{mcv}},(\mbf{X},\varepsilon)}$, and $\td{\sigma}^2$ ranges over all estimators of $\sigma^2$. 
\end{proposition}

We note that Proposition \ref{prop:multivariate_upper} is only proved for $\alpha_k\in(0,1]$, $k\in[d]$. The general case when $\alpha_k$ is possibly larger than 1 is much more involved due to the difficulty in the random design analysis. Propositions \ref{prop:multivariate_upper} and \ref{prop:multivariate_lower}, combined, imply that the minimax rate is $n^{-8\under{\alpha}/(4\under{\alpha} + d)}\vee n^{-1}$ for $\alpha_k\in(0,1]$, $k\in[d]$. In particular, when $f$ is in an isotropic $\alpha$-H\"older class ($0<\alpha\leq 1$), this rate becomes $n^{-8\alpha/(4\alpha + d)}\vee n^{-1}$. We also remark that a different estimator achieving the rate $n^{-8\alpha/(4\alpha + d)}\vee n^{-1}$ over an isotropic $\alpha$-H\"older class has been briefly sketched in \cite{robins2008higher}.

For completeness, we also state without proof some results for model \eqref{model:multivariate_nonpar} in the fixed design setting. In particular, we consider the following two types of fixed designs in the $d$-dimensional unit cube $[0,1]^d$, namely, the grid design (GD):
\begin{align}
\label{eq:GD}
\quad(X_{(i_1,\ldots,i_d),1},\ldots, X_{(i_1,\ldots,i_d),d}) = (i_1/n^{1/d}, \ldots, i_d/n^{1/d}), \\ 
(i_1,\ldots,i_d) \in [n^{1/d}]\times \ldots\times[n^{1/d}] \notag
\end{align}
assuming $n^{1/d}$ is an integer, and the diagonal design (DD):
\begin{align}
\label{eq:DD}
(X_{i,1},\ldots,X_{i,d}) = (i/n, \ldots, i/n),\quad i\in[n].
\end{align}
Here for any positive integer $n$, $[n]$ denotes the set $\{1,2,\ldots,n\}$. Let $\alpha_{\max} \define \max_{k\in[d]}\alpha_k$ and $\alpha_{\min} \define \min_{k\in[d]}\alpha_k$. The first result for (GD) is a simple modification of the isotropic result in \cite{cai2009variance} by taking differences along the smoothest direction with index $\alpha_{\max}$. The second result can be readily deduced from the fact that $Y_i = \td{f}(i/n) + \sigma\varepsilon_i$, $i\in[n]$, where $\td{f}(x)\define f(x,\ldots,x)$ is $\alpha_{\min}$-H\"older smooth.

\begin{proposition}
\label{prop:multivariate_fixed_GD}
Under \eqref{model:multivariate_nonpar} with fixed design (GD), it holds that
\begin{align*}
\inf_{\td{\sigma}^2}\sup_{f\in\Lambda_{\mbf{\alpha},[0,1]^d}(C_\cal{F})}\sup_{\sigma^2\leq C_\sigma}\sup_{\E\varepsilon^4\leq C_\varepsilon} \E\parr*{\td{\sigma}^2 - \sigma^2}^2 \asymp n^{-4\alpha_{\max}/d}\vee n^{-1}
\end{align*}
up to some fixed positive constant that only depends on $\mbf{\alpha}, C_\cal{F},C_\sigma,C_\varepsilon$, where $\td{\sigma}^2$  ranges over all estimators of $\sigma^2$.
\end{proposition}

\begin{proposition}
\label{prop:multivariate_fixed_DD}
Under \eqref{model:multivariate_nonpar} with fixed design (DD), it holds that
\begin{align*}
\inf_{\td{\sigma}^2}\sup_{f\in\Lambda_{\mbf{\alpha},[0,1]^d}(C_\cal{F})}\sup_{\sigma^2\leq C_\sigma}\sup_{\E\varepsilon^4\leq C_\varepsilon} \E\parr*{\td{\sigma}^2 - \sigma^2}^2 \asymp n^{-4\alpha_{\min}}\vee n^{-1}
\end{align*}
up to some fixed positive constant that only depends on $\mbf{\alpha}, C_\cal{F},C_\sigma,C_\varepsilon$, where $\td{\sigma}^2$ ranges over all estimators of $\sigma^2$.
\end{proposition}

When $f(\cdot)$ belongs to an isotropic $\alpha$-H\"older class, Proposition \ref{prop:multivariate_fixed_GD} implies the minimax rate $n^{-4\alpha/d} \vee n^{-1}$ derived in \cite{cai2009variance}. Comparison with the random design rate $n^{-8\alpha/(4\alpha + d)}\vee n^{-1}$ thus shows that, for $0 <\alpha \leq 1$, a faster rate is again achievable in the random design setting for $\alpha < d/4$. 

\subsection{Nonparametric additive model}
\label{subsec:additive}

Consider variance estimation in the additive model \eqref{model:additive}:
\begin{align*}
Y_i = \sum_{k=1}^d f_k(X_{i,k}) + \sigma\varepsilon_i, \quad i=1,2,\ldots,n,
\end{align*}
for some fixed integer $d \geq 2$, where $\{\varepsilon_i\}_{i=1}^n$ are i.i.d. with zero mean and unit variance and are independent from $\{\mbf{X}_i\}_{i=1}^n = \{(X_{i,1},\ldots,X_{i,d})^\top\}_{i=1}^n$ in the random design setting. Unlike Section \ref{subsec:multivariate_nonpar}, we specify $d\geq 2$, since the minimax rate in the fixed design (GD) has completely different behavior for $d=1$ and $d\geq 2$ (see Proposition \ref{prop:additive_fixed_GD} below).

\subsubsection{Fixed design}
We first consider the two fixed designs (GD) and (DD) defined in \eqref{eq:GD} and \eqref{eq:DD}. For both designs, we consider an error distribution class with only a finite fourth moment condition. We start with (GD), where by iteratively taking pairwise differences, one is able to estimate the variance at the parametric rate $n^{-1}$ without any smoothness assumption on the additive components $\{f_k\}_{k=1}^d$. For simplicity, we illustrate this idea with $d = 2$ with two additive components $f(\cdot)$ and $g(\cdot)$, and assume that $\sqrt{n}$ is an even number. In this case,
\begin{align*}
Y_{i,j} = f\parr*{\frac{i}{\sqrt{n}}} + g\parr*{\frac{j}{\sqrt{n}}} + \sigma\varepsilon_{i,j}, \quad (i,j)\in[\sqrt{n}]\times [\sqrt{n}],
\end{align*}  
where $\{\varepsilon_{i,j}\}_{i,j\in[\sqrt{n}]}$ are i.i.d. with zero mean and unit variance. By taking the pairwise difference in the first dimension, we have
\begin{align*}
Y_{(i_1,i_2), j} \define Y_{i_1,j} - Y_{i_2,j} = f\parr*{\frac{i_1}{\sqrt{n}}} - f\parr*{\frac{i_2}{\sqrt{n}}} + \sigma(\varepsilon_{i_1,j} - \varepsilon_{i_2,j})
\end{align*}
for all $j\in[\sqrt{n}]$ and $(i_1,i_2)\in[\sqrt{n}]\times [\sqrt{n}]$ such that $i_1\neq i_2$. Taking again the pairwise difference in the second dimension, we have
\begin{align*}
Y_{(i_1,i_2),(j_1,j_2)}\define Y_{(i_1,i_2),j_1} - Y_{(i_1,i_2),j_2} = \sigma(\varepsilon_{i_1,j_1} - \varepsilon_{i_2,j_1} - \varepsilon_{i_1,j_2} + \varepsilon_{i_2,j_2})
\end{align*} 
for all $(i_1,i_2,j_1,j_2)\in[\sqrt{n}]\times [\sqrt{n}]\times[\sqrt{n}]\times[\sqrt{n}]$ such that $i_1\neq i_2$ and $j_1\neq j_2$. Clearly, we have $\E Y_{(i_1,i_2),(j_1,j_2)} = 0$ and $\var(Y_{(i_1,i_2),(j_1,j_2)}) = 4\sigma^2$. Let $m \define \sqrt{n}/2$ and define $\cal{I}\define \{(1,2), (3,4), \ldots, (2m-1, 2m)\}$ with cardinality $m$. Then, for the set of data points $\{Y_{(i_1,i_2),(j_1,j_2)}\}_{(i_1,i_2),(j_1,j_2)\in \cal{I}}$ with cardinality $m^2 = n/4$, it can be readily verified that they are i.i.d. with mean $0$ and variance $4\sigma^2$. Therefore, with $\bar{Y}$ defined as the sample average of $\{Y_{(i_1,i_2), (j_1,j_2)}\}_{(i_1,i_2),(j_1,j_2)\in\cal{I}}$, the sample variance estimator,
\begin{align*}
\widehat{\sigma}_{\tiny{\text{add, GD}}}^2 \define \frac{1}{n}\sum_{(i_1,i_2), (j_1,j_2)\in\cal{I}} \parr*{Y_{(i_1,i_2),(j_1,j_2)} - \bar{Y}}^2, 
\end{align*}
achieves the parametric rate $n^{-1}$. A similar derivation holds for general $d$.


\begin{proposition}
\label{prop:additive_fixed_GD}
Suppose $d\geq 2$. Under \eqref{model:additive} with fixed design (GD), it holds that
\begin{align*}
\inf_{\td{\sigma}^2}\sup_{f_k,k\in[d]}\sup_{\sigma^2\leq C_\sigma}\sup_{\E\varepsilon^4\leq C_\varepsilon} \E\parr*{\td{\sigma}^2 - \sigma^2}^2 \asymp n^{-1}
\end{align*}
up to some fixed positive constant that only depends on $C_\sigma$ and $C_\varepsilon$, where $\td{\sigma}^2$  ranges over all estimators of $\sigma^2$, and the first supremum is taken over all functions defined on $[0,1]$ for each $k\in[d]$.
\end{proposition}

Now we move on to the design (DD), where we assume each additive component $f_k$ in \eqref{model:additive} is $\alpha_k$-H\"older smooth on $[0,1]$ with some fixed constant $C_\cal{F}$. In this case, the model can equivalently be written as 
\begin{align*}
Y_i = \td{f}(i/n) + \sigma\varepsilon_i, \quad i=1,2,\ldots,n,
\end{align*}
where $\td{f}\define \sum_{k=1}^d f_k$ is $\alpha_{\min}$-H\"older smooth. Therefore, the univariate estimator and lower bound in \cite{wang2008effect} can be directly applied. 
\begin{proposition}
\label{prop:additive_fixed_DD}
Under \eqref{model:additive} with fixed design (DD), it holds that
\begin{align*}
\inf_{\td{\sigma}^2}\sup_{f_k\in\Lambda_{\alpha_k,[0,1]}(C_\cal{F}),k\in[d]}\sup_{\sigma^2\leq C_\sigma}\sup_{\E\varepsilon^4\leq C_\varepsilon} \E\parr*{\td{\sigma}^2 - \sigma^2}^2 \asymp n^{-4\alpha_{\min}}\vee n^{-1}
\end{align*}
up to some fixed positive constant that only depends on $C_\cal{F},C_\sigma,C_\varepsilon$, where $\td{\sigma}^2$  ranges over all estimators of $\sigma^2$.
\end{proposition}
Comparison of Propositions \ref{prop:additive_fixed_DD} and \ref{prop:multivariate_fixed_DD} shows that, in contrast to grid design (GD) and random design below, there is no gain from an additive structure in the mean function for the diagonal design (DD). 
 
\subsubsection{Random design}
We now discuss \eqref{model:additive} with a random design for $\{\mbf{X}_i\}_{i=1}^n$ when $f_k$ is $\alpha_k$-H\"older smooth on some fixed set $I_k$ for each $k\in[d]$. Since a shift in the mean does not affect the estimation of variance, we assume $\E f_k(X_{1,k}) = 0$ for each $k\in[d]$ for simplicity. 
Recall the definition of $\cal{P}_{\tiny{\text{cv}},(X,\varepsilon)}$ in the beginning of Section \ref{sec:nonpar}. Define the joint distribution class $\cal{P}_{\text{\tiny{add}},(\mbf{X},\varepsilon)}$ (where ``add" stands for ``additive") as:
\begin{itemize}
\item[] For each $k\in[d]$, the joint distribution of $(X_k,\varepsilon)$ belongs to $\cal{P}_{\tiny{\text{cv}},(X,\varepsilon)}$ and the components of $\mbf{X}$ are mutually independent.
\end{itemize}

In view of Theorem \ref{thm:nonpar_lower}, the following lower bound is immediate.
\begin{proposition}
\label{prop:additive_random_lower}
Under \eqref{model:additive} with random design, it holds that
\begin{align*}
\inf_{\td{\sigma}^2}\sup_{f_k\in\Lambda_{\alpha_k,I_k}(C_\cal{F}),k\in[d]}\sup_{\sigma^2\leq C_\sigma}\sup_{\P_{(\mbf{X},\varepsilon)}\in\cal{P}_{\text{\tiny{add}},(\mbf{X},\varepsilon)}}\E\parr*{\td{\sigma}^2 - \sigma^2}^2 \geq c\parr*{n^{-\frac{8\alpha_{\min}}{4\alpha_{\min}+1}}\vee n^{-1}},
\end{align*}
where $c$ is a fixed positive constant that only depends on $\mbf{\alpha}, C_\cal{F},C_\sigma$ and $C_0,c_0,C_\varepsilon$ in $\cal{P}_{\text{\tiny{add}},(\mbf{X},\varepsilon)}$, and $\td{\sigma}^2$  ranges over all estimators of $\sigma^2$.
\end{proposition} 


We now describe a procedure that matches the lower bound in Proposition \ref{prop:additive_random_lower}, but depends crucially on mutual independence. For illustrative purposes, we again consider the case of only two additive components $f(\cdot)$ and $g(\cdot)$, which are $\alpha$- and $\beta$-H\"older smooth, respectively. Let $X$ and $W$ denote the two covariates. For each $i\in[n]$, define
\begin{align*}
\varepsilon^{X}_i\define f(X_i) + \sigma\varepsilon_i \quad \text{ and }\quad  \varepsilon^{W}_i \define g(W_i) + \sigma\varepsilon_i,
\end{align*}
and their corresponding variances
\begin{align*}
\sigma_{X}^2 \define \E f^2(X) + \sigma^2\quad  \text{ and } \quad \sigma_{W}^2 \define \E g^2(W) + \sigma^2.
\end{align*}
Clearly, we have $\E \varepsilon^{X}_i = 0$ and $\E \varepsilon^{W}_i = 0$, and $\varepsilon^{X}_i$ and $\varepsilon^{W}_i$ are independent of $g(W_i)$ and $f(X_i)$, respectively. Now, notice that the additive model in \eqref{model:additive} can be equivalently viewed as $Y_i = f(X_i) + \varepsilon^{W}_i$. Thus by applying the univariate kernel estimator defined in \eqref{eq:nonpar_estimator} to $\{(Y_i,X_i)\}_{i=1}^n$, which we denote as $\widehat{\sigma}^2_W$, one obtains
\[
\E\parr*{\widehat{\sigma}^2_W - \sigma_W^2}^2 \leq C(n^{-8\alpha/(4\alpha + 1)}\vee n^{-1})
\]
for some fixed positive constant $C$. Similarly, defining $\widehat{\sigma}^2_X$ as $\widehat{\sigma}^2_W$, one has
\[
\E\parr*{\widehat{\sigma}^2_X - \sigma_X^2}^2  \leq C(n^{-8\beta/(4\beta + 1)}\vee n^{-1}).
\]
Lastly, under a finite fourth moment assumption on $\varepsilon$, a sample variance estimator of $\{Y_i\}_{i=1}^n$, denoted as $\widehat{\sigma}^2_{Y}$, achieves the parametric rate $n^{-1}$ in estimating the total variance $\var(Y)$, which can be decomposed as $\E f^2(X) + \E g^2(W) + \sigma^2$.
Consequently, we have shown that the method-of-moments estimator 
\begin{align}
\label{eq:mom_random}
\widehat{\sigma}_{\tiny{\text{moment}},2}^2\define \widehat{\sigma}_{X}^2 + \widehat{\sigma}_{W}^2 - \widehat{\sigma}^2_{Y}
\end{align}
achieves the optimal rate in Proposition \ref{prop:additive_random_lower}.
We summarize the above derivation for the natural extension $\widehat{\sigma}_{\tiny{\text{moment}},d}^2$ to general $d$.

\begin{proposition}
\label{prop:additive_random_independent}
Under \eqref{model:additive} with random design, it holds that 
\begin{align*}
\sup_{f_k\in\Lambda_{\alpha_k,I_k}(C_\cal{F}),k\in[d]}\sup_{\sigma^2\leq C_\sigma}\sup_{\P_{(\mbf{X},\varepsilon)}\in\cal{P}_{\text{\tiny{add}},(\mbf{X},\varepsilon)}}\E\parr*{\widehat{\sigma}_{\tiny{\text{moment}},d}^2 - \sigma^2}^2 \leq C\parr*{n^{-\frac{8\alpha_{\min}}{4\alpha_{\min}+1}}\vee n^{-1}},
\end{align*}
where $C$ is some fixed positive constant that only depends on $\mbf{\alpha},C_\cal{F},C_\sigma$ and $C_0,c_0,C_\varepsilon$ in $\cal{P}_{\text{\tiny{add}},(\mbf{X},\varepsilon)}$.
\end{proposition}

Propositions \ref{prop:additive_random_lower} and \ref{prop:additive_random_independent} together imply the minimax rate over $\cal{P}_{\text{\tiny{add}},(\mbf{X},\varepsilon)}$, which further illustrates the fact that an additive structure in the mean function could possibly avoid the ``curse of dimensionality" in variance estimation. 
However, we note that our results crucially rely on the mutual independence condition. It is still largely unclear if the same minimax rate could apply to the general case without this condition, though a discussion of an interesting connection to variance estimation under linear models shall be made in Section \ref{subsec:linear_model}. 

\subsection{Connection to quadratic functional estimation}
\label{subsec:quad_est}

We now formally state the connection between quadratic functional estimation and variance estimation in \eqref{model:nonpar}, the first of which has been studied in, for example, \cite{doksum1995nonparametric}, \cite{ruppert1995effective}, \cite{huang1999nonparametric}, and \cite{robins2009semiparametric}.

Recall the definition of $Q$ in \eqref{eq:Q} with some non-negative weight function $w(\cdot)$. Squaring both sides of   \eqref{model:nonpar}, multiplying by $w(X_i)$, and then taking the expectation, one has
\begin{align*}
\E\parr*{Y_i^2w(X_i)} = \E\parr*{f^2(X_i)w(X_i)} + \sigma^2\E(w(X_i)\varepsilon_i^2) = Q + \sigma^2\E w(X_i).
\end{align*} 
Under a finite fourth moment assumption on $\varepsilon$, both $\E\parr*{Y_i^2w(X_i)}$ and $\E w(X_i)$ can be estimated at the parametric rate via the sample mean estimator, and $\sigma^2$ can be estimated via $\widehat{\sigma}^2$ in \eqref{eq:nonpar_estimator} with rate $n^{-8\alpha/(4\alpha + 1)}\vee n^{-1}$ under the quadratic risk. Therefore, the estimator
\begin{align*}
\widehat{Q}\define \frac{1}{n}\sum_{i=1}^n Y_i^2w(X_i) - \parr*{\frac{1}{n}\sum_{i=1}^n w(X_i)}\cdot \widehat{\sigma}^2
\end{align*}
achieves the same rate $n^{-8\alpha/(4\alpha + 1)}\vee n^{-1}$. In fact, it is not possible to improve upon this rate since if there exists an estimator $\td{Q}$ with a faster convergence rate, then the ``conjugate" estimator of $\sigma^2$ defined as
\begin{align*}
\td{\sigma}^2 \define \max\bbrace*{\frac{\frac{1}{n}\sum_{i=1}^n Y_i^2w(X_i) - \td{Q}}{\frac{1}{n}\sum_{i=1}^n w(X_i)}, 0}\cdot\mathbbm{1}\bbrace*{\frac{1}{n}\sum_{i=1}^n w(X_i) > 0}
\end{align*}
will also converge to $\sigma^2$ at a faster rate, violating the lower bound in Theorem \ref{thm:nonpar_lower}. 

The following result summarizes the derivation. Recall the definition of $\cal{P}_{\tiny{\text{cv}},(X,\varepsilon)}$ in the beginning of Section \ref{sec:nonpar}.

\begin{proposition}
\label{prop:quad_est}
Suppose the weight function $w(\cdot)$ in the definition of $Q$ is uniformly bounded on $\RR$. Then, it holds that
\begin{align*}
\inf_{\td{Q}}\sup_{f\in\Lambda_{\alpha,I}(C_\cal{F})}\sup_{\sigma^2\leq C_\sigma}\sup_{\P_{(X,\varepsilon)}\in\cal{P}_{\tiny{\text{cv}},(X,\varepsilon)}} \E\parr*{\td{Q} - Q}^2 \asymp n^{-8\alpha/(4\alpha + 1)}\vee n^{-1}
\end{align*}
up to some fixed positive constant that only depends on $w(\cdot)$, $\alpha, C_\cal{F}, C_\sigma$ and $C_0,c_0,C_\varepsilon$ in $\cal{P}_{\tiny{\text{cv}},(X,\varepsilon)}$, where $\td{Q}$  ranges over all estimators of $Q$.
\end{proposition}

\subsection{Connection to the linear model}
\label{subsec:linear_model}
Throughout this paper, we have treated the distribution of $\mbf{X}$ as a nuisance parameter. Interestingly, when we do know the distribution of $\mbf{X}$, variance estimation in nonparametric regression with random design becomes substantially easier with the aid of parallel work in the high-dimensional linear model \citep{verzelen2010goodness, dicker2014variance, kong2018estimating, verzelen2018adaptive}. We first elaborate on this point using the simple model \eqref{model:nonpar}, and then formulate corresponding results for \eqref{model:multivariate_nonpar} and \eqref{model:additive}.

By applying the inverse of the distribution function $F$ of $X$, \eqref{model:nonpar} can be equivalently written as
\begin{align*}
Y_i = \overline{f}(U_i) + \sigma\varepsilon_i, \quad i=1,2,\ldots,n,
\end{align*}
where $\{U_i\}_{i=1}^n = \{F(X_i)\}_{i=1}^n$ are i.i.d. uniform on $[0,1]$, and $\overline{f}(\cdot) \define f\circ F^{-1}(\cdot)$ is still $\alpha$-H\"older smooth under Lipschitz continuity on $F^{-1}$. Then, using a wavelet expansion for H\"older classes (cf. Proposition 2.5 in \cite{meyer1990ondelettes}), one has
\begin{align}
\label{eq:linear_approx}
Y_i = \overline{f}_1(U_i) + \sum_{j=1}^{2^J} \psi_{j}(U_i) + \sigma\varepsilon_i, \quad i=1,2,\ldots,n,
\end{align}
where $\{\psi_j\}_{j=1}^{\infty}$ is an $L_2$-orthonormal wavelet basis under the Lebesgue measure on $[0,1]$, and $\overline{f}_1(\cdot)$ is the remainder term after truncation at resolution $J = J_n$ which satisfies $\|\overline{f}_1\|_\infty = O(2^{-\alpha J_n})$. Let $\mbf{\psi}\define (\psi_1,\ldots, \psi_{2^J})$ and assume without loss of generality that $\E\mbf{\psi} = \mbf{0}_{2^J}$, since a mean shift does not affect the estimation of variance. Moreover, due to the orthonormality of $\{\psi_j\}_{j=1}^{\infty}$, we have $\cov(\mbf{\psi}) = \E(\mbf{\psi}\mbf{\psi}^\top) = \mathbf{I}_{2^J}$. Following \cite{verzelen2018adaptive} and \cite{kong2018estimating},  the estimator
\begin{align*}
\widehat{\sigma}^2_{\tiny{\text{proj}}}\define \frac{1}{n-1}\sum_{i=1}^n (Y_i - \bar{Y})^2 - {n\choose 2}^{-1}\sum_{i<j} Y_iY_j \mbf{\psi}^\top(U_i)\mbf{\psi}(U_j)
\end{align*}
has a variance term of the order $(2^{J_n} + n)/n^2$ and a bias term of the order $2^{-2\alpha J_n}$. Therefore, by choosing the optimal truncation level $2^{J_n} \asymp n^{2/(4\alpha+1)}$, $\widehat{\sigma}^2_{\tiny{\text{proj}}}$ recovers the optimal rate $n^{-8\alpha/(4\alpha+1)}\vee n^{-1}$ in Theorem \ref{thm:nonpar_upper}.

Define $\widehat{\sigma}^2_{\tiny{\text{proj}}, d}$ (with tensor wavelet basis) and $\widehat{\sigma}^2_{\tiny{\text{proj}}, \tiny{\text{add}}}$ as the natural extensions of $\widehat{\sigma}^2_{\tiny{\text{proj}}}$ under \eqref{model:multivariate_nonpar} and \eqref{model:additive}, respectively (see the proofs of Propositions \ref{prop:multivariate_random_known} and \ref{prop:additive_random_known} in the supplement for exact definitions). In the wavelet expansion, we will use $J_k$ to denote the truncation level for the $k$th component of $f(\cdot)$ in \eqref{model:multivariate_nonpar} and $f_k$ in \eqref{model:additive}, and we use $F_k$ to denote the marginal distribution of $X_{1,k}$. Recall that $\under{\alpha}= d/(\sum_{k=1}^d 1/\alpha_k)$ for $\mbf{\alpha} = (\alpha_1,\ldots,\alpha_d)^\top$.

\begin{proposition}[Multivariate nonparametric regression, design known]
\label{prop:multivariate_random_known}
Suppose the distribution of $\mbf{X}$ is known with $\supp(\mbf{X})\subset\mbf{I}$ for some fixed set $\mbf{I}\subset\RR^d$, and $F_k^{-1}(\cdot)$ is Lipschitz continuous for all $k\in[d]$ with some fixed positive constant. Then, when $2^{J_k}$ is chosen to be of the order $n^{2\under{\alpha}/(\alpha_k(4\under{\alpha}+d))}$ for $k\in[d]$ in $\widehat{\sigma}^2_{\tiny{\text{proj}},d}$, it holds that 
\begin{align*}
\sup_{f\in\Lambda_{\mbf{\alpha},\mbf{I}}(C_\cal{F})}\sup_{\sigma^2\leq C_\sigma}\sup_{\E\varepsilon^4\leq C_\varepsilon}\E\parr*{\widehat{\sigma}^2_{\tiny{\text{proj}},d} - \sigma^2}^2 \leq C\parr*{n^{-8\under{\alpha}/(4\under{\alpha} + d)}\vee n^{-1}},
\end{align*}
where $C$ is some fixed positive constant that only depends on $\mbf{\alpha},C_\cal{F},C_\sigma,C_\varepsilon$, and the distribution of $\mbf{X}$.
\end{proposition}

\begin{proposition}[Nonparametric additive model, design known]
\label{prop:additive_random_known}
Suppose the distribution of $\mbf{X}$ is known with $\supp(\mbf{X})\subset I_1\times\ldots\times I_d$ for some fixed intervals $I_1,\ldots,I_d$ on the real line, and $F_k^{-1}(\cdot)$ is Lipschitz continuous for all $k\in[d]$ with some fixed positive constant. Then, when $2^{J_k}$ is chosen to be of the order $n^{2\alpha_k/(4\alpha_k+1)}$ for $k\in[d]$ in $\widehat{\sigma}^2_{\tiny{\text{proj}},\tiny{\text{add}}}$, it holds that 
\begin{align*}
\sup_{f_k\in\Lambda_{\alpha_k,I_k}(C_\cal{F}),k\in[d]}\sup_{\sigma^2\leq C_\sigma}\sup_{\E\varepsilon^4\leq C_\varepsilon}\E\parr*{\widehat{\sigma}^2_{\tiny{\text{proj}},\tiny{\text{add}}} - \sigma^2}^2 \leq C\parr*{n^{-\frac{8\alpha_{\min}}{4\alpha_{\min}+1}}\vee n^{-1}},
\end{align*}
where $C$ is some fixed positive constant that only depends on $\mbf{\alpha},C_\cal{F},C_\sigma,C_\varepsilon$, and the distribution of $\mbf{X}$.
\end{proposition}

As in the classical setting of mean function estimation via orthogonal series, the difference of the rates in Propositions \ref{prop:multivariate_random_known} and \ref{prop:additive_random_known} is clearly explained by the number of wavelet bases used to approximate $f$ in \eqref{model:multivariate_nonpar} and $\{f_k\}_{k=1}^d$ in \eqref{model:additive}. We also note that, quite interestingly,  Proposition \ref{prop:multivariate_random_known} gives results beyond the case $0<\alpha_1,\ldots,\alpha_d\leq 1$ considered in Proposition \ref{prop:multivariate_upper}, and Proposition \ref{prop:additive_random_known} does not rely on the mutual independence of the components of $\mbf{X}$.

\subsection{Adaptive estimation of constant variance}
\label{subsec:adaptive}
In this subsection, we consider adaptive estimation of the variance $\sigma^2$ in model \eqref{model:nonpar}. This is achieved by a Lepski-type procedure \citep{lepskii1991problem,lepskii1992asymptotically}. Let $\widehat{\sigma}^2(h)$ be the estimator in \eqref{eq:nonpar_estimator} with an explicit dependence on the bandwidth parameter $h$. For any given sample size $n$ and fixed positive constant $\delta$, define two positive integers $m_1$ and $m_2$ such that $2^{-m_1}\leq n^{-1} \leq 2^{-m_1+1}$ and $2^{-m_2-1}\leq n^{-(2-\delta)} \leq 2^{-m_2}$, and define the following dyadic grid
\begin{align*}
\cal{H}_\delta\define \bbrace*{2^{-j}: m_1\leq j\leq m_2, j\in\mathbb{Z}}.
\end{align*}
Then, define the estimator $\widehat{\sigma}_{\text{\tiny{adapt}}}^2 \define \widehat{\sigma}^2\parr*{\widehat{h}_\delta}$ with
\begin{align*}
\widehat{h}_\delta \define \max\bbrace*{h\in\cal{H}_\delta: \abs*{\widehat{\sigma}^2(h) - \widehat{\sigma}^2(h^\prime)}\leq \tau (\log n)^{1/2}n^{-1}(h^\prime)^{-1/2}, \forall h^\prime\in\cal{H}_\delta, h^\prime <h}
\end{align*}
for some sufficiently large positive constant $\tau$. If the set being maximized is empty, we will take $\widehat{h}_{\delta} = n^{-(2-\delta)}$.

To state the error bound of $\widehat{\sigma}_{\text{\tiny{adapt}}}^2$, we need the following variant $\cal{P}^{\text{\tiny{adapt}}}_{\text{\tiny{cv}}, (X,\varepsilon)}$ of the distribution class $\cal{P}_{\text{\tiny{cv}}, (X,\varepsilon)}$ considered in Theorem \ref{thm:nonpar_upper}, where we replace the finite fourth-moment assumption (d) therein by the stronger sub-Gaussian tail condition:
\begin{itemize}
\item[(d$^\prime$)] There exist some fixed positive constants $C_{1,\varepsilon}$ and $C_{2,\varepsilon}$ such that $\E\exp(t\varepsilon)\leq C_{1,\varepsilon}\exp(C_{2,\varepsilon}t^2)$ for any $t\in\mathbb{R}$.
\end{itemize}
A similar exponential moment assumption has been made in the context of adaptive estimation under fixed design (cf. Theorems 1 and 2 in \cite{cai2008adaptive}).

\begin{proposition}
\label{prop:cv_upper}
For any given sufficiently small fixed $\alpha_* > 0$, fix some $\delta_*\in(0, 8\alpha_*/(4\alpha_*+1))$. Suppose the kernel $K(\cdot)$ in $\widehat{\sigma}_{\text{\tiny{adapt}}}^2 = \widehat{\sigma}^2\parr*{\widehat{h}_{\delta_*}}$ is chosen such that \eqref{eq:kernel} is satisfied with constants $\overline{M}_K$ and $\underline{M}_K$, and $\tau$ in $\widehat{h}_{\delta_*}$ is chosen to be sufficiently large (only depending on $\delta_*, C_{1,\varepsilon}, C_{2,\varepsilon}$). Then, under \eqref{model:nonpar} with random design, it holds uniformly over all $\alpha \geq \alpha_*$ that
\begin{align*}
\sup_{f\in\Lambda_{\alpha,I}(C_\cal{F})}\sup_{\sigma^2\leq C_\sigma}\sup_{\P_{(X,\varepsilon)}\in\cal{P}^{\text{\tiny{adapt}}}_{\text{\tiny{cv}}, (X,\varepsilon)}}\E\parr*{\widehat{\sigma}_{\text{\tiny{adapt}}}^2 - \sigma^2}^2 \leq C\bbrace*{\parr*{\frac{\log n}{n^2}}^{4\alpha/(4\alpha+1)}\vee n^{-1}},
\end{align*}
where $C$ is some fixed positive constant that only depends on $\delta_*, \overline{M}_K,\underline{M}_K, C_\cal{F}, C_\sigma,$ and $C_0,c_0,C_{1,\varepsilon},C_{2,\varepsilon}$ in $\cal{P}^{\text{\tiny{adapt}}}_{\text{\tiny{cv}}, (X,\varepsilon)}$.
\end{proposition}


The following proposition shows that the extra poly-logarithmic term cannot be removed.

\begin{proposition}
\label{prop:cv_lower}
Let $\phi_{n,\alpha} \define (\log n/n^2)^{2\alpha/(4\alpha+1)}\vee n^{-1/2}$ for any $\alpha > 0$ and positive integer $n$. Consider any fixed positive $\alpha_*$ and $\alpha_*\leq \alpha_1<\alpha_2<\infty$. Then, for any sufficiently large $n$ and sufficiently small fixed positive constant $c$, any estimator $\td{\sigma}^2$ will satisfy that, if
\begin{align*}
\sup_{f\in\Lambda_{\alpha_2,I}(C_\cal{F})}\sup_{\sigma^2\leq C_\sigma}\sup_{\P_{(X,\varepsilon)}\in\cal{P}^{\text{\tiny{adapt}}}_{\text{\tiny{cv}}, (X,\varepsilon)}}\E\parr*{(\td{\sigma}^2 - \sigma^2)/\phi_{n,\alpha_2}}^2 \leq c,
\end{align*}
then
\begin{align*}
\sup_{f\in\Lambda_{\alpha_1,I}(C_\cal{F})}\sup_{\sigma^2\leq C_\sigma}\sup_{\P_{(X,\varepsilon)}\in\cal{P}^{\text{\tiny{adapt}}}_{\text{\tiny{cv}}, (X,\varepsilon)}}\E\parr*{(\td{\sigma}^2 - \sigma^2)/\phi_{n,\alpha_1}}^2 \geq c.
\end{align*}
\end{proposition}


The above two results combined are in line with analogous adaptation results in quadratic functional estimation \citep{efromovich1996optimal,cai2006optimal}.

\section{Proof of Theorem \ref{thm:nonpar_lower}}
\label{sec:proof}

\begin{proof}
We will only prove the lower bound $n^{-8\alpha/(4\alpha + 1)}$ in the regime $0<\alpha<1/4$ since for $\alpha \geq 1/4$, the rate reduces to the parametric rate $n^{-1}$ and the proof is straightforward. Throughout the proof, $C$ represents a generic sufficiently large positive constant and $c$ represents a generic sufficiently small positive constant always taken to be smaller than $1/4$. Both $C$ and $c$ only depend on $\alpha,C_\cal{F},C_\sigma,C_\varepsilon, C_0,c_0$ and might have different values for each occurrence. By appropriately rescaling the parameters in the lower bound construction, without loss of generality, we assume that the sample size $n$ and the constants $C_\cal{F},C_\sigma,C_\varepsilon, C_0$ are sufficiently large, $c_0$ is sufficiently small, and $[0,1]\subset I$. 

We will make use of Le Cam's two point method. Introduce the following constants:
\begin{align}
\label{eq:nonpar_lower_constant}
\theta_n^2 \define h_n^{2\alpha} \define cn^{-4\alpha/(4\alpha+1)} \quad \text{ and } \quad N \define N_n \define 1/(6h_n),
\end{align}
where we tune the constant $c$ in $h_n$ so that $N$ is a positive integer. We now specify $f(\cdot)$, distribution of $X$ and distribution of $\varepsilon$ in the null and alternative hypotheses, $H_0$ and $H_1$, respectively.

\begin{itemize}
\item[] \emph{Choice of $\sigma^2$}: Under $H_0$, let $\sigma^2 = 1+\theta_n^2$. Under $H_1$, let $\sigma^2 = 1$.
\item[] \emph{Choice of $\varepsilon$}: Under both $H_0$ and $H_1$, let $\varepsilon\sim \cal{N}(0, 1)$.
\item[] \emph{Choice of $X$}: Under both $H_0$ and $H_1$, let $X$ be uniformly distributed on the union of the intervals $[(6i-5)h_n, (6i-1)h_n]$ for $i\in[N]$.
\item[] \emph{Choice of $f(\cdot)$}: Under $H_0$, let $f \equiv 0$. Under $H_1$, let $f$ take the value $h_n^{\alpha}r_i$ on $[(6i-5)h_n, (6i-1)h_n]$, where $\{r_i\}_{i=1}^N$ are $N$ i.i.d. symmetric and bounded random variables with distribution $\G$ satisfying
\begin{align}
\label{eq:nonpar_moment_matching}
\int_{-\infty}^\infty x^j\G(dx) = \int_{-\infty}^\infty x^j\varphi(x)dx, \quad j=1,\ldots, q,
\end{align}
where $q$ is some fixed odd integer strictly larger than $1+1/(2\alpha)$. Let $f$ be $0$ at points $6(i-1)h_n$ for $i\in[N]$, and then linearly interpolate $f$ for the rest of the unspecified points on $[0, 1]$.
\end{itemize}
See Figure \ref{fig:nonpar_lower} for an illustration. In the definition of $f(\cdot)$ under $H_1$, the existence of the distribution $\G$ is guaranteed by Lemma \ref{lemma:moment_matching}, and the range of $\{r_i\}_{i=1}^N$, which we denote as $B$, only depends on $\alpha$. 

Clearly, $\sigma^2 \leq C_\sigma$ under both $H_0$ and $H_1$. Moreover, $f(\cdot)$ under both $H_0$ and $H_1$ belongs to $\Lambda_{\alpha,[0,1]}(C_\cal{F})$ due to the boundedness of $\{r_i\}_{i=1}^N$ in $H_1$. Next, we show that the joint distribution of $(X,\varepsilon)$ belongs to $\cal{P}_{\tiny{\text{cv}},(X,\varepsilon)}$. Condition (d) clearly holds and Condition (a) holds with $I = [0,1]$. Condition (b) holds as well by the fact that $p_X(u) = 3/2$ for $u\in[(6i-5)h_n, (6i-1)h_n]$ for $i\in[N]$ and $p_X(u) = 0$ otherwise. Lastly, for Condition (c), it holds by the convolution formula that for any $0 < u < 1/2$
\begin{align*}
p_{\XX}(u) &= \int_u^1 p_X(t)p_X(t-u)dt \geq \sum_{i=\ceil{u/(6h_n)}+1}^{N} \int_{(6i-5)h_n}^{(6i-1)h_n} p_X(t)p_X(t-u)dt\\
&\geq \sum_{i=\ceil{u/(6h_n)}+1}^N  \frac{3}{2}\cdot \frac{3}{2}\cdot 2h_n \geq \frac{3}{8} - 9h_n \geq \frac{1}{4}
\end{align*}
for sufficiently large $n$. Here, the second inequality follows from the fact that for any fixed $t\in[(6i-5)h_n, (6i-1)h_n]$, $p_X(t) = 3/2$ and $p_X(t-u)=0$ on a subset with Lebesgue measure at most $2h_n$. By symmetry of $\XX$, Condition (c) also holds with $\delta_0 = 1/2$ and $\cal{U}_\delta \equiv [-1,1]$.


Denote by $\sigma^2_i, f_i, \P_{i,(X,\varepsilon)}$, $i=0,1$, the choice of $\sigma^2, f$, and $\P_{(X,\varepsilon)}$ under $H_0$ and $H_1$, respectively. Let $\pi$ be the distribution on $\Lambda_{\alpha,I}(C_{\cal{F}})$  such that $f_1\sim \pi$. Moreover, let $\E_{\sigma^2,f,\P_{(X,\varepsilon)}}$ represent the expectation with respect to the model \eqref{model:nonpar} with parameters $\sigma^2, f, \P_{(X,\varepsilon)}$. Then, we have
\begin{align*}
&\ms\inf_{\td{\sigma}^2}\sup_{f\in\Lambda_{\alpha,I}(C_\cal{F})}\sup_{\sigma^2\leq C_\sigma}\sup_{\P_{(X,\varepsilon)}\in \cal{P}_{\tiny{\text{cv}},(X,\varepsilon)}}\E\parr*{\td{\sigma}^2 - \sigma^2}^2\\
&\geq \inf_{\td{\sigma}^2}\bbrace*{\frac{1}{2}\E_{\sigma^2_0, f_0, \P_{0,(X,\varepsilon)}}\parr*{\td{\sigma}^2 - \sigma^2}^2 + \frac{1}{2}\int \E_{\sigma_1^2,f,\P_{1,(X,\varepsilon)}}\parr*{\td{\sigma}^2 - \sigma^2}^2d\pi(f)}\\
&\geq \inf_{\td{\sigma}^2}\bbrace*{\frac{1}{2}\E_{\sigma^2_0, f_0, \P_{0,(X,\varepsilon)}}\parr*{\td{\sigma}^2 - \sigma^2}^2 + \frac{1}{2}\E_{\sigma_1^2,f_1,\P_{1,(X,\varepsilon)}}\parr*{\td{\sigma}^2 - \sigma^2}^2},
\end{align*} 
where the first inequality follows by lower bounding the maximum risk with Bayes risk with prior $\pi$. In what follows, we will use $\P_0$ and $\P_1$ to denote the joint distribution of $\{Y_i,X_i\}_{i=1}^n$ under $H_0$ and $H_1$, respectively. Note that the choice of $\theta_n^2$ in \eqref{eq:nonpar_lower_constant} leads to the desired lower bound under the quadratic loss. Therefore, adopting the standard reduction scheme with Le Cam's two point method (cf. Theorem 2.2 in \cite{tsybakov2009introduction}), it suffices to show that $\TV(\P_0,\P_1) \leq c < 1$. To show this, let $\{\td{r}_i\}_{i=1}^N$ be $N$ i.i.d. standard normal random variables, and $\td{\P}_1$ be the joint distributions of $\{X_i,Y_i\}_{i=1}^n$ under $H_1$ with $\{r_i\}_{i=1}^N$ replaced by $\{\td{r}_i\}_{i=1}^N$. Then, by triangle inequality, we have
\begin{align*}
\TV(\P_0, \P_1) \leq \TV(\P_0, \td{\P}_1) + \TV(\P_1, \td{\P}_1).
\end{align*} 
We will show $\TV(\P_0, \td{\P}_1) \leq c$ and $\TV(\P_1, \td{\P}_1) \leq c$ seperately. 

For the first inequality, define $\mbf{x} \define (x_1,\ldots,x_n)$, $d\mbf{x} \define dx_1\ldots dx_n$ and similarly for $\mbf{y}$ and $d\mbf{y}$. Denote $p_0$, $p_1$, and $\td{p}_1$ as the densities of $\P_0$, $\P_1$, and $\td{\P}_1$ with respect to the Lebesgue measure. Then, we have
\begin{equation}
\begin{aligned}
\label{eq:tv1_1}
\TV(\P_0, \td{\P}_1) &= \frac{1}{2}\int\int \abs*{p_0(\mbf{x}, \mbf{y})-\td{p}_1(\mbf{x},\mbf{y})} d\mbf{x}d\mbf{y}\\
&= \int p(\mbf{x})d\mbf{x}\bbrace*{\frac{1}{2}\int \abs*{p_0(\mbf{y}\mid \mbf{x}) - \td{p}_1(\mbf{y}\mid \mbf{x})}d\mbf{y}}\\
&=  \int p(\mbf{x})d\mbf{x}\TV(\P_0(\mbf{y}\mid\mbf{x}), \td{\P}_1(\mbf{y}\mid \mbf{x})),
\end{aligned}
\end{equation}
where $p(\mbf{x})\define \prod_{i=1}^np_X(X_i)$ stands for the common density of $\{X_i\}_{i=1}^n$ under $\P_0$ and $\td{\P}_1$. Note that under $\P_0$, $\mbf{y}\mid \mbf{x}\sim\cal{N}_n(0, \mbf{\Sigma_0})$, with $\mbf{\Sigma_0} = (1+\theta_n^2)\mathbf{I}_n$. Define $\{b_i\}_{i=1}^n$ to be the location index sequence of $\{X_i\}_{i=1}^n$ taking values in $[N]$, that is, 
\begin{align*}
b_i = j \quad \text{if} \quad X_i \in[(6j-5)h_n, (6j-1)h_n].
\end{align*}
 Then, due to the symmetry of $\{r_i\}_{i=1}^N$ and design of the nonparametric component $f$, it holds that under $\td{\P}_1$, $\mbf{y}\mid \mbf{x}\sim \cal{N}_n(0, \mbf{\Sigma_1})$, with $(\mbf{\Sigma_1})_{ii} = 1 + h_n^{2\alpha} = 1 + \theta_n^2$ and $(\mbf{\Sigma_1})_{ij} = h_n^{2\alpha}\mathbbm{1}\{b_i = b_j\}$ for $i\neq j$. Define $N_0\define \sum_{i\neq j}\mathbbm{1}\{b_i = b_j\}$. Since $\mbf{\Sigma_1}$ is positive definite (see Lemma \ref{lemma:PD} in the supplement), we have by Lemma \ref{lemma:gaussian_tv} that
\begin{align*}
\TV(\P_0(\mbf{y}\mid\mbf{x}), \td{\P}_1(\mbf{y}\mid \mbf{x})) \leq C\frac{\theta_n^2}{1+\theta_n^2}N_0^{1/2} \leq C\theta_n^2N_0^{1/2}.
\end{align*}
Note that $N_0$ is a random variable that depends on $\{X_i\}_{i=1}^n$, and by \eqref{eq:tv1_1} and Jensen's inequality we have 
\begin{align*}
\TV(\P_0, \td{\P}_1) \leq C\theta_n^2\E N_0^{1/2} \leq C\theta_n^2(\E N_0)^{1/2}.
\end{align*}
Some simple algebra shows that $\E N_0 \leq Cn^2h_n$, thus by choosing a sufficiently small $c$ in the definition of $h_n$ in \eqref{eq:nonpar_lower_constant}, we have 
\begin{align*}
\TV(\P_0, \td{\P}_1) \leq C\theta_n^2nh_n^{1/2} \leq c.
\end{align*}

To complete the proof, we now show that $\TV(\P_1,\td{\P}_1) \leq c$. Consider an arbitrary realization of $\{X_i\}_{i=1}^n$, and assume that based on their location indices $\{b_i\}_{i=1}^n$, $\{X_i\}_{i=1}^n$ is partitioned into $L$ clusters with corresponding cardinality $s_\ell$ so that the $X_i$'s in the same cluster have the same value $b_i$. Apparently, we have the relations $1\leq L\leq n$ and $\sum_{\ell=1}^L s_\ell = n$. Let $m_{\max}$ be the maximum cluster size, and define the ``good event" $\Omega_n\define \{m_{\max}\leq K\}$, where $K\define \floor{2/(1-4\alpha))} + 2$. Then, it holds that
\begin{align*}
\TV(\P_1,\td{\P}_1) &= \E\parr*{\mathbbm{1}_{\Omega_n}\TV(\P_1(\mbf{y}\mid \mbf{x}), \td{\P}_1(\mbf{y}\mid \mbf{x})} + \E\parr*{\mathbbm{1}_{\Omega_n^c}\TV(\P_1(\mbf{y}\mid \mbf{x}),\td{\P}_1(\mbf{y}\mid \mbf{x}))}\\
&\leq  \E\parr*{\mathbbm{1}_{\Omega_n}\TV(\P_1(\mbf{y}\mid \mbf{x}), \td{\P}_1(\mbf{y}\mid \mbf{x})} + \P(\Omega_n^c).
\end{align*}
Under the choice of $h_n$ in \eqref{eq:nonpar_lower_constant}, $N$ is of the order $n^{2/(4\alpha + 1)}$, and
\begin{align*}
\lambda_K \define \lim_{n\rightarrow\infty} \frac{n^{K}}{K!N^{K-1}} = 0.
\end{align*}
Thus by Lemma \ref{lemma:multinomial} (and continuity), it holds that $\Omega_n$ has asymptotic probability $1$ under both $\P_1$ and $\td{\P}_1$. As a result, it suffices to upper bound $\TV(\P_1(\mbf{y}\mid \mbf{x}), \td{\P}_1(\mbf{y}\mid \mbf{x}))$ for each realization $\mbf{x}$ in $\Omega_n$, where the maximum cluster size $m_{\max}$ is bounded by a fixed constant.

Denoting $p_{1,\pi_\ell}$ and $\td{p}_{1,\pi_\ell}$ for each $\ell\in[L]$ as the joint density of those $y_i$'s in the $\ell$th cluster $\pi_\ell$ conditioning on the given realization of $\{X_i\}_{i=1}^n$ under $\P_1$ and $\td{\P}_1$, we obtain that 
\begin{align*}
p_1(\mbf{y}\mid \mbf{x}) - \td{p}_1(\mbf{y}\mid \mbf{x}) = \prod_{\ell=1}^L p_{1,\pi_\ell} - \prod_{\ell=1}^L \td{p}_{1,\pi_\ell}. 
\end{align*}
The above inequality further implies by telescoping that
\begin{align*}
\abs*{p_1(\mbf{y}\mid \mbf{x}) - \td{p}_1(\mbf{y}\mid \mbf{x})} \leq \sum_{\ell=1}^L\abs*{p_{1,\pi_\ell} - \td{p}_{1,\pi_\ell}}.
\end{align*}
For each $\ell\in[L]$, $\abs*{p_{1,\pi_\ell} - \td{p}_{1,\pi_\ell}}$ only depends on the $\ell$th cluster through its cardinality, which we now control for a general cluster size $d\geq 1$. Without loss of generality, we assume that $\ell=1$ and the $y_i$'s in this cluster are $\{y_1,\ldots, y_d\}$ with common location index $b_i = 1$ for $i\in[d]$. Then, under the choice of $\theta_n^2$ in \eqref{eq:nonpar_lower_constant}, we clearly have $Y_i  = \theta_nr_1 + \varepsilon_i$ under $\P_1$ and $Y_i = \theta_n\td{r}_1 + \varepsilon_i$ under $\td{\P}_1$ for $i\in[d]$, where the sequence $\{\varepsilon_i\}_{i=1}^d$ follows the standard normal distribution under both $\P_1$ and $\td{\P}_1$. Therefore it holds that
\begin{align*}
p_{1,\pi_1}(y_1,\ldots,y_d) &= \int_{-\infty}^\infty\varphi(y_1-\theta_nv)\ldots\varphi(y_d-\theta_nv)\G(dv), \\
\td{p}_{1,\pi_1}(y_1,\ldots,y_d)  &=\int_{-\infty}^\infty\varphi(y_1-\theta_nv)\ldots\varphi(y_d-\theta_nv)\varphi(v)dv,
\end{align*} 
where $\G$ is the distribution of $\{r_i\}_{i=1}^N$ specified in \eqref{eq:nonpar_moment_matching}. Using the well-known equality $\varphi(t-\theta_nv) = \varphi(t)(\sum_{k=0}^\infty v^k\theta_n^kH_k(t)/k!)$ for any $t,v$, where $H_k$ is the $k$th order Hermite polynomial, it holds that
\begin{align*}
&\ms\varphi(y_1-\theta_nv)\ldots\varphi(y_d-\theta_nv)\\
&= \varphi(y_1)\ldots\varphi(y_d)\sum_{k_1,\ldots,k_d=0}^\infty v^{\sum_{i=1}^d k_i}\theta_n^{\sum_{i=1}^d k_i}\frac{H_{k_1}(y_1)}{k_1!}\ldots\frac{H_{k_d}(y_d)}{k_d!}\\
&= \varphi(y_1)\ldots\varphi(y_d)\sum_{k=0}^\infty v^k\theta_n^k\sum_{k_1+\ldots+k_d=k}\frac{H_{k_1}(y_1)}{k_1!}\ldots\frac{H_{k_d}(y_d)}{k_d!}
\end{align*}
and therefore
\begin{align*}
&\ms p_{1,\pi_1}(y_1,\ldots,y_d) - \td{p}_{1,\pi_1}(y_1,\ldots,y_d)\\
 &= \varphi(y_1)\ldots\varphi(y_d)\sum_{k=0}^\infty \theta_n^k\sum_{k_1+\ldots+k_d =k} \frac{H_{k_1}(y_1)}{k_1!}\ldots\frac{H_{k_d}(y_d)}{k_d!}\int v^k (\G - \Phi)(dv)\\
&=\varphi(y_1)\ldots\varphi(y_d)\sum_{k=p}^\infty \theta_n^{2k}\sum_{k_1+\ldots+k_d =2k} \frac{H_{k_1}(y_1)}{k_1!}\ldots\frac{H_{k_d}(y_d)}{k_d!}\int v^{2k} (\G - \Phi)(dv),
\end{align*}
where the second equality follows by the symmetry and moment matching property of $\G$ in \eqref{eq:nonpar_moment_matching} and $p \define (q+1)/2$ is a positive integer. This further yields
\begin{align*}
&\ms\abs*{p_{1,\pi_1}(y_1,\ldots,y_d) - \td{p}_{1,\pi_1}(y_1,\ldots,y_d)}\\
& \leq \varphi(y_1)\ldots\varphi(y_d)\sum_{k=p}^\infty\theta_n^{2k}\sum_{k_1+\ldots+k_d=2k}\frac{\abs*{H_{k_1}(y_1)}}{k_1!}\ldots\frac{\abs*{H_{k_d}(y_d)}}{k_d!}\int v^{2k}\G(dv) +\\
& \ms\varphi(y_1)\ldots\varphi(y_d)\sum_{k=p}^\infty\theta_n^{2k}\sum_{k_1+\ldots+k_d=2k}\frac{\abs*{H_{k_1}(y_1)}}{k_1!}\ldots\frac{\abs*{H_{k_d}(y_d)}}{k_d!}\int v^{2k}\varphi(v) dv\\
&\define I + II.
\end{align*}
For term $I$, since $\G$ is compactly supported on $[-B,B]$, one clearly has
\begin{align*}
I \leq \varphi(y_1)\ldots\varphi(y_d)\sum_{k=p}^\infty\theta_n^{2k}B^{2k}\sum_{k_1+\ldots+k_d=2k}\frac{\abs*{H_{k_1}(y_1)}}{k_1!}\ldots\frac{\abs*{H_{k_d}(y_d)}}{k_d!}.
\end{align*}
For term $II$, using the equality $\int \varphi(v)v^{2k}dv = (2k-1)!!$, with $(2k-1)!!\define (2k-1)(2k-3)\ldots 1$, we obtain
\begin{align*}
II = \varphi(y_1)\ldots\varphi(y_d)\sum_{k=p}^\infty\theta_n^{2k}(2k-1)!!\sum_{k_1+\ldots+k_d=2k}\frac{\abs*{H_{k_1}(y_1)}}{k_1!}\ldots\frac{\abs*{H_{k_d}(y_d)}}{k_d!}.
\end{align*}
We now upper bound $\int_{-\infty}^\infty \abs*{H_k(t)}\varphi(t)dt$ for an arbitrary positive integer $k$. When $k$ is even, as has been calculated in \cite{wang2008effect} (cf. chain of inequality after Equation (19) on Page 662), $\int_{-\infty}^\infty \abs*{H_k(t)}\varphi(t)dt \leq 2^{k/2}(k-1)!!$. When $k$ is odd, set $k = 2\td{k} + 1$, then we have
\begin{align*}
\int_{-\infty}^\infty \abs*{H_k(t)}\varphi(t)dt &= \int_{-\infty}^\infty \varphi(t)\abs*{(2\td{k}+1)!\sum_{m=0}^{\td{k}}\frac{(-1)^mt^{2\td{k}+1-2m}}{m!(2\td{k}+1-2m)!2^m}}dt\\
&\leq \sum_{m=0}^{\td{k}} \frac{(2\td{k}+1)!}{m!(2\td{k}+1-2m)!2^m}\int_{-\infty}^\infty |t|^{2\td{k}+1-2m}\varphi(t)dt\\
&= \sqrt{\frac{2}{\pi}}\sum_{m=0}^{\td{k}} \frac{(2\td{k}+1)!(2\td{k}-2m)!!}{m!(2\td{k}+1-2m)!2^m}\\
&= \sqrt{\frac{2}{\pi}}\sum_{m=0}^{\td{k}} \frac{(2\td{k}+1)!(2m)!!}{(\td{k}-m)!(2m+1)!2^{\td{k}-m}}\\
&= \sqrt{\frac{2}{\pi}}(2\td{k}+1)!!\sum_{m=0}^{\td{k}}\frac{\td{k}!}{m!(\td{k}-m)!}\frac{(m!)^22^{2m}}{(2m+1)!}\\
&\leq (2\td{k}+1)!!\sum_{m=0}^{\td{k}}\frac{\td{k}!}{m!(\td{k}-m)!}\\
&= (2\td{k}+1)!!2^{\td{k}},
\end{align*}
where in the third line we use the fact that $\int_{-\infty}^{\infty} |t|^{2\ell+1}\varphi(t)dt = \sqrt{2/\pi}(2\ell)!!$ for any positive integer $\ell$. Define for any positive integer $k$: $[k]_1 \define k-1$ if $k$ is even and $k$ if $k$ is odd, and $[k]_2 \define k/2$ if $k$ is even and $(k-1)/2$ if $k$ is odd. Then, the above calculation implies that $\int_{-\infty}^\infty \abs*{H_k(t)}\varphi(t)dt \leq ([k]_1)!!2^{[k]_2}$ for any $k$, and moreover, it can be readily checked that $([k]_1)!!/(k!) = 1/(2^{[k]_2}([k]_2)!)$. Therefore, for term $I = I(y_1,\ldots,y_d)$, we have
\begin{align*}
&\ms\int_{\RR^d}I(y_1,\ldots,y_d)dy_1\ldots dy_d\\
&\leq \sum_{k=p}^\infty\theta_n^{2k}(B^2)^k\sum_{k_1+\ldots+k_d=2k}\frac{1}{(k_1)!\ldots(k_d)!}([k_1]_1)!!2^{[k_1]_2}\ldots([k_d]_1)!!2^{[k_d]_2}\\
&= \sum_{k=p}^\infty\theta_n^{2k}(B^2)^k\sum_{k_1+\ldots+k_d=2k}\frac{1}{([k_1]_2)!\ldots([k_d]_2)!}.
\end{align*}
Now note that the number of $d$-tuple $(k_1,\ldots,k_d)$ such that $k_1+\ldots+k_d = 2k$ is upper bounded by $(Ck)^d$, which is further bounded by $C^k$ for every $k\geq 0$ with some sufficiently large $C$ that only depends on $d$, and for each such tuple, it holds that 
\begin{align*}
k-\frac{d}{2}= \sum_{i=1}^d \frac{k_i-1}{2}\leq \sum_{i=1}^d [k_i]_2\leq \sum_{i=1}^d \frac{k_i}{2} = k,
\end{align*}
thus we have
\[
\sum_{k_1+\ldots+k_d=2k}\{([k_1]_2)!\ldots([k_d]_2)!\}^{-1}\leq C^k\sum_{k-d/2\leq \bar{k}_1+\ldots+\bar{k}_d\leq k}\{(\bar{k}_1)!\ldots(\bar{k}_d)!\}^{-1}.
\]
For the latter quantity, we have by the multinomial identity
\[
\sum_{x_1+\ldots+x_{d+1}=k}k!/(x_1!\ldots x_{d+1}!)(d+1)^{-k} = 1
\]
that
\begin{align*}
\frac{(d+1)^k}{k!} &= \sum_{\bar{k}_1+\ldots+\bar{k}_{d+1} = k}\frac{1}{(\bar{k}_1)!\ldots(\bar{k}_{d+1})!}\\
& = \sum_{\bar{k}_1+\ldots+\bar{k}_d \leq k}\frac{1}{(\bar{k}_1)!\ldots(\bar{k}_d)!(k-(\bar{k}_1+\ldots+\bar{k}_d))!}\\
&\geq \sum_{k-d/2\leq \bar{k}_1+\ldots+\bar{k}_d \leq k} \frac{1}{(\bar{k}_1)!\ldots(\bar{k}_d)!(k-(\bar{k}_1+\ldots+\bar{k}_d))!}\\
&\geq \parr*{\parr*{\frac{d}{2}}!}^{-1}\sum_{k-d/2\leq \bar{k}_1+\ldots+\bar{k}_d \leq k} \frac{1}{(\bar{k}_1)!\ldots(\bar{k}_d)!}.
\end{align*}
This concludes that
\begin{align*}
\int_{\RR^d}I(y_1,\ldots,y_d)dy_1\ldots dy_d \leq \theta_n^{2p}\sum_{k=p}^\infty \frac{(CB^2)^k}{k!} \leq \theta_n^{2p}e^{CB^2}.
\end{align*}
Using a similar argument for $II = II(y_1,\ldots,y_d)$, we obtain
\begin{align}
\label{eq:nonpar_ii}
\int_{\RR^d} II(y_1,\ldots,y_d)dy_1\ldots dy_d \leq \sum_{k=p}^\infty \frac{(2k-1)!!}{k!}\theta_n^{2k}C^k  &= \sum_{k=p}^\infty \frac{(2k-1)!!}{(2k)!!}\theta_n^{2k}(2C)^k  \leq \theta_n^{2p}C^p
\end{align}
since $\theta_n^2 < 1/C$ for sufficiently large $n$. 

Putting together the pieces, we have for every realization $\mbf{x}$ in $\Omega_n$
\begin{align*}
\int_{\RR^n} \abs*{p_1(\mbf{y}\mid \mbf{x}) - \td{p}_1(\mbf{y}\mid \mbf{x})} d\mbf{y}&\leq \sum_{\ell=1}^L\int_{\RR^{|\pi_\ell|}} \abs*{p_{1,\pi_\ell} - \td{p}_{1,\pi_\ell}} \leq L\max_{1\leq d\leq K}\theta_n^{2p}(e^{CB^2} + C^p) \\
&\leq n\theta_n^{2p}(e^{CB^2} + C^p) \leq c.
\end{align*}
Here, the second inequality follows since every $\abs*{p_{1,\pi_\ell} - \td{p}_{1,\pi_\ell}}$ depends on the $\ell$th cluster only through its cardinality, the third inequality follows since $L\leq n$ and $K$ is a fixed absolute constant that only depends on $\alpha$, and the last inequality follows due to the choice $\theta_n^2= h_n^{2\alpha} = cn^{-4\alpha/(4\alpha+1)}$ and the value of $p$. This completes the proof.
\end{proof}

\begin{lemma}[Lemma 1, \cite{wang2008effect}]
\label{lemma:moment_matching}
For any fixed positive integer $q$, there exist a $B < \infty$ and a symmetric distribution $\G$ on $[-B, B]$ such that $\G$ and the standard normal distribution have the same first $q$ moments, that is,
\begin{align*}
\int_{-B}^{B} x^j\G(dx) = \int_{-\infty}^\infty x^j\varphi(x)dx, \quad j=1,\ldots, q.
\end{align*}
\end{lemma}

\begin{lemma}[Theorem 1.1, \cite{devroye2018total}]
\label{lemma:gaussian_tv}
If $\mbf{\mu}\in\RR^d$ and $\mbf{\Sigma}_1$ and $\mbf{\Sigma}_2$ are positive definite $d\times d$ matrices, then
\begin{align*}
\frac{1}{100}\leq \frac{\text{TV}\parr*{\cal{N}_d(\mbf{\mu}, \mbf{\Sigma}_1), \cal{N}_d(\mbf{\mu}, \mbf{\Sigma}_2)}}{\min\{1, \|\mbf{\Sigma}_1^{-1}\mbf{\Sigma}_2 - \mathbf{I}_d\|_F\}} \leq \frac{3}{2}.
\end{align*} 
\end{lemma}


For the following lemma, we first introduce some terminology regarding the multinomial distribution. Let $m, M$ be two positive integers, and the random vector $(f_1,\ldots, f_M)$ be the multinomial count with total count $m$ and equal probability $(1/M,1/M,\ldots, 1/M)$. Define $\rho \define m/M$. For any positive integer $r\geq 2$, define 
$\lambda \define \lambda_r\define \lim_{m\rightarrow \infty} m^r/(r!M^{r-1})$. Following \cite{kolchin1978random} (Chapter 2, Equation (11)), we will call the domain of variation $m,M\rightarrow \infty$, in which 
\begin{align*}
\rho \rightarrow 0, \quad 0 < \lambda_r < \infty
\end{align*}
the \emph{left-hand $r$-domain}. The following lemma characterizes the asymptotic behavior of the maximum frequency $f_{\tiny{\text{max}}}$ defined as $\max_{1\leq j\leq M}f_j$.

\begin{lemma}[Theorem 1 of Section 2.6, \cite{kolchin1978random}]
\label{lemma:multinomial}
Suppose the multinomial distribution with total count $m$ and equal probability $(1/M,\ldots, 1/M)$ is in the left-hand $r$-domain for some positive integer $r\geq 2$ with limit $\lambda_r$, then it holds that
\begin{align*}
\P(f_{\tiny{\text{max}}} = r - 1)\rightarrow e^{-\lambda_r}\text{ and } \quad \P(f_{\tiny{\text{max}}} = r)\rightarrow 1 - e^{-\lambda_r},
\end{align*}
i.e., the maximum frequency converges asymptotically to a two-point distribution.
\end{lemma}




\bibliographystyle{apalike}
\bibliography{reference}

\newpage

\appendix

\section*{Appendix}

Throughout the supplement, we continue to use the notation introduced in the main paper. We also use the following new notation. For any positive integer $d\geq 2$, $\mathbb{S}^{d-1}$ stands for the unit Euclidean sphere in $\RR^d$. For a real vector $\mbf{x}$, define $\|\mbf{x}\|_0$ as the number of nonzero coordinates of $\mbf{x}$. When writing the H\"older class $\Lambda_{\alpha,I}(C)$, we will omit the domain $I$ in the subscript for simplicity. For two distributions $\P$ and $\Q$ on $\RR$, we will write $\P*\Q$ as their convolution.


\section{Proofs of results in Section \ref{sec:nonpar}}

\subsection{Proof of Theorem \ref{thm:nonpar_upper}}
\label{subsec:proof_upper}
\begin{proof}

Throughout the proof, we will use $C, c$ to denote two generic fixed positive constants that only depend on $\overline{M}_K, \under{M}_K,\alpha, C_{\cal{F}}, C_\sigma, C_\varepsilon, C_0, c_0$. $C$ and $c$ might have different values at each occurrence. We also use the notation $\td{W}_{ij}\define W_i - W_j$ for a generic random variable $W$. 

Denote the two U-statistics on the numerator and denominator of $\hat{\sigma}^2$ respectively as $U_1, U_2$, with corresponding mean values $\theta_1,\theta_2$. That is, with $i\neq j$,
\begin{align*}
\theta_1 \define \E\bbrace*{K_h(X_i-X_j)(Y_i-Y_j)^2/2} \quad \text{ and } \quad \theta_2\define  \E K_h(X_i-X_j).
\end{align*}
Define the ``good" event $\cal{E}\define \{U_2 \geq \theta_2/2\}$ and $\cal{E}^c$ as its complement, then it holds that
\begin{align}
\label{eq:exp_decom}
\E\parr*{\hat{\sigma}^2 - \sigma^2}^2 = \E\bbrace*{\parr*{\frac{U_1 - U_2\sigma^2}{U_2}}^2\mathbbm{1}\{\cal{E}\}} + \E\bbrace*{\parr*{\frac{U_1 - U_2\sigma^2}{U_2}}^2\mathbbm{1}\{\cal{E}^c\}}.
\end{align}
By definition of $\cal{E}$, the first term satisfies that
\begin{align*}
\E\bbrace*{\parr*{\frac{U_1 - U_2\sigma^2}{U_2}}^2\mathbbm{1}\{\cal{E}\}} \leq \frac{4}{\theta_2^2}\E\parr*{U_1 - U_2\sigma^2}^2.
\end{align*}
For $\theta_2$, we have
\begin{align*}
\theta_2 &= \E K_h(X_i-X_j) = \int \frac{1}{h}K\parr*{\frac{v}{h}}p_{\XX}(v)dv = \int_{-1}^1 K(u)p_{\XX}(uh)du\\
&\geq \int_{\cal{U}_h} K(u)p_{\XX}(uh)du\geq\inf_{u\in\cal{U}_h}p_{\XX}(uh)\inf_{u\in[-1,1]}K(u)\lambda(\cal{U}_h) \geq \under{M}_Kc_0^2.
\end{align*}
Here, the third equality follows from the fact that $K(\cdot)$ is supported in $[-1,1]$, and $\cal{U}_h$ starting from the first inequality is defined in Condition (c) in $\cal{P}_{\tiny{\text{vf}},(X,\varepsilon)}$ (note that for any fixed $\delta_0 > 0$ given therein, $h \leq \delta_0$ for sufficiently large $n$). Moreover, it holds that
\begin{align*}
\E\parr*{U_1 - U_2\sigma^2}^2 \leq 3\bbrace*{\E\parr*{U_1 - \theta_1}^2 + \sigma^4\E\parr*{U_2 - \theta_2}^2 + \parr*{\theta_1 - \theta_2\sigma^2}^2}.
\end{align*}
By Lemmas \ref{lemma:concentrate_U1} and \ref{lemma:concentrate_U2} and the fact that $\sigma^4 \leq C_\sigma^2$, we have
\begin{align*}
\E\parr*{U_1 - \theta_1}^2 + \sigma^4\E\parr*{U_2 - \theta_2}^2 \leq C(n^{-1} + n^{-2}h^{-1}).
\end{align*}
For the third term $\parr*{\theta_1 - \theta_2\sigma^2}^2$, we have 
\begin{align*}
\theta_1 = \E\bbrace*{K_h(X_i-X_j)(Y_i-Y_j)^2/2} = \E\bbrace*{K_h(X_i-X_j)(f(X_i)-f(X_j))^2/2} + \theta_2\sigma^2
\end{align*} 
and
\begin{align*}
&\ms\E\bbrace*{K_h(X_i-X_j)(f(X_i)-f(X_j))^2/2} \leq C\E\bbrace*{\frac{1}{h}K\parr*{\frac{\XX}{h}}\abs*{\XX}^{2(\alpha\wedge 1)}}\\
&= C\int \frac{1}{h}K\parr*{\frac{u}{h}}\abs*{u}^{2(\alpha\wedge 1)}p_{\XX}(u)du = C\int K(v)h^{2(\alpha\wedge 1)}|v|^{2(\alpha\wedge 1)}p_{\XX}(vh)dv\\
&\leq Ch^{2(\alpha\wedge 1)}\sup_{u\in\RR}p_{\XX}(u)\int K(v)|v|^{2(\alpha\wedge 1)}dv \leq Ch^{2(\alpha\wedge 1)}.
\end{align*}
Here, the first inequality follows since $f\in\Lambda_{\alpha}(C_{\cal{F}})$, and the last inequality follows from Condition (b) in $\cal{P}_{\tiny{\text{vf}},(X,\varepsilon)}$ and the convolution formula. 
Putting together the pieces, the choice of $h$ in \eqref{eq:h_opt_nonpar} in the main paper yields
\begin{align*}
\E\bbrace*{\parr*{\frac{U_1 - U_2\sigma^2}{U_2}}^2\mathbbm{1}\{\cal{E}\}} \leq C(h^{4(\alpha\wedge 1)} + n^{-1} + n^{-2}h^{-1}) \leq C(n^{-8\alpha/(4\alpha +1)}\vee n^{-1}).
\end{align*}
For the second term in \eqref{eq:exp_decom}, we have 
\begin{align*}
\E\bbrace*{\parr*{\frac{U_1 - U_2\sigma^2}{U_2}}^2\mathbbm{1}\{\cal{E}^c\}} \leq 2\sigma^4\P\parr*{\cal{E}^c} + 2\E\bbrace*{\parr*{\frac{U_1}{U_2}}^2\mathbbm{1}\{\cal{E}^c\}}.
\end{align*}
Direct calculation shows that
\begin{align*}
\parr*{\frac{U_1}{U_2}}^2 &= \frac{\sum_{i<j,i^\prime < j^\prime} K_h\parr*{\td{X}_{ij}}K_h\parr*{\td{X}_{i^\prime j^\prime}}(Y_i - Y_j)^2(Y_{i^\prime}-Y_{j^\prime})^2/4}{\sum_{i<j, i^\prime < j^\prime} K_h\parr*{\td{X}_{ij}}K_h\parr*{\td{X}_{i^\prime j^\prime}}}\\
&\leq \frac{\sum_{i<j,i^\prime < j^\prime} K_h\parr*{\td{X}_{ij}}K_h\parr*{\td{X}_{i^\prime j^\prime}}\bbrace*{(f(X_i) - f(X_j))^2 + \sigma^2\td{\varepsilon}_{ij}^2}\bbrace*{(f(X_{i^\prime}) - f(X_{j^\prime}))^2 + \sigma^2\td{\varepsilon}_{i^\prime j^\prime}^2}}{\sum_{i<j,i^\prime < j^\prime} K_h\parr*{\td{X}_{ij}}K_h\parr*{\td{X}_{i^\prime j^\prime}}}\\
&\leq C\frac{\sum_{i<j,i^\prime < j^\prime} K_h\parr*{\td{X}_{ij}}K_h\parr*{\td{X}_{i^\prime j^\prime}}\bbrace*{\abs*{\td{X}_{ij}}^{2(\alpha\wedge 1)} + \sigma^2\td{\varepsilon}_{ij}^2}\bbrace*{\abs*{\td{X}_{i^\prime j^\prime}}^{2(\alpha\wedge 1)} + \sigma^2\td{\varepsilon}_{i^\prime j^\prime}^2}}{\sum_{i<j, i^\prime < j^\prime} K_h\parr*{\td{X}_{ij}}K_h\parr*{\td{X}_{i^\prime j^\prime}}}\\
&\leq C\parr*{h^{4(\alpha\wedge 1)} + \sigma^2h^{2(\alpha\wedge 1)}\max_{i<j}\td{\varepsilon}_{ij}^2 + \sigma^4\max_{i<j,i^\prime < j^\prime}\td{\varepsilon}_{ij}^2\td{\varepsilon}_{i^\prime j^\prime}^2},
\end{align*}
where the last inequality follows by the support of $K(\cdot)$. By the condition $\sigma^2\leq C_\sigma$ and the independence of $\{\varepsilon_i\}_{i=1}^n$ and $\mathbbm{1}\{\cal{E}\}$, this implies that 
\begin{align*}
\E\bbrace*{\parr*{\frac{U_1 - U_2\sigma^2}{U_2}}^2\mathbbm{1}\{\cal{E}^c\}} \leq C\P\parr*{\cal{E}^c}\bbrace*{1 + \E\max_{i<j}\td{\varepsilon}_{ij}^2 + \E\max_{i<j,i^\prime <j^\prime}\td{\varepsilon}_{ij}^2\td{\varepsilon}_{i^\prime j^\prime}^2}.
\end{align*}
Applying the first part of Lemma \ref{lemma:concentrate_U2} with $v = n\theta_2^2/16$, $u = n^2h\theta_2^2/16$ with the condition $h = \Omega(n^{-(2-\delta)})$ being satisfied with $\delta = 8\alpha/(4\alpha+1)$ and $\delta = 1$ for $\alpha \leq 1/4$ and $\alpha > 1/4$ respectively, it holds that
\begin{align*}
\P\parr*{\cal{E}^c} = \P(|U_2 - \theta_2|\geq \theta_2/2) \leq C\bbrace*{\exp(-\theta_2^2n/16) + \exp(-\theta_2^2n^2h/16)}.
\end{align*}
Moreover, by Condition (d) in $\cal{P}_{\tiny{\text{vf}},(X,\varepsilon)}$, there exists some fixed positive constant $\eta$ such that
\[
1 + \E\max_{i<j}\td{\varepsilon}_{ij}^2 + \E\max_{i<j,i^\prime <j^\prime}\td{\varepsilon}_{ij}^2\td{\varepsilon}_{i^\prime j^\prime}^2 \leq Cn^\eta.
\]
Putting together the pieces and using the fact that $n^2h\rightarrow\infty$ as $n\rightarrow\infty$, it yields
\begin{align*}
\E\bbrace*{\parr*{\frac{U_1 - U_2\sigma^2}{U_2}}^2\mathbbm{1}\{\cal{E}^c\}} = o(n^{-8\alpha/(4\alpha +1)}\vee n^{-1}).
\end{align*}
This completes the proof.
\end{proof} 

\subsection{Supporting lemmas}
\begin{lemma}
\label{lemma:concentrate_U1}
Suppose $f\in\Lambda_{\alpha}(C_{\cal{F}})$ and $\sigma^2\leq C_\sigma$ for some fixed constants $C_{\cal{F}},C_\sigma$ and the joint distribution of $(X,\varepsilon)$ satisfies Conditions (a), (b) and (d) in $\cal{P}_{\tiny{\text{cv}},(X,\varepsilon)}$ with constants $C_0,C_\varepsilon$. Then, the U-statistic $U_1$ defined in the proof of Theorem \ref{thm:nonpar_upper} satisfies 
\begin{align*}
\E\parr*{U_1 - \theta_1}^2 \leq C\parr*{n^{-1} \vee n^{-2}h^{-1}},
\end{align*}
where $C$ is some fixed positive constant that only depends on $\overline{M}_K, \under{M}_K, \alpha, C_{\cal{F}}, C_\sigma, C_\varepsilon, C_0$.
\end{lemma}
\begin{proof}
Denote $g$ as the kernel of $U_1$, that is,
\[
g(\mbf{D}_i, \mbf{D}_j) \define K_h(X_i - X_j)(Y_i - Y_j)^2/2, \quad \mbf{D}_i \define (X_i, \varepsilon_i)^\top.
\]
Recall that $\theta_1 = \E g(\mbf{D}_i, \mbf{D}_j)$ for $i\neq j$. Then, it holds that 
\begin{align}
\label{eq:summand}
\E\parr*{U_1 - \theta_1}^2 = {n\choose 2}^{-2}\sum_{i<j,i^\prime<j^\prime} \E\bbrace*{\parr*{g(\mbf{D}_i, \mbf{D}_j) - \theta_1}\parr*{g(\mbf{D}_{i^\prime}, \mbf{D}_{j^\prime}) - \theta_1}}.
\end{align}
When $i,j,i^\prime,j^\prime$ take four different values, the expectation is zero. When they take three values, say, $i = i^\prime < j < j^\prime$, by writing $\E_\varepsilon$ as the conditional expectation given $\{X_i\}_{i=1}^n$, we have
\begin{align*}
&\ms\E\parr*{g(\mbf{D}_i,\mbf{D}_j)g(\mbf{D}_i, \mbf{D}_{j^\prime})}\\
&= \frac{1}{4}\E\bbrace*{\frac{1}{h^2}K\parr*{\frac{X_i - X_j}{h}}K\parr*{\frac{X_i - X_{j^\prime}}{h}}(Y_i - Y_j)^2(Y_i - Y_{j^\prime})^2}\\
&\lesssim \E\bbrace*{\frac{1}{h^2}K\parr*{\frac{X_i - X_j}{h}}K\parr*{\frac{X_i - X_{j^\prime}}{h}}((f(X_i) - f(X_j))^2 + \sigma^2\td{\varepsilon}_{ij}^2)((f(X_i) - f(X_{j^\prime}))^2 + \sigma^2\td{\varepsilon}_{ij^\prime}^2)}\\
&\lesssim \E\bbrace*{\frac{1}{h^2}K\parr*{\frac{X_i - X_j}{h}}K\parr*{\frac{X_i - X_{j^\prime}}{h}}\E_\varepsilon\bbrace*{\parr*{\abs*{\XX}^{2(\alpha\wedge 1)} + \sigma^2\td{\varepsilon}_{ij}^2}\parr*{\abs*{\td{X}_{ij^\prime}}^{2(\alpha\wedge 1)} + \sigma^2\td{\varepsilon}_{ij^\prime}^2}}}\\
&\lesssim  \E\bbrace*{\frac{1}{h^2}K\parr*{\frac{X_i - X_j}{h}}K\parr*{\frac{X_i - X_{j^\prime}}{h}}\parr*{\abs*{\XX\td{X}_{ij^\prime}}^{2(\alpha\wedge 1)} + 2\sigma^2\parr*{\abs*{\XX}^{2(\alpha\wedge1)}+\abs*{\td{X}_{ij^\prime}}^{2(\alpha\wedge 1)}} + \E\varepsilon^4\cdot\sigma^4}}\\
&\lesssim \E\bbrace*{\frac{1}{h^2}K\parr*{\frac{X_i - X_j}{h}}K\parr*{\frac{X_i - X_{j^\prime}}{h}}} \\
&= \int K(v)K(w)p_{X}(u)p_X(u +hv)p_{X}(u + hw)dudvdw \lesssim 1.
\end{align*}
In the last line, we invoke Conditions (b) and (d) in $\cal{P}_{\tiny{\text{cv}},(X,\varepsilon)}$. Moreover, it can be readily calculated that $\theta_1 = O(1)$. This concludes that the summand in \eqref{eq:summand} is bounded by a fixed constant when $i,j,i^\prime,j^\prime$ take three different values. Lastly, performing a similar analysis, we obtain that $\E\bbrace*{\parr*{g(\mbf{D}_i, \mbf{D}_j) - \theta_1}\parr*{g(\mbf{D}_{i^\prime}, \mbf{D}_{j^\prime}) - \theta_1}} = O(1/h)$ when $i = i^\prime$ and $j = j^\prime$. We therefore conclude that
\begin{align*}
\var(U_1) \lesssim \frac{n^3 + n^2h^{-1}}{n^4} \asymp n^{-1} + n^{-2}h^{-1}.
\end{align*}
This completes the proof.
\end{proof}

\begin{lemma}
\label{lemma:concentrate_U2}
Suppose $h_n \gtrsim n^{-(2-\delta)}$ for some $0 < \delta < 2$. Then, assuming Condition (b) in $\cal{P}_{\tiny{\text{cv}},(X,\varepsilon)}$ with constant $C_0$, the U-statistic $U_2$ defined in the proof of Theorem \ref{thm:nonpar_upper} satisfies 
\begin{align*}
\P\parr*{\abs*{U_2 - \theta_2}\geq C(v^{1/2}n^{-1/2} + u^{1/2}n^{-1}h^{-1/2})} \leq C(\exp(-u) + \exp(-v))
\end{align*}
for any $u, v> 0$, and
\begin{align*}
\E\parr*{U_2 - \theta_2}^2 \leq C(n^{-1} \vee n^{-2}h^{-1}),
\end{align*}
where $C$ is some fixed positive constant that only depends on $\overline{M}_K, \under{M}_K,\alpha, C_0$.
\end{lemma}
\begin{proof}
We first prove the concentration inequality by upper bounding the 5 quantities in Lemma \ref{lemma:Ubernstein}. Denote $g$ as the kernel of $U_2$ and $g_1$ as its linear part, that is, for some  $i\neq j$,
\begin{align*}
g_1(X_i) \define \E\parr*{g(X_i, X_j)\mid X_i} \define \E\parr*{K_h(X_i - X_j)\mid X_i}.
\end{align*}  
For $B_1$, we have
\begin{align*}
g_1(X_i) = \int \frac{1}{h}K\parr*{\frac{u-X_i}{h}}p_X(u)du = \int K(u)p_X(uh + X_i)du \lesssim 1
\end{align*}
due to Condition (b) in $\cal{P}_{\tiny{\text{cv}},(X,\varepsilon)}$. Thus it also holds $\nu^2_1 \lesssim 1$. For $B_2$, we have
\begin{align*}
B_2^2 &= n\sup_{X_i}\E\bbrace*{g^2(X_i, X_j)\mid X_i} = n\sup_{X_i} \int \frac{1}{h^2}K^2\parr*{\frac{u-X_i}{h}}p_X(u)du\\
& \lesssim \frac{n}{h}\sup_{X_i}\int K(u)p_X(uh+X_i)du \lesssim nh^{-1},
\end{align*}
where in the first inequality we use the condition that $K(\cdot)$ is bounded by $\overline{M}_K$. Moreover, we clearly have $B_3\lesssim h^{-1}$. Lastly, for $\nu_2^2$, it holds that 
\begin{align*}
\nu_2^2 = \int \frac{1}{h^2}K^2\parr*{\frac{u}{h}}p_{\XX}(u)du \lesssim \frac{1}{h}\int K(u)p_{\XX}(uh)du \lesssim \frac{1}{h},
\end{align*}
where the last inequality follows by Condition (b) in $\cal{P}_{\tiny{\text{cv}},(X,\varepsilon)}$ and the convolution formula
\begin{align*}
p_{\XX}(u)= \int p_X(t)p_X(t-u)dt \leq \sup_{u\in\RR}p_X(u)\int p_X(t)dt = \sup_{u\in\RR}p_X(u).
\end{align*}
Therefore, Lemma \ref{lemma:Ubernstein} yields that
\begin{align*}
\P(|U_2 - \theta_2| \geq a_1v^{1/2} + a_2v + b_1u^{1/2} + b_2u + b_3u^{3/2} + b_4u^2) \leq C(\exp(-v) + \exp(-u)),
\end{align*}
where $a_1 \lesssim n^{-1/2}, a_2 \lesssim n^{-1}, b_1 \lesssim n^{-1}h^{-1/2}, b_2\lesssim n^{-1}, b_3 \lesssim n^{-3/2}h^{-1/2}, b_4 \lesssim n^{-2}h^{-1}$. Under the condition that $h \gtrsim n^{-(2-\delta)}$ for some $\delta > 0$ and $n$ is sufficiently large, the dominant terms in the above inequality are $a_1$ and $b_1$, that is,
\begin{align*}
n^{-1/2} \vee n^{-1}h^{-1/2}.
\end{align*}
This proves the first part of the theorem. The expectation version follows by Lemma \ref{lemma:tail_expectation}.
\end{proof}

\begin{lemma}
\label{lemma:tail_expectation}
Suppose a random variable $X$ satisfies the tail condition $\P(|X|\geq a_1t^{1/2} + a_2t + a_3t^{3/2} + a_4t^2) \leq C_1\exp(-C_2t)$ for any $t > 0$ and some positive constants $a_1,a_2,a_3,a_4, C_1,C_2$. Then, for any positive integer $p$, it holds that
\begin{align*}
\E(|X|^p)^{1/p} \leq C_3p^{1/p}(a_1 + a_2 + a_3 + a_4)
\end{align*}
for some positive constant $C_3$ that only depends on $C_1,C_2$.
\end{lemma}
\begin{proof}
We use $C_3$ to denote a positive constant that only depends on $C_1$ and $C_2$, which might have different values at each occurrence. The tail condition in the assumption is equivalent to
\begin{align*}
\P(|X| \geq t) \leq C_3\exp\bbrace*{-C_3\parr*{\frac{t^2}{a_1^2}\wedge \frac{t}{a_2} \wedge \frac{t^{2/3}}{a_3^{2/3}}\wedge \frac{t^{1/2}}{a_4^{1/2}}}}.
\end{align*} 
Let $I_1$-$I_4$ be a partition of $(0,+\infty)$ such that for $t\in I_1$, $t^2/a_1^2 = \min\bbrace*{\frac{t^2}{a_1^2}\wedge \frac{t}{a_2} \wedge \frac{t^{2/3}}{a_3^{2/3}}\wedge \frac{t^{1/2}}{a_4^{1/2}}}$, and similarly for $I_2,I_3, I_4$. Then, we have
\begin{align*}
&\ms\E(|X|^p)\\
&= \int_0^\infty pt^{p-1}\P(|X|\geq t)dt\\
&\leq C_3\big\{\int_{I_1} pt^{p-1}\exp\parr*{-C_3\frac{t^2}{a_1^2}}dt + \int_{I_2} pt^{p-1}\exp\parr*{-C_3\frac{t}{a_2}}dt + \int_{I_3} pt^{p-1}\exp\parr*{-C_3\frac{t^{2/3}}{a_3^{2/3}}}dt +\\
&\ms \int_{I_4} pt^{p-1}\exp\parr*{-C_3\frac{t^{1/2}}{a_4^{1/2}}}dt\big\}\\
&\leq C_3\big\{\int_0^\infty pt^{p-1}\exp\parr*{-C_3\frac{t^2}{a_1^2}}dt + \int_0^\infty pt^{p-1}\exp\parr*{-C_3\frac{t}{a_2}}dt + \int_0^\infty pt^{p-1}\exp\parr*{-C_3\frac{t^{2/3}}{a_3^{2/3}}}dt + \\
&\ms\int_0^\infty pt^{p-1}\exp\parr*{-C_3\frac{t^{1/2}}{a_4^{1/2}}}dt\big\}\\
&= C_3p\big\{a_1^p\int_0^\infty t^{p-1}\exp\parr*{-C_3t^2}dt + a_2^p\int_0^\infty t^{p-1}\exp\parr*{-C_3t}dt + a_3^p\int_0^\infty t^{p-1}\exp\parr*{-C_3t^{2/3}}dt + \\
&\ms a_4^p\int_0^\infty t^{p-1}\exp\parr*{-C_3t^{1/2}}dt\big\}\\
&\leq C_3p(a_1^p + a_2^p + a_3^p + a_4^p).
\end{align*}
This completes the proof.
\end{proof}

\begin{lemma}
\label{lemma:marginal_condition}
Recall the Condition (c) in the definition of $\cal{P}_{\tiny{\text{cv}},(X,\varepsilon)}$ in the main paper, and Condition $(c^\prime)$ in the subsequent paragraph. We have $(c^\prime)\Rightarrow (c)$.
\end{lemma}
\begin{proof}
Choose some $\delta_0 < 1/8$ and fix any $\delta \leq \delta_0$ and $u\in[-1,1]$. By the convolution formula, we have
\begin{align*}
p_{\XX}(u\delta) = \int_S p_X(s)p_X(s - u\delta)ds \geq c_0\int_S p_X(s - u\delta)ds.
\end{align*}
Therefore, it suffices to show that $\lambda\parr*{\bbrace*{S - u\delta}\bigcap[0,1]\bigcap S}$ is lower bounded by some fixed constant, say, $1/8$. Assume this does not hold, then we have
\begin{align*}
1 = \lambda([0,1])\geq \lambda\parr*{\bbrace*{S - u\delta}\bigcap[0,1]} + \lambda(S) - 1/8 \geq \lambda(S) - \delta + \lambda(S) - 1/8 \geq 5/4,
\end{align*}
which is a contradiction. 
\end{proof}

\begin{lemma}
\label{lemma:PD}
The covariance $\mbf{\Sigma}_1$ defined in the proof of Theorem \ref{thm:nonpar_lower} in the main paper is positive definite.
\end{lemma}
\begin{proof} 
We will prove that for any $\mbf{a}\in\RR^n$ with $\|\mbf{a}\| = 1$, it holds that $\mbf{a}^\top\mbf{\Sigma}_1 \mbf{a} > 0$. For each realization of $\{X_i\}_{i=1}^n$, partition the set $[n]$ into $L$ clusters for some positive integer $1\leq L \leq n$ such that the $Y_i$'s in each cluster fall into the same trapezoid in the lower bound construction. Denote these $L$ clusters as $\cal{M}_1,\ldots,\cal{M}_L$. Then, it follows that
\begin{align*}
\var\parr*{\sum_{i=1}^n a_iY_i\mid \{X_i\}_{i=1}^n} = \sum_{\ell=1}^L\var\parr*{\sum_{i\in\cal{M}_\ell} a_iY_i\mid\{X_i\}_{i=1}^n}.
\end{align*}
Since $\|\mbf{a}\| = 1$, there exists some $\ell_0\in[L]$ such that $A\define \sum_{i\in \cal{M}_{\ell_0}}a_i^2 > 0$. Partition $\{i:i\in \cal{M}_{\ell_0}\}$ according to the sign:
\begin{align*}
\cal{A}_+ \define \{i\in\cal{M}_{\ell_0}: a_i\geq 0\}, \quad \cal{A}_-\define\{i\in\cal{M}_{\ell_0}: a_i<0\}
\end{align*}
and define $S_+ \define \sum_{i\in\cal{A}_+} a_i$ and $A_+ \define \sum_{i\in\cal{A}_+} a_i^2$, and $S_-$ and $A_-$ similarly. Then, $A = A_+ + A_-$. Moreover, it holds that
\begin{align*}
&\ms\var\parr*{\sum_{i\in\cal{M}_{\ell_0}} a_iY_i\mid\{X_i\}_{i=1}^n}\\
&= (1+h_n^{2\alpha})\sum_{i\in\cal{M}_{\ell_0}}a_i^2 + h_n^{2\alpha}\parr*{\sum_{i,j\in \cal{A}_+;i\neq j}a_ia_j + \sum_{i,j\in \cal{A}_-;i\neq j}a_ia_j + 2\sum_{i\in\cal{A}_+,j\in \cal{A}_-}a_ia_j}\\
&= A(1+h_n^{2\alpha}) + h_n^{2\alpha}\parr*{S_+^2 - A_+ + S_-^2  - A_- + 2S_+S_-}\\
&=A + h_n^{2\alpha}\parr*{S_+ + S_-}^2 \geq A > 0.
\end{align*}
This completes the proof.
\end{proof}

\begin{lemma}[Theorem 3.3, \cite{gine2000exponential}]
\label{lemma:Ubernstein}
Let $Z_1,\ldots,Z_n,Z\in\cal{Z}$ be i.i.d., and $g:\cal{Z}^2\rightarrow \RR$ be a symmetric measurable function with $\E\bbrace*{g(Z_1,Z_2)} < \infty$. Write $U_n(g) \define \sum_{i<j}g(Z_i,Z_j)$ and $g_1(z) \define \E\bbrace*{g(Z,z)}$. Define
\begin{align*}
B_1 \define \sup_{Z_2}\E\bbrace*{\abs*{g(Z_1,Z_2)}\mid Z_2}, \quad B_2 \define \parr*{n\sup_{Z_2}\E\bbrace*{g^2(Z_1,Z_2)\mid Z_2}}^{1/2},
\quad B_3 \define \|g\|_\infty 
\end{align*}
and
\begin{align*}
\nu_1^2 \define \E\bbrace*{g_1^2(Z_2)}, \quad  \nu_2^2 \define \E\bbrace*{g^2(Z_1,Z_2)}.
\end{align*} 
Then, it holds that
\begin{align*}
&\ms\P\parr*{\abs*{U_n(g) - \E\bbrace*{U_n(g)}}\geq t + C_1n\nu_2 u^{1/2} + C_2nB_1u + C_3B_2u^{3/2} + C_4B_3u^2} \\
&\leq 2\exp\parr*{\frac{-t^2/n^2}{8n\nu_1^2 + 4B_1\cdot t/n}} + C_5e^{-u}, 
\end{align*}
where $C_1$-$C_5$ are absolute constants.
\end{lemma}


\section{Proofs of results in Section \ref{sec:v_function}}

\subsection{Proof of Theorem \ref{thm:v_function_upper}}
\begin{proof}
Throughout the proof, $C$ and $c$ will denote two generic positive constants that do not depend on $n$ and might have different values at each occurrence. We only prove the case for the pointwise error and the result for the integrated error will follow. Consider a fixed $x^*\in\supp(X)$. We will continue to use the notation $\ell,\mbf{q}(\cdot),\mathbf{B}_n, X_{ij},K_{ij}$ introduced in Section \ref{subsec:v_upper} in the main paper. We will drop the subscript in $\widehat{V}_{\tiny{\LP}}(x^*)$ for notational simplicity. Recall the choice of $(h_1,h_2)$ in \eqref{eq:v_h_choice} in the main paper.

Define $\mathbf{B} \define \E \mathbf{B}_n$ and the good event $\Omega_n \define \{\nm*{\mathbf{B}_n - \mathbf{B}}\leq 1/(2\|\mathbf{B}^{-1}\|)\}$. Note that $\Omega_n$ is well-defined as we now prove $\mathbf{B}$ is indeed invertible. For any $\mbf{a}\in\mathbb{S}^{\ell}$, it holds that
\begin{align*}
&\ms \mbf{a}^\top \mathbf{B}\mbf{a}\\
&= \int\int \bbrace*{\mbf{a}^\top \mbf{q}\parr*{\frac{(u+v)/2 - x^*}{h_2}}}^2 \frac{1}{h_1}K\parr*{\frac{u-v}{h_1}}\frac{1}{h_2}K\parr*{\frac{(u+v)/2 - x^*}{h_2}}p_X(u)p_X(v)dudv\\
&= \int\int \bbrace*{\mbf{a}^\top \mbf{q}(v)}^2 K(u)K(v)p_X(x^* + h_2v + h_1u/2)p_X(x^* + h_2v - h_1u/2)dudv\\
&= \int\int \bbrace*{\mbf{a}^\top \mbf{q}(v - h_1u/(2h_2))}^2K(u)K(v - h_1u/(2h_2))p_X(x^* + h_2v)p_X(x^* + h_2v - h_1u)dudv\\
&= \int_{-1}^1\int_{-1+h_1u/(2h_2)}^{1+h_1u/(2h_2)}\bbrace*{\mbf{a}^\top \mbf{q}\parr*{v - \frac{h_1u}{2h_2}}}^2K(u)K\parr*{v - \frac{h_1u}{2h_2}}p_X(x^* + h_2v)p_X(x^* + h_2v - h_1u)dvdu\\
&\geq \under{M}_K^2\int_{-1}^1\int_{-1+h_1u/(2h_2)}^{1+h_1u/(2h_2)}\bbrace*{\mbf{a}^\top \mbf{q}\parr*{v - \frac{h_1u}{2h_2}}}^2p_X(x^* + h_2v)p_X(x^* + h_2v - h_1u)dvdu,
\end{align*}
where in the last inequality we use the lower bound $\under{M}_K$ on $K(\cdot)$. Note that the first term of the integrand $\bbrace*{\mbf{a}^\top \mbf{q}\parr*{v - \frac{h_1u}{2h_2}}}^2$ is a polynomial of variables $u,v$ and thus only takes zero value with Lebesgue measure at most $0$. By the second part of Condition (c) in $\cal{P}_{\tiny{\text{vf}},(X,\varepsilon)}$ with $\delta = h_2$ (note that $h_2\leq \delta_0$ for any fixed $\delta_0 >0$ and sufficiently large $n$), for the given $x^*$, there exists a set $\cal{A}_{x^*}\subset [-1,1]$ with Lebesgue measure at least $c_0$ such that for all $v\in\cal{A}_{x^*}$, $x^* + h_2v\in\supp(X)$, and moreover, for each $v\in\cal{A}_{x^*}$, the second part of Condition (c) with $\delta = h_1$ again implies the existence of a set $\cal{A}_{x^*,v}\subset[-1,1]$ with Lebesgue measure at least $c_0$ such that for all $u\in\cal{A}_{x^*,v}$, it holds that $x^* + h_2v - h_1u\in\supp(X)$. Therefore, by the first part of Condition (c)  and the fact that $h_1/h_2\rightarrow 0$, there exists a fixed positive constant $c$ such that
\begin{align*}
\lambda_{\tiny{\min}}\define \inf_{\mbf{a}\in\mathbb{S}^{\ell}}\mbf{a}^\top \mathbf{B}\mbf{a}\geq c > 0.
\end{align*}
This concludes that $\Omega_n$ is well-defined. By triangle inequality, we have
\begin{equation}
\begin{aligned}
\label{eq:v_event_decom}
&\ms\E\parr*{\widehat{V}(x^*) - V(x^*)}^2\\
&\lesssim \E\parr*{\widehat{V}(x^*)\indc - V(x^*)}^2 +\E\parr*{\widehat{V}(x^*)\indcc}^2\\
&\lesssim \bbrace*{\E\parr*{(\widehat{V}(x^*) - V(x^*))\indc}}^2 + \E\parr*{\widehat{V}(x^*)\indc - \E\parr*{\widehat{V}(x^*)\indc}}^2 + \\
&\ms \E\parr*{\parr*{V^2(x^*) + \widehat{V}^2(x^*)}\indcc}.
\end{aligned}
\end{equation}
By Lemma \ref{lemma:v_bias}, we have for the first term
\begin{align*}
\bbrace*{\E\parr*{(\widehat{V}(x^*) - V(x^*))\indc}}^2 \leq C\parr*{h_1^{4(\alpha\wedge 1)} + h_2^{2\beta} + h_1^{2(\beta\wedge 1)} + \tau_n^2}.
\end{align*} 
By Lemma \ref{lemma:v_variance} with conditions $nh_2\rightarrow\infty$ and $n^2h_1h_2\rightarrow\infty$ satisfied with the choices of $(h_1,h_2)$ in \eqref{eq:v_h_choice}, we have for the second term
\begin{align*}
\E\parr*{\widehat{V}(x^*)\indc - \E\parr*{\widehat{V}(x^*)\indc}}^2 \leq C\parr*{n^{-1}h_2^{-1} + n^{-2}(h_1h_2)^{-1} + \tau_n^2}.
\end{align*}
Plugging in the values of $(h_1,h_2)$ as in \eqref{eq:v_h_choice} and choosing $\tau_n \asymp n^{-\kappa}$ for some fixed $\kappa \geq 1$, we obtain that
\begin{align*}
\E\parr*{\widehat{V}(x^*) - V(x^*)}^2 \leq C(n^{-\frac{8\alpha\beta}{4\alpha\beta + 2\alpha + \beta}} + n^{-\frac{2\beta}{2\beta + 1}}) + \E\parr*{\parr*{V^2(x^*) + \widehat{V}^2(x^*)}\indcc}.
\end{align*}
Lastly, note that $\widehat{V}(x^*) = \sum_{i<j}D_{ij}w_{ij} / (\sum_{i<j}w_{ij} + \tau_n)$ is a linear estimator with weight $w_{ij} = {n\choose 2}^{-1}\mbf{q}^\top(0)\mathbf{B}_n^*\mbf{q}((X_{ij}-x^*)/h_2)K_{ij}$. By definition of $\mathbf{B}_n^*$, $w_{ij}$ is thus a weighted polynomial of $K_{ij}(X_{ij}-x^*)/h_2$ up to some order that only depends on $\ell$. Therefore, in view of the choice of $\tau_n$ (decaying to $0$ polynomially with $n$), $h_1,h_2$, there exists some sufficiently large constant $\eta$ (only depending on $\alpha,\beta,\kappa$) such that
\begin{align*}
\E\parr*{\E\parr*{\widehat{V}^2(x^*)\mid \{X_i\}_{i=1}^n}}^2 \lesssim  n^{\eta},
\end{align*}
and thus by Cauchy-Schwarz and the exponential inequality in Lemma \ref{lemma:B_concentration}, it holds that
\begin{align*}
\E\parr*{\parr*{V^2(x^*) + \widehat{V}^2(x^*)}\indcc} &= \E\parr*{(\E \widehat{V}^2(x^*)\mid \{X_i\}_{i=1}^n + V^2(x^*)) \indcc}\\
&\lesssim (\E(\E\widehat{V}^2(x^*)\mid\{X_i\}_{i=1}^n)^2 + V^4(x^*))^{1/2}\P^{1/2}(\Omega_n^c) \\
&= o\parr*{n^{-\frac{8\alpha\beta}{4\alpha\beta + 2\alpha + \beta}} + n^{-\frac{2\beta}{2\beta + 1}}}.
\end{align*}
This completes the proof.
\end{proof}

\subsection{Proof of Theorem \ref{thm:v_lower_point}}
\begin{proof}
Note that the boundary of $n^{-8\alpha\beta/(4\alpha\beta + \beta+ 2\alpha)}$ and $n^{-2\beta/(2\beta + 1)}$ lies at $\alpha = \beta/(4\beta + 2)$. When $\alpha\geq \beta/(4\beta + 2)$, the statement can be proved using a slight variation of the proof of Theorem 4.2 in \cite{brown2007variance}, and we omit the details here. Next, we will focus on the case where $\alpha < \beta/(4\beta + 2)$. Consider a fixed point $x^*\in\supp(X)$. Throughout the proof, $C$ and $c$ represent two generic positive constants which only depend on $\alpha,\beta,C_{\cal{F}},C_{\cal{V}}, C_\sigma,C_0,c_0,C_\varepsilon$ 
and might have different values at each occurrence, but like in the proof of Theorem \ref{thm:nonpar_lower}, let c be always smaller than 1/4. Also, without loss of generality, assume that the sample size $n$ and $C_{\cal{F}},C_{\cal{V}}, C_\sigma,C_\varepsilon,C_0$ are sufficiently large, $c_0$ is sufficiently small, and $[0,1]\subset I$.

We will make use of Le Cam's two point method. Introduce the constants
\begin{align}
\label{eq:v_constant}
\theta_n^2 \define h_1^{2\alpha} \define h_2^\beta \define cn^{-\frac{4\alpha\beta}{4\alpha\beta + \beta + 2\alpha}}, \quad M\define h_2/(4h_1) - 1/2, \quad N \define 2M + 1 = h_2/(2h_1),
\end{align}
where we tune the constant $c$ in $\theta_n^2$ so that $M$ is a positive integer. Note that under the above choice, $h_2/h_1\rightarrow\infty$ as $n\rightarrow\infty$. We now specify $f(\cdot), V(\cdot)$, distribution of $X$ and distribution of $\varepsilon$ in the null and alternative hypotheses, $H_0$ and $H_1$, respectively.

\begin{itemize}
\item[] \emph{Choice of $\varepsilon$}: Under both $H_0$ and $H_1$, let $\varepsilon\sim\cal{N}(0,1)$.
\item[] \emph{Choice of $V(\cdot)$}: Under $H_0$, let $V\equiv 1$. Under $H_1$, let $V = 1 - \theta_n^2H((x - x^*)/h_2)$, where $H(\cdot)$ is $\beta$-H\"older smooth, infinitely differentiable, compactly supported on $[-2,2]$, and takes value $1$ on $[-1,1]$.
\item[] \emph{Choice of $f(\cdot)$}: Under $H_0$, let $f\equiv 0$. Under $H_1$, let $f$ be zero outside $[x^* - h_2, x^* + h_2]$, and inside this interval, the linear interpolant of the function that takes value $r_i$ on $[x^* - h_2 + (4i-3)h_1, x^* - h_2 +  (4i-1)h_1]$ and zero at $x^* - h_2 + 4(i-1)h_1$ for all $i\in[N]$, where $\{r_i\}_{i=1}^N$ is an i.i.d. sequence of symmetric and compactly supported random variables with distribution $\G$ satisfying
\begin{align*}
\int_{-\infty}^\infty x^j\G(dx) = \int_{-\infty}^\infty x^j\varphi(x)dx, \quad j=1,\ldots,q,
\end{align*} 
where $q$ is some fixed odd integer strictly larger than $1 + (\beta+2\alpha)/(2\alpha\beta)$.
\item[] \emph{Choice of $X$}: Under both $H_0$ and $H_1$, let $X$ be uniformly distributed on the union of the intervals 
\begin{align*}
[0, 1] \bigcap \parr*{[0,x^* - 2h_2] \bigcup [x^* + 2h_2, 1] \bigcup_{i=1}^N [x^* - h_2 + (4i-3)h_1, x^* - h_2 +  (4i-1)h_1]}.
\end{align*}
\end{itemize}
See Figure \ref{fig:v_point_lower} for an illustration. We now make a few remarks about the above construction. For the design of $V(\cdot)$ under $H_1$, one example of the smooth bump function $H(\cdot)$ is $(\mathbbm{1}_{[-3/2,3/2]}\ast \varphi_{1/2})(\cdot)$, where $\varphi_\varepsilon(x)\define \varphi(x/\varepsilon)/\varepsilon$ with $\varphi(x)\define \exp(-1/(1-x^2))\mathbbm{1}\{|x|\leq 1\}$ being a smooth and compactly supported mollifier. The design of $f(\cdot)$ under $H_1$ is a ``localized" version of $f(\cdot)$ in the proof of Theorem \ref{thm:nonpar_lower}. The existence of $\{r_i\}_{i=1}^N$ is again guaranteed by Lemma \ref{lemma:moment_matching}, and their range, which we denote as $B$, only depends on $\alpha$ and $\beta$ and is thus fixed. Lastly, we indeed have $x^*\in\supp(X)$ since it is in the $(M+1)$th interval in the $N$ intervals specified in the support of $X$. Moreover, under $H_1$, conditioning on the event that $X_i \in [x^*-h_2, x^* + h_2]$ and any realization of $\{r_i\}_{i=1}^{N}$, $f(X_i)$ is uniformly distributed over $\{h_1^{\alpha}r_1,\ldots,h_1^{\alpha}r_N\}$.

Clearly, under both the null and the alternative hypotheses, $V(\cdot)$ is $\beta$-H\"older smooth, and under $H_1$, $f(\cdot)$ is $\alpha$-H\"older smooth for each realization of $\{r_i\}_{i=1}^N$ due to their compact support. Next, we show that the joint distribution of $(X,\varepsilon)$ satisfies the three conditions in $\cal{P}_{\tiny{\text{vf}},(X,\varepsilon)}$. Condition (d) clearly holds and Condition (a) holds with $I=[0,1]$. Condition (b) holds as well since for any $u$ in the support of $X$, $p_X(u) = 1/(1-3h_2)$ for $x^*\in(0,1)$ and $p_X(u) = 2/(2-3h_2)$ for $x^*\in\{0,1\}$, both of which are smaller than $2$ for sufficiently large $n$. Lastly, for Condition (c), the first part clearly holds since $\inf_{u\in\supp(X)} p_X(u) \geq 1$. For the second part, define $\cal{A}_{x,\delta} \define \{u\in[-1,1]: x+\delta u\in\supp(X)\}$. Then, for any $x^*\in(0,1)$ and any $0 < \delta \leq 1/2$, we have $\lambda(\cal{A}_{x,\delta}) \geq 1/2$ if $x\in(0,x^* - 2h_2]\bigcup[x^* + 2h_2, 1)$ and $\lambda(\cal{A}_{x,\delta}) \geq 1/4$ if $x\in \bigcup_{i=1}^N [x^* - h_2 + (4i-3)h_1, x^* - h_2 +  (4i-1)h_1]$. A similar statement holds for $x^*\in\{0,1\}$. We therefore conclude that Condition (c) also holds.

Denote $\P_0$ and $\P_1$ to be the joint distributions of $\{X_i,Y_i\}_{i=1}^n$ under $H_0$ and $H_1$, then the pointwise squared distance between $\P_0$ and $\P_1$ $\parr*{V_0(x^*) - V_1(x^*)}^2 \asymp \theta_n^4$ is the desired minimax rate. Further define $\td{\P}_1$ as the corresponding joint distributions of $\{X_i,Y_i\}_{i=1}^n$ under $H_1$ with $\{r_i\}_{i=1}^N$ replaced by an i.i.d. standard normal sequence $\{\td{r}_i\}_{i=1}^N$. Then, following the same line of proof of Theorem \ref{thm:nonpar_lower}, it suffices to show that $\TV(\P_0, \td{\P}_1) \leq c$ and $\TV(\P_1, \td{\P}_1) \leq c$. 

For the first inequality, in view of \eqref{eq:tv1_1} in the proof of Theorem \ref{thm:nonpar_lower}, it suffices to upper bound $\TV(\P_0(\mbf{y}\mid\mbf{x}), \td{\P}_1(\mbf{y}\mid \mbf{x}))$ for each realization $\{x_i\}_{i=1}^n$.
Note that under $\P_0$, $\mbf{y}\mid \mbf{x}\sim\cal{N}_n(0, \mbf{\Sigma_0})$, with $\mbf{\Sigma_0} = \mathbf{I}_n$. Denote $\{b_i\}_{i=1}^n$ as the location index sequence of $\{X_i\}_{i=1}^n$ taking values in $\{0,1,\ldots, N\}$, that is, $b_i = 0$ if $X_i\notin [x^* - h_2, x^* + h_2]$ and $b_i = j$ if $X_i\in[x^* - h_2 + (4j-3)h_1, x^* - h_2 +  (4j-1)h_1]$ for $j\in[N]$. Then, due to the symmetry of $\{r_i\}_{i=1}^N$, the design of the nonparametric component $f$, and the fact that $H(\cdot)$ takes value $1$ on $[-1, 1]$, it holds that under $\td{\P}_1$, $\mbf{y}\mid \mbf{x}\sim \cal{N}_n(0, \mbf{\Sigma_1})$, with 
\begin{align*}
(\mbf{\Sigma_1})_{ii} = 1 - \theta_n^2\mathbbm{1}\bbrace*{\abs*{X_i-x^*}\leq h_2} +  h_1^{2\alpha}\mathbbm{1}\bbrace*{\abs*{X_i-x^*}\leq h_2} = 1
\end{align*}
and $(\mbf{\Sigma_1})_{ij} = h_1^{2\alpha}\mathbbm{1}\{b_i = b_j, b_i \geq 1, b_j\geq 1\}$ for $i\neq j$. Define $N_0\define \sum_{i\neq j}\mathbbm{1}\{b_i = b_j, b_i \geq 1, b_j\geq 1\}$. Then, we have by Lemma \ref{lemma:gaussian_tv} that
\begin{align*}
\TV(\P_0(\mbf{y}\mid\mbf{x}), \td{\P}_1(\mbf{y}\mid \mbf{x})) \leq C\parr*{h_1^{4\alpha}N_0}^{1/2} = C\theta_n^2N_0^{1/2}.
\end{align*}
Note that $N_0$ is a random variable that depends on $\{X_i\}_{i=1}^n$, and by \eqref{eq:tv1_1} in the proof of Theorem \ref{thm:nonpar_lower},
\begin{align*}
\TV(\P_0, \td{\P}_1) \leq C\theta_n^2\E(N_0^{1/2}) \leq C\theta_n^2(\E N_0)^{1/2}.
\end{align*}
Since direct calculation implies that $\E(N_0) \leq Cn^2Nh_1^2 = Cn^2h_1h_2$, we have $\TV(\P_0, \td{\P}_1) \lesssim 1$ under the given choice of $h_1,h_2$ and $\theta_n$. 

Using a conditioning argument, the second part of proving $\TV(\td{\P}_1, \P_1) \lesssim 1$ follows similarly from that of Theorem \ref{thm:nonpar_lower} by noting that $n^2h_1 \rightarrow \infty$ and $nh_1\rightarrow 0$ as $n\rightarrow\infty$ under the constraint $\alpha < \beta/(4\beta+2)$ so that Lemma \ref{lemma:multinomial} can be similarly applied. The proof is complete. 
\end{proof}

\subsection{Proof of Theorem \ref{thm:v_lower_int}}
\begin{proof}
As in the proof of Theorem \ref{thm:v_lower_point}, we focus on the regime $\alpha < \beta/(4\beta + 2)$. We will couple the proof of Theorem \ref{thm:v_lower_point} with a standard technique via multiple hypotheses in the classic setting of mean function estimation. 

Introduce the following notation:
\begin{align*}
&\theta_n^2 \define h_1^{2\alpha} \define h_2^\beta \define  cn^{-\frac{4\alpha\beta}{4\alpha\beta + \beta + 2\alpha}}, \quad N_2 \define 1/(4h_2), \quad N_1 \define h_2/(2h_1),\\
&\text{and } \quad x^*_i \define 2h_2 + (i-1)4h_2, \quad i\in[N_2],
\end{align*}
where we tune the constant $c$ in $\theta_n^2$ so that $N_1$ and $N_2$ are both positive integers. Note that under the above choice, $h_2/h_1\rightarrow\infty$ as $n\rightarrow\infty$. By the renowned Varshamov-Gilbert bound (cf. Lemma 2.8 in \cite{tsybakov2009introduction}), there exists a set of length-$N_2$ binary sequences $\{\mbf{\Delta}_j\}_{j=0}^{M}$ with $M \geq 2^{N_2/8}$ such that $\mbf{\Delta}_0 = \mbf{0}_{N_2}$ and for any $0 \leq k < \ell \leq M$, it holds that $\rho(\mbf{\Delta}_k, \mbf{\Delta}_\ell) \geq N_2/8$, where $\rho$ is the Hamming distance. We now choose a number of $M + 1$ hypotheses with $\{\mbf{\Delta}_j\}_{j=0}^{M}$ satisfying the above property, which we denote as $\P_0,\P_1,\ldots,\P_M$. We now specify $f(\cdot),V(\cdot)$, distribution of $X$ and distribution of $\varepsilon$ under each hypothesis. 

\begin{itemize}
\item[] \emph{Choice of $\varepsilon$}: Under $\P_0$ and $\P_j$ for all $j\in[M]$, let $\varepsilon\sim\cal{N}(0,1)$.
\item[] \emph{Choice of $V(\cdot)$}: Under $\P_0$, let $V_0\equiv 1$. Under $\P_j$, let $V_j(x) \define 1 - \sum_{i=1}^{N_2}\Delta_{j,i}\theta_n^2H\parr*{(x-x^*_i)/h_2}$, where $H(\cdot)$ is infinitely differentiable, compactly supported on $[-2,2]$ and takes value $1$ on $[-1,1]$.
\item[] \emph{Choice of $f(\cdot)$}: Under $\P_0$, let $f_0\equiv 0$. Under $\P_j$, for all $i\in[N_2]$ such that $\Delta_{j,i} = 1$, let $f$ be the linear interpolation of the function that takes value $r^{(j)}_{i,k}$ on the interval $[x^*_i - h_2 + (4k-3)h_1, x^*_i - h_2 + (4k-1)h_1]$ and value zero at $x^*_i - h_2 + 4(k-1)h_1$ for all $k\in[N_1]$, where by denoting $m_j\define \|\Delta_j\|_0$, $\{r^{(j)}_{i,k}\}_{j\in[M], i\in[m_j], k\in[N_1]}$ is an i.i.d. sequence of symmetric and compactly supported random variables with distribution $\G$ satisfying
\begin{align*}
\int_{-\infty}^\infty x^j\G(dx) = \int_{-\infty}^\infty x^j\varphi(x)dx, \quad j=1,\ldots,q,
\end{align*}
where $q$ is some fixed odd integer that only depends on $\alpha$ and $\beta$.
\item[] \emph{Choice of $X$}: Under $\P_0$ and $\P_j$ for all $j\in[M]$, let $X$ be uniformly distributed on the union of the disjoint intervals 
\begin{align*}
\bigcup_{i=1}^{N_2}\bigcup_{k=1}^{N_1} [x^*_i - h_2 + (4k-3)h_1, x^*_i - h_2 + (4k-1)h_1].
\end{align*}
\end{itemize}
The existence of $H(\cdot)$ in the design of $V(\cdot)$ and variables $\{r^{(j)}_{i,k}\}_{j\in[M], i\in[m_j], k\in[N_1]}$ is as argued in the proof of Theorem \ref{thm:v_lower_point}. Moreover, one can readily check that for each $0\leq k<\ell \leq M$, the integrated squared distance between each $\P_k$ and $\P_\ell$ satisfies
\begin{align*}
d(\P_k, \P_\ell)\define \int \parr*{V_k(x) - V_\ell(x)}^2p_X(x)dx \gtrsim h_2^{2\beta} \asymp n^{-\frac{8\alpha\beta}{4\alpha\beta + \beta + 2\alpha}},
\end{align*}
which is the desired lower bound.

Clearly, under each $\P_j$, $0\leq j\leq M$ and for each realization of $\{r^{(j)}_{i,k}\}_{j\in[M], i\in[m_j], k\in[N_1]}$, $f_j(\cdot)$ and $V_j(\cdot)$ are $\alpha$- and $\beta$-H\"older smooth, respectively, due to the compact support of $\{r^{(j)}_{i,k}\}_{j\in[M], i\in[m_j], k\in[N_1]}$. Moreover, the joint distribution of $(X,\varepsilon)$ (same in all hypothese) satisfies the conditions in $\cal{P}_{\tiny{\text{vf}},(X,\varepsilon)}$ with a similar argument as in Theorem \ref{thm:v_lower_point}.

We now proceed with the proof. Note that under the above design, the support of $X$ is segmented into $N_3\define N_1\times N_2$ intervals, and we let $\{b_i\}_{i=1}^n$ be the location index of $\{X_i\}_{i=1}^n$, taking values in $[N_2]\times [N_1]$, that is, $b_i = (k,\ell)$ if $X_i\in[x^*_k - h_2 + (4\ell-3)h_1, x^*_k - h_2 + (4\ell-1) h_1]$. As in the proof of Theorem \ref{thm:nonpar_lower}, define the event $\Omega_n \define \big\{\max_{(k,\ell)\in\P[N_2]\times[N_1]}\#\bbrace*{b_i = (k,\ell)}\leq K\big\}$, where $K$ is the smallest integer strictly larger than $2\beta/(\beta - 4\alpha\beta - 2\alpha)$. Then, by Lemma \ref{lemma:multinomial}, it holds that $\Omega_n$ has asymptotic probability $1$ under all of $\P_j$ and $\td{\P}_j$ for $0\leq j\leq M$. Now, by a standard reduction scheme with multiple hypotheses (cf. Chapter 2.2 in \cite{tsybakov2009introduction}) and Lemma \ref{lemma:cond_kull}, it suffices to show that
\begin{align}
\label{eq:lower_multiple}
\frac{1}{M}\sum_{j=1}^M K(\P_j, \P_0; \Omega_n) \leq c\log(M)
\end{align}
for some $0 < c < 1/8$, where $K(\P, \Q; \Omega_n)$ is the ``conditional" Kullback divergence between probability measures $\P$ and $\Q$ defined as $K(\P, \Q; \cal{E}) \define \int_{\cal{E}}\log(d\P/d\Q)d\P$ for any measurable set $\cal{E}$. In order to show \eqref{eq:lower_multiple}, it further suffices to show that $K(\P_j, \P_0; \Omega_n) \leq \log(M)$ for all $j\in[M]$. We now focus on a particular $j\in[M]$. For notational brevity, we will drop the superscript $(j)$ in the sequence of variables $\{r^{(j)}_{i,k}\}_{i\in[m_j], k\in[N_1]}$ for this particular $j$. Note that $N_2/8 \leq m_j \leq N_2$ by the property that $\rho(\mbf{\Delta}_0, \mbf{\Delta}_j)\geq N_2/8$. Moreover, by the design of $f$, there are a total of $m_jN_1$ trapezoids in the union of the intervals $[x^*_i - h_2, x^*_i + h_2]$ for those $i$ such that $\Delta_{j,i} = 1$. Define $\td{\P}_j$ as the joint distribution of $\{(X_i,Y_i)\}_{i=1}^n$ under $\P_j$ but with $\{r_{i,k}\}_{i\in[m_j], k\in[N_1]}$ replaced by a sequence of i.i.d. standard normal variables denoted as $\{\td{r}_{i,k}\}_{i\in[m_j], k\in[N_1]}$. By definition, it holds that 
\begin{align*}
K(\P_j, \P_0; \Omega_n) &= \int_{\Omega_n} p_j\log\frac{p_j}{p_0}\\
& = \int_{\Omega_n} p_j\log\frac{p_j}{\td{p}_j} + \int_{\Omega_n} \td{p}_j\log\frac{\td{p}_j}{p_0} + \int_{\Omega_n} (p_j - \td{p}_j) \log\frac{\td{p}_j}{p_0}\\
&= K(\P_j, \td{\P}_j;\Omega_n) + K(\td{\P}_j, \P_0; \Omega_n) + \int_{\Omega_n} (p_j - \td{p}_j) \log\frac{\td{p}_j}{p_0}
\end{align*}
for density functions with respect to some common dominating measure. Next, we will show respectively that, by matching the moments of $\{r_{i,k}\}_{i\in[m_j],k\in[N_1]}$ and the standard Gaussian random variable up to some sufficiently high order, it holds that
\begin{align*}
K(\P_j, \td{\P}_j; \Omega_n)\lesssim 1,\quad K(\td{\P}_j, \P_0;\Omega_n)\lesssim \log(M),\quad \text{ and }\quad \int_{\Omega_n} (p_j - \td{p}_j) \log\parr*{\td{p}_j/p_0}\lesssim 1.
\end{align*}

First note that, by denoting $\mbf{x} \define (x_1,\ldots,x_n)$, $d\mbf{x} \define dx_1\ldots dx_n$ and similarly for $\mbf{y}$ and $d\mbf{y}$, we have
\begin{align*}
K(\P_j, \td{\P}_j; \Omega_n) &= \int \mathbbm{1}\bbrace*{\Omega_n}p(\mbf{x})d\mbf{x}\int \log\parr*{\frac{d\P_j(\mbf{y}\mid \mbf{x})}{d\td{\P}_j(\mbf{y}\mid \mbf{x})}}\P_j(d\mbf{y}\mid \mbf{x})\\
&= \E\bbrace*{\mathbbm{1}\bbrace*{\Omega_n}K\parr*{\P_j(\mbf{y}\mid \mbf{x}), \td{\P}_j(\mbf{y}\mid \mbf{x})}}\\
&\leq \E\bbrace*{\mathbbm{1}\bbrace*{\Omega_n}\chi^2\parr*{\P_j(\mbf{y}\mid \mbf{x}), \td{\P}_j(\mbf{y}\mid \mbf{x})}}\\
&\define \chi^2\parr*{\P_j, \td{\P}_j; \Omega_n},
\end{align*}
where the inequality follows by Lemma 2.7 in \cite{tsybakov2009introduction}. Therefore the first inequality $K(\P_j, \td{\P}_j; \Omega_n)\lesssim 1$ holds by Lemma \ref{lemma:chi_match}. 

Next we prove $K(\td{\P}_j, \P_0;\Omega_n)\lesssim \log(M)$. Again, it suffices to prove that for any realization of $\{X_i\}_{i=1}^n$ in $\Omega_n$, it holds that $K(\td{\P}_j(\mbf{y}\mid\mbf{x}), \P_0(\mbf{y}\mid\mbf{x})) \lesssim \log(M) \asymp N_2$. Note that under $\P_0$, $\mbf{y}\mid \mbf{x}\sim\cal{N}_n(0, \mbf{\Sigma_0})$, with $\mbf{\Sigma_0} = \mathbf{I}_n$. Recall that the location index sequence $\{b_i\}_{i=1}^n = \{(k_i,\ell_i)\}_{i=1}^n$ takes the value $(k_i,\ell_i) = (k,\ell)$ if $X_i\in[x^*_k - h_2 + (4\ell-3)h_1, x^*_k - h_2 + (4\ell-1) h_1]$. Then, due to the symmetry of $\{r_{i,k}\}_{i\in[m_j],k\in[N_1]}$, the design of the nonparametric component $f$, and the fact that $K(\cdot)$ takes value $1$ on $[-1, 1]$, it holds that under $\td{\P}_j$, $\mbf{y}\mid \mbf{x}\sim \cal{N}_n(0, \mbf{\Sigma_1})$, with 
\begin{align*}
(\mbf{\Sigma_1})_{ii} = 1 - \theta_n^2\mathbbm{1}\bbrace*{\Delta_{j,k_i} = 1} +  h_1^{2\alpha}\mathbbm{1}\bbrace*{\Delta_{j,k_i} = 1} = 1
\end{align*}
and $(\mbf{\Sigma_1})_{i_1i_2} = h_1^{2\alpha}\mathbbm{1}\bbrace*{\Delta_{j,k_{i_1}} = 1, (k_{i_1},\ell_{i_1}) = (k_{i_2},\ell_{i_2})}$ for $i_1\neq i_2$. Define 
\begin{align*}
N_0\define \sum_{i_1\neq i_2}\mathbbm{1}\bbrace*{\Delta_{j,k_{i_1}} = 1, (k_{i_1},\ell_{i_1}) = (k_{i_2},\ell_{i_2})}.
\end{align*}
Then, by the proof of Lemma 3.6 in \cite{gao2016rate}, it holds that
\begin{align*}
K(\td{\P}_j(\mbf{y}\mid\mbf{x}), \P_0(\mbf{y}\mid \mbf{x})) \leq Ch_1^{4\alpha}N_0 = C\theta_n^4N_0.
\end{align*}
Note that $N_0$ is a random variable that depends on $\{X_i\}_{i=1}^n$, and by direct calculation we have 
\begin{align*}
\E(N_0) \leq n^2m_jh_2h_1 \asymp n^2N_2h_1h_2.
\end{align*}
Putting together the pieces, we have $K(\td{\P}_j, \P_0; \Omega_n) \leq \theta_n^4n^2h_1h_2N_2 \lesssim N_2$. This completes the proof of the second inequality.

Lastly, we show that $\int_{\Omega_n} (p_j - \td{p}_j) \log\parr*{\td{p}_j/p_0}\lesssim 1$. First note that
\begin{align*}
&\ms\int_{\Omega_n} (p_j - \td{p}_j) \log\parr*{\td{p}_j/p_0}\\
&\leq \int_{\Omega_n} \abs*{p_j - \td{p}_j}\abs*{\log(\td{p}_j/p_0)}\\
&\leq \parr*{\int_{\Omega_n} \abs*{p_j - \td{p}_j}}^{1/2}\parr*{\int \abs*{p_j - \td{p}_j} \log^2(\td{p}_j/p_0)}^{1/2}\\
&\leq \parr*{\int_{\Omega_n} \abs*{p_j - \td{p}_j}}^{1/2}\bbrace*{\parr*{\int p_j \log^2(\td{p}_j/p_0)}^{1/2} + \parr*{\int \td{p}_j \log^2(\td{p}_j/p_0)}^{1/2}}.
\end{align*}
By Lemmas \ref{lemma:cond_pinsker} and \ref{lemma:chi_match}, by matching moments up to some sufficiently high order, the first term above can be upper bounded (up to some constant) by $n^{-\eta}$ for any $\eta > 0$, therefore it suffices to show that both $\int p_j \log^2(\td{p}_j/p_0)$ and $\int \td{p}_j \log^2(\td{p}_j/p_0)$ can be upper bounded by some polynomial of $n$ of fixed order. Consider any realization of $\{X_i\}_{i=1}^n$ in $\Omega_n$, and assume that based on their location indices $\{b_i\}_{i=1}^n$, the $n$ data points are partitioned into $L_1 + L_2$ clusters with cardinality $s_\ell$ such that the $X_i$'s in the same cluster have the same value $b_i$. Moreover, for each data point in the first $L_1$ clusters, the location index $b_i = (k_i,\ell_i)$ satisfies that $\Delta_{j,k_i} = 1$ while for the data points in the last $L_2$ clusters, it holds that $\Delta_{j,k_i} = 0$. Apparently, we have the relations $1\leq L_1 + L_2 \leq n$, $\sum_{\ell=1}^{L_1+L_2}s_\ell = n$ and $1\leq s_\ell \leq K$ for $\ell\in[L_1 + L_2]$. Moreover, denoting $\td{\P}_{j,\pi_\ell}$ and $\P_{0,\pi_\ell}$ (resp. $\td{p}_{j,\pi_\ell}$ and $p_{0,\pi_\ell}$) for each $\ell\in[L_1 + L_2]$ as the joint distribution (resp. density) of those $Y_i$'s in the $\ell$th cluster conditioning on the given realization $\{X_i\}_{i=1}^n$ under $\td{\P}_j$ and $\P_0$, we have 
\begin{align*}
\td{p}_j = \prod_{\ell=1}^{L_1 + L_2} \td{p}_{j,\pi_\ell} \quad \text{ and }\quad  p_0 = \prod_{\ell=1}^{L_1 + L_2} p_{0,\pi_\ell}.
\end{align*}
Moreover, for any $L_1 + 1\leq \ell \leq L_1 + L_2$, it holds that $\td{p}_{j,\pi_\ell} = p_{0,\pi_\ell}$, therefore it holds that
\begin{align*}
\log^2\parr*{\frac{\td{p}_j}{p_0}} = \parr*{\sum_{\ell=1}^{L_1} \log(\td{p}_{j,\pi_\ell}) - \log\parr*{p_{0,\pi_\ell}}}^2 \lesssim n\sum_{\ell=1}^{L_1}\log^2(\td{p}_{j,\pi_\ell}/p_{0,\pi_\ell}).
\end{align*}
Now consider any $\ell\in[L_1]$ and assume that $s_\ell = d$ for some positive integer $d$. Without loss of generality, assume the $y_i$'s in this cluster are $\{y_1,\ldots,y_d\}$, and they take the form $Y_i = h_1^\alpha \td{r}_{1,1} + (1-h_2^\beta)^{1/2}\varepsilon_i = \theta_n\td{r}_{1,1} + (1-h_2^\beta)^{1/2}\varepsilon_i$ under $\td{\P}_j$ and $Y_i = \varepsilon_i$ under $\P_0$. Define $\sigma^2 \define (1 - h_2^\beta)$ which is positive for large enough $n$. Then, the previous equalities imply that
\begin{align*}
p_{0,\pi_\ell} = \varphi(y_1)\ldots\varphi(y_d) = (2\pi)^{-d/2}\exp\parr*{-\frac{\sum_{i=1}^d y_i^2}{2}}
\end{align*}
and
\begin{align*}
\td{p}_{j,\pi_\ell} &= \int \frac{1}{\sigma}\varphi\parr*{\frac{y_1-\theta_nv}{\sigma}}\ldots\frac{1}{\sigma}\varphi\parr*{\frac{y_{d}-\theta_nv}{\sigma}}\varphi(v)dv\\
&= (2\pi)^{-d/2}\frac{1}{\sigma^{d-1}(d\theta_n^2 + \sigma^2)^{1/2}}\exp\parr*{-\frac{\sum_{i=1}^d y_i^2}{2\sigma^2} + \frac{(\sum_{i=1}^d y_i\theta_n)^2}{2\sigma^2(d\theta_n^2 + \sigma^2)^2}}.
\end{align*} 
Putting together the pieces, we obtain that
\begin{align*}
\log^2(\td{p}_{j,\pi_\ell}/p_{0,\pi_\ell}) \lesssim d^2\log^2(1/\sigma) + \parr*{\sum_{i=1}^d y_i^2}^2 + \parr*{\sum_{i=1}^d y_i\theta_n}^4 \lesssim 1 + \sum_{i=1}^d y_i^4.
\end{align*}
Therefore we have
\begin{align*}
\int p_j \log^2(\td{p}_j/p_0) \lesssim n\sum_{\ell=1}^{L_1} \int p_j(1 + \sum_{i=1}^d y_i^4) \lesssim n\sum_{\ell=1}^{L_1}\sum_{i=1}^d \int y_i^4\P_j(dy_i) \lesssim n^2,
\end{align*}
where we use the fact that $L_1 \leq n$. Similarly, we have $\int \td{p}_j \log^2(\td{p}_j/p_0)\lesssim n^2$. The proof is thus complete.
\end{proof}

\subsection{Supporting lemmas}

\begin{lemma}
\label{lemma:v_bias}
Suppose $f\in\Lambda_{\alpha}(C_{\cal{F}})$, $V\in\Lambda_{\beta}(C_{\cal{V}})$, $\sigma^2\leq C_\sigma$ for some fixed constants $C_{\cal{F}},C_{\cal{V}},C_\sigma$, and the joint distribution of $(X,\varepsilon)$ belongs to $\cal{P}_{\tiny{\text{vf}},(X,\varepsilon)}$. Then, with $\Omega_n$ defined in the proof of Theorem \ref{thm:v_function_upper}, it holds that
\begin{align*}
\abs*{\E\bbrace*{\parr*{\widehat{V}(x^*) - V(x^*)}\indc}} \leq C\parr*{h_1^{2(\alpha\wedge 1)} + h_2^{\beta} + h_1^{\beta\wedge 1} + \tau_n}
\end{align*}
for some fixed positive constant $C$ that only depends on $\alpha,\beta,C_{\cal{F}},C_{\cal{V}},C_\sigma,C_0,C_\varepsilon$.
\end{lemma}

\begin{proof}
We adopt the notation $\ell,\mbf{q}(\cdot), \mathbf{B}_n, X_{ij}$, and $K_{ij}$ from the proof of Theorem \ref{thm:v_function_upper}. Also recall the definition of $\mathbf{B}_n^*$, $D_{ij}$, $w_{ij}$, and $\td{w}_{ij}$ from the definition of $\widehat{V}(x^*)$. Writing $\E_\varepsilon$ as the conditional expectation given $\{X_i\}_{i=1}^n$, it holds that
\begin{align*}
\abs*{\E\bbrace*{\parr*{\widehat{V}(x^*) - V(x^*)}\indc}} = \abs*{\E\bbrace*{\indc\E_\varepsilon\parr*{\widehat{V}(x^*) - V(x^*)}}}.
\end{align*}
Then $\widehat{V}(x^*) = \sum_{i<j}\td{w}_{ij}D_{ij}$ and
\begin{align*}
&\ms\abs*{\E\bbrace*{\indc\E_\varepsilon\parr*{\widehat{V}(x^*) - V(x^*)}}}\\
& \leq \abs*{\E\bbrace*{\indc \sum_{i<j}\td{w}_{ij}\parr*{\E_\varepsilon D_{ij} - V(x^*)}} + V(x^*)\tau_n\abs*{\E\bbrace*{\indc (|\mathbf{B}_n| + \tau_n)^{-1}}}}.
\end{align*}
By definition, it holds on $\Omega_n$ that $\|\mathbf{B}_n - \mathbf{B}\| \leq \lambda_{\tiny{\min}}(\mathbf{B})/2$, where $\lambda_{\tiny{\min}}(\mathbf{B})$ is the smallest eigenvalue of $\mathbf{B}$. Thus by Weyl's inequality, it holds that $\lambda_{\tiny{\min}}(\mathbf{B}_n)\geq \lambda_{\tiny{\min}}(\mathbf{B})/2 \geq c$ for some fixed positive constant $c$ due to the invertibility of $\mathbf{B}$ as proved in Theorem \ref{thm:v_function_upper} under Condition (c) in $\cal{P}_{\tiny{\text{vf}},(X,\varepsilon)}$. Then, using the fact that $|\mathbf{B}_n| \geq \lambda^{\ell+1}_{\tiny{\min}}(\mathbf{B}_n)$ and the boundedness of $V(\cdot)$, it holds that
\begin{align*}
\abs*{\E\bbrace*{\indc\E_\varepsilon\parr*{\widehat{V}(x^*) - V(x^*)}}} \leq \abs*{\E\bbrace*{\indc \sum_{i<j}\td{w}_{ij}\parr*{\E_\varepsilon D_{ij} - V(x^*)}}} + C\tau_n.
\end{align*}
Direct calculation shows that $\E_\varepsilon D_{ij} - V(x^*)= (f(X_i) - f(X_j))^2/2 + (V(X_i) - V(x^*))/2 + (V(X_j) - V(x^*))/2$. Due to symmetry, we only need to control the first two terms. For the first term, using the fact $\mathbf{B}_n^* = |\mathbf{B}_n|\mathbf{B}_n^{-1}$ on $\Omega_n$ ($\mathbf{B}_n$ invertible on $\Omega_n$), we have
\begin{align*}
&\ms\E\bbrace*{\indc\sum_{i<j} (|\mathbf{B}_n| +\tau_n)^{-1}w_{ij}(f(X_i) - f(X_j))^2}\\
&= \E\bbrace*{\indc {n\choose 2}^{-1}\sum_{i<j} (|\mathbf{B}_n|+\tau_n)^{-1}\mbf{q}^\top(0)\mathbf{B}_n^*\mbf{q}\parr*{\frac{X_{ij}-x^*}{h_2}}K_{ij}(f(X_i)-f(X_j))^2}\\
&= \E\bbrace*{(|\mathbf{B}_n|+\tau_n)^{-1}\mbf{q}^\top(0)\mathbf{B}_n^*\mbf{q}\parr*{\frac{X_{ij} - x^*}{h_2}}K_{ij}(f(X_i)-f(X_j))^2\indc}\\
&\lesssim h_1^{2(\alpha\wedge 1)}\E\bbrace*{(|\mathbf{B}_n|+\tau_n)^{-1}\|\mathbf{B}_n^*\|K_{ij}\indc} = h_1^{2(\alpha\wedge1)}\E\bbrace*{\frac{|\mathbf{B}_n|}{|\mathbf{B}_n| + \tau_n}\|\mathbf{B}_n^{-1}\|K_{ij}\indc}\\
&\lesssim Ch_1^{2(\alpha\wedge 1)}\E K_{ij} \lesssim Ch_1^{2(\alpha\wedge 1)}.
\end{align*}
Here, the second inequality follows from the fact that $\|\mathbf{B}_n^{-1}\| = \lambda_{\tiny{\min}}(\mathbf{B}_n)^{-1}$ is bounded by some fixed constant on $\Omega_n$, and in the last inequality we use the fact that $\E K_{ij}$ is bounded Condition (b) in $\cal{P}_{\tiny{\text{vf}},(X,\varepsilon)}$. For the second term, it holds that
\begin{align*}
&\ms\abs*{\E\parr{\indc\sum_{i<j}\td{w}_{ij}(V(X_i) - V(x^*))}}\\
& \leq \abs*{\E\parr{\indc{\sum_{i<j}\td{w}_{ij}(V(X_{ij}) - V(x^*))}}} + \abs*{\E\parr{\indc\sum_{i<j}\td{w}_{ij}(V(X_i) - V(X_{ij}))}}\define I + II.
\end{align*}
For $I$, using the H\"older property of $V(\cdot)$ and the reproducing property of local polynomial estimators (see \eqref{eq:rep_lp} in the main paper), that is,
\begin{align*}
\sum_{i<j} w_{ij}(X_{ij} - x^*)^k = \sum_{i<j} \td{w}_{ij}(X_{ij} - x^*)^k = 0, \quad k=1,2,\ldots, \ell,
\end{align*}
it holds that
\begin{align*}
&\ms\abs*{\E\bbrace*{\indc{\sum_{i<j}\td{w}_{ij}(V(X_{ij}) - V(x^*))}}}\\
&= \abs*{\E\bbrace*{\indc\sum_{i<j} \td{w}_{ij}\parr*{\sum_{k=1}^{\ell-1} \frac{V^{(k)}(x^*)}{k!}(X_{ij} - x^*)^k + \frac{V^{(\ell)}(x^* + \tau(X_{ij} - x^*))}{\ell!}(X_{ij} - x^*)^\ell}}}\\
&= \abs*{\E\bbrace*{\indc\sum_{i<j}\td{w}_{ij}\frac{(X_{ij}-x^*)^\ell}{\ell!}\parr*{V^{(\ell)}(x^* + \tau(X_{ij} - x^*)) - V^{(\ell)}(x^*)}}}\\
&= \abs*{\E\bbrace*{\indc (|\mathbf{B}_n|+\tau_n)^{-1}\mbf{q}^\top(0)\mathbf{B}_n^*\mbf{q}\parr*{\frac{X_{ij}-x^*}{h_2}}K_{ij}\frac{(X_{ij}-x^*)^\ell}{\ell!}\parr*{V^{(\ell)}(x^* + \tau(X_{ij} - x^*)) - V^{(\ell)}(x^*)}}}\\
&\lesssim \E\bbrace*{\indc |X_{ij} -x^*|^\beta (|\mathbf{B}_n|+\tau_n)^{-1}\|\mathbf{B}_n^*\|K_{ij}}\\
&\lesssim h_2^\beta\E\parr*{\indc \|\mathbf{B}_n^{-1}\|K_{ij}} \lesssim h_2^\beta, 
\end{align*}
where in the second line we use the Taylor expansion of $V(\cdot)$ around $x^*$ with some $\tau\in[0,1]$, in the fifth line we use the fact that $\mbf{q}((X_{ij}-x^*)/h_2)$ has bounded $\ell_2$ norm due to the compact support of $K(\cdot)$, and in the last inequality we use again the fact that $\|\mathbf{B}_n^{-1}\|$ is bounded by some fixed constant on $\Omega_n$. With a similar calculation and the fact that $\beta$-smooth functions are Lipschitz for $\beta \geq 1$, we have $II \lesssim h_1^{\beta\wedge 1}$. Therefore, putting together the pieces, we obtain
\begin{align*}
\abs*{\E\bbrace*{\parr*{\widehat{V}(x^*) - V(x^*)}\indc}} \leq C\parr*{h_1^{2(\alpha\wedge 1)} + h_2^{\beta} + h_1^{\beta\wedge 1} + \tau_n}.
\end{align*}
This completes the proof.
\end{proof}

\begin{lemma}
\label{lemma:v_variance}
Suppose $f\in\Lambda_{\alpha}(C_{\cal{F}})$, $V\in\Lambda_{\beta}(C_{\cal{V}})$, $\sigma^2\leq C_\sigma$ for some fixed constants $C_{\cal{F}},C_{\cal{V}},C_\sigma$, and the joint distribution of $(X,\varepsilon)$ belongs to $\cal{P}_{\tiny{\text{vf}},(X,\varepsilon)}$. Assume that $nh_2\rightarrow\infty$ and $n^2h_1h_2\rightarrow\infty$ as $n\rightarrow\infty$. Then, with $\Omega_n$ defined in the proof of Theorem \ref{thm:v_function_upper}, it holds that
\begin{align*}
\var\parr*{\widehat{V}(x^*)\indc} \leq C\parr*{n^{-1}h_2^{-1} + n^{-2}(h_1h_2)^{-1} + \tau_n^2}
\end{align*}
for some fixed positive constant $C$ that only depends on $\alpha,\beta,C_{\cal{F}},C_{\cal{V}},C_\sigma,C_0,C_\varepsilon$.
\end{lemma}
\begin{proof}
We adopt the notation $\ell,\mbf{q}(\cdot), \mathbf{B}_n, X_{ij},$ and $K_{ij}$ from the proof of Theorem \ref{thm:v_function_upper}. Also recall the definition of $\mathbf{B}_n^*$ and $D_{ij}$ from the definition of $\widehat{V}(x^*)$. Define the vector-valued U-statistic 
\begin{align*}
\mbf{U}_n\define {n\choose 2}^{-1} \mbf{g}(X_i,X_j) \define {n\choose 2}^{-1}\sum_{i<j} D_{ij}\mbf{q}\parr*{\frac{X_{ij}-x^*}{h_2}}K_{ij}
\end{align*}
and let $\mbf{\theta} \define \E \mbf{g}(X_i,X_j)\in\RR^{\ell+1}$. Then, for each $j\in[\ell+1]$, writing $\E_\varepsilon$ as the conditional expectation given $\{X_i\}_{i=1}^n$, we have
\begin{align*}
\theta_j &= \E\parr*{D_{ij}\frac{((X_{ij}-x^*)/h_2)^j}{j!}K_{ij}} = \E\parr*{\frac{((X_{ij}-x^*)/h_2)^j}{j!}K_{ij}\E_\varepsilon D_{ij}} \\
&= \E\parr*{\frac{((X_{ij}-x^*)/h_2)^j}{2j!}K_{ij}((f(X_i) -f(X_j))^2 + V(X_i) + V(X_j))} \lesssim \E K_{ij} \lesssim 1,
\end{align*}
where we have used Condition (b) in $\cal{P}_{\tiny{\text{vf}},(X,\varepsilon)}$, boundedness of $V(\cdot)$ and compact support of $K(\cdot)$. Therefore we have $\|\mbf{\theta}\| = O(1)$. With the above notation, and using the fact that $\mathbf{B}_n$ is invertible and satisfies $\mathbf{B}_n^* = |\mathbf{B}_n|\mathbf{B}_n^{-1}$ on $\Omega_n$, we have
\begin{align*}
\widehat{V}(x^*)\indc = \frac{|\mathbf{B}_n|}{|\mathbf{B}_n|+\tau_n} \mbf{q}^\top(0)\mathbf{B}_n^{-1}\mbf{U}_n\indc = \mbf{q}^\top(0)\mathbf{B}_n^{-1}\mbf{U}_n\indc - \frac{\tau_n}{|\mathbf{B}_n|+\tau_n}\mbf{q}^\top(0)\mathbf{B}_n^{-1}\mbf{U}_n\indc.
\end{align*}
Thus in order to upper bound $\var\parr*{\widehat{V}(x^*)\indc}$, it suffices to upper bound the variances of the two terms in the above display. For the second term, using the fact that $|\mathbf{B}_n|$ is bounded away from zero on $\Omega_n$ under Condition (c) in $\cal{P}_{\tiny{\text{vf}},(X,\varepsilon)}$, we have
\begin{align*}
\var\parr*{\frac{\tau_n}{|\mathbf{B}_n|+\tau_n}q^\top(0)\mathbf{B}_n^{-1}\mbf{U}_n\indc} &\lesssim \tau_n^2\E\parr*{q^\top(0)\mathbf{B}_n^{-1}\mbf{U}_n\indc}^2 \lesssim \tau_n^2\E\parr*{\|\mathbf{B}_n^{-1}\indc\|^2\|\mbf{U}_n\|^2}\lesssim \tau_n^2,
\end{align*}
where the last inequality follows since $\|\mbf{\theta}\| = O(1)$ and $\mbf{U}_n$ concentrates to $\theta$ by Lemma \ref{lemma:v_term_variance} since $nh_2\rightarrow\infty$ and $n^2h_1h_2\rightarrow\infty$. The first term can be decomposed as
\begin{align*}
\mbf{q}^\top(0)\mathbf{B}_n^{-1}\mbf{U}_n\indc = \mbf{q}^\top(0)(\mathbf{B}_n^{-1} - \mathbf{B}^{-1})\mbf{U}_n\indc + \mbf{q}^\top(0)\mathbf{B}^{-1}\mbf{U}_n\indc \define I + II.
\end{align*}
For the first term, it holds on the event $\Omega_n$ that
\begin{align*}
\|\mathbf{B}_n^{-1} - \mathbf{B}^{-1}\| = \|\mathbf{B}^{-1}\|\|\mathbf{B}_n^{-1}\|\|\mathbf{B}_n - \mathbf{B}\| \leq \|\mathbf{B}^{-1}\|(\|\mathbf{B}_n^{-1} - \mathbf{B}^{-1}\| + \|\mathbf{B}^{-1}\|)\|\mathbf{B}_n - \mathbf{B}\|.
\end{align*}
Thus on the event $\Omega_n$, it holds that $\|\mathbf{B}_n^{-1} - \mathbf{B}^{-1}\|\leq \|\mathbf{B}^{-1}\|^2\|\mathbf{B}_n - \mathbf{B}\|/(1 - \|\mathbf{B}^{-1}\|\|\mathbf{B}_n - \mathbf{B}\|)\leq 2\|\mathbf{B}^{-1}\|^2\|\mathbf{B}_n-\mathbf{B}\|$. This implies that
\begin{align*}
\var\parr*{I} &\leq \E\parr*{\mbf{q}^\top(0)(\mathbf{B}_n^{-1} - \mathbf{B}^{-1})\mbf{U}_n\indc}^2 \leq \E\parr*{\|(\mathbf{B}_n^{-1} - \mathbf{B}^{-1})\indc\|^2\|\mbf{U}_n\|^2}\\
& \lesssim \E\parr*{\|\mathbf{B}_n -\mathbf{B}\|^2\|\mbf{U}_n\|^2} \leq (\E\|\mathbf{B}_n - \mathbf{B}\|^4)^{1/2}(\E\|\mbf{U}_n\|^4)^{1/2}.
\end{align*}
Clearly, $(\E\|\mbf{U}_n\|^4)^{1/2} = O(1)$ since $\mbf{U}_n$ concentrates to $\mbf{\theta}$ and $\|\mbf{\theta}\| = O(1)$, and by Lemmas \ref{lemma:B_concentration} and \ref{lemma:tail_expectation}, it holds that $(\E\|\mathbf{B}_n - \mathbf{B}\|^4)^{1/2}\lesssim n^{-1}h_2^{-1} + n^{-2}h_1^{-1}h_2^{-1}$. This concludes that
\begin{align*}
\var(I) \lesssim n^{-1}h_2^{-1} + n^{-2}h_1^{-1}h_2^{-1}.
\end{align*}
Lastly, for $II$, writing $Z\define \mbf{q}^\top(0)\mathbf{B}^{-1}{n\choose 2}^{-1}\sum_{i<j}\mbf{g}(X_i,X_j)$, we have
\begin{align*}
\var(II) = \E(Z\indc - \E(Z\indc))^2 = \E((Z-\E Z) + \E(Z\indcc) - Z\indcc)^2 \lesssim \var(Z)  + \E(Z\indcc)^2.
\end{align*} 
By Lemma \ref{lemma:v_term_variance}, it holds that
\begin{align*}
\var(Z) \lesssim n^{-1}h_2^{-1} + n^{-2}h_1^{-1}h_2^{-1}.
\end{align*}
Lastly, by Cauchy's inequality and Lemma \ref{lemma:B_concentration}, 
\begin{align*}
\E(Z\indcc)^2 \leq (\E Z^4)^{1/2}\P^{1/2}(\Omega_n^c) \lesssim \parr*{ \|\mathbf{B}^{-1}\|^4\E\|\mbf{U}_n\|^4}^{1/2}\P^{1/2}(\Omega_n^c) = o(n^{-1}h_2^{-1} + n^{-2}h_1^{-1}h_2^{-1}).
\end{align*}
Thus we conclude that
\begin{align*}
\var(II) \lesssim n^{-1}h_2^{-1} + n^{-2}h_1^{-1}h_2^{-1}.
\end{align*}
Putting together the pieces, we have proved 
\begin{align*}
\E\parr*{\widehat{V}(x^*)\indc - \E\parr*{\widehat{V}(x^*)\indc}}^2 \lesssim n^{-1}h_2^{-1} + n^{-2}(h_1h_2)^{-1} + \tau_n^2.
\end{align*}
This completes the proof.
\end{proof}

\begin{lemma}
\label{lemma:v_term_variance}
Consider the term $Z$ defined in the proof of Lemma \ref{lemma:v_variance}:
\begin{align*}
Z = \mbf{q}^\top(0)\mathbf{B}^{-1}{n\choose 2}^{-1}\sum_{i<j} D_{ij}\mbf{q}\parr*{\frac{X_{ij}-x^*}{h_2}}K_{ij}.
\end{align*}
Then, under the same conditions of Lemma \ref{lemma:v_variance}, it holds that 
\begin{align*}
\var\parr*{Z} \leq C\parr*{n^{-1}h_2^{-1} + n^{-2}h_1^{-1}h_2^{-1}}
\end{align*}
for some fixed positive constant $C$ that only depends on $\alpha,\beta,C_0$.
\end{lemma}
\begin{proof}
Denote $\mbf{g}(X_i, X_j)\define D_{ij}\mbf{q}\parr*{\frac{X_{ij}-x^*}{h_2}}K_{ij}$ and $\mbf{\theta} \define \E \mbf{g}(X_i,X_j)$. Then, it holds that
\begin{align*}
\var(Z) &= \E\parr*{\mbf{q}^\top(0)\mathbf{B}^{-1}{n\choose 2}^{-1}\sum_{i<j} (\mbf{g}(X_i, X_j) - \mbf{\theta})}^2\\
&\leq \|\mathbf{B}^{-1}\|^2\E\nm*{{n\choose 2}^{-1}\sum_{i<j}(\mbf{g}(X_i,X_j)-\mbf{\theta})}^2\\
&\lesssim n^{-4}\sum_{i<j,i^\prime < j^\prime}\E\bbrace*{(\mbf{g}(X_i,X_j) - \mbf{\theta})^\top(\mbf{g}(X_{i^\prime},X_{j^\prime}) - \mbf{\theta})},
\end{align*}
where the last inequality follows since $\mathbf{B}^{-1}$ is invertible under Condition (c) in $\cal{P}_{\tiny{\text{vf}},(X,\varepsilon)}$ as proved in Theorem \ref{thm:v_function_upper}. Apparently, when $i,j,i^\prime, j^\prime$ are all different, the summand equals to zero. When $i,j,i^\prime, j^\prime$ take three different values, say, $i = i^\prime < j < j^\prime$, we have
\begin{align*}
&\ms\E\bbrace*{(\mbf{g}(X_i,X_j) - \mbf{\theta})^\top(\mbf{g}(X_i,X_{j^\prime}) - \mbf{\theta})}\\
 &= \E\bbrace*{D_{ij}D_{ij^\prime}\mbf{q}^\top\parr*{\frac{X_{ij}-x^*}{h_2}}\mbf{q}^\top\parr*{\frac{X_{ij^\prime}-x^*}{h_2}}K_{ij}K_{ij^\prime}} - \|\mbf{\theta}\|^2.
\end{align*}
Let $\mbf{Z}_1\define \mbf{g}(X_i, X_j)$ and $\mbf{Z}_2 \define \mbf{g}(X_i, X_{j^\prime})$. Then, for any $k\in[\ell+1]$, it holds that $\abs*{Z_{1,k}}\lesssim D_{ij}K_{h_1}(X_i-X_j)K_{h_2}(X_{ij}-x^*)$ and similarly for $Z_{2,k}$. Therefore, using the finite fourth moment of $\varepsilon$ in Condition (d) of $\cal{P}_{\tiny{\text{vf}},(X,\varepsilon)}$ in the calculation of $\E_\varepsilon (D_{ij}D_{ij^\prime})$ and the fact that both $f(\cdot)$ and $V(\cdot)$ are bounded, we have
\begin{align*}
&\ms\E(|Z_{1,k}Z_{2,k}|)\\
&\lesssim \E\bbrace*{D_{ij}K_{h_1}(X_i-X_j)K_{h_2}(X_{ij}-x^*)D_{ij^\prime}K_{h_1}(X_i - X_{j^\prime})K_{h_2}(X_{ij^\prime}-x^*)}\\
&\lesssim \E\bbrace*{K_{h_1}(X_i-X_j)K_{h_2}(X_{ij}-x^*)K_{h_1}(X_i - X_{j^\prime})K_{h_2}(X_{ij^\prime}-x^*)}\\
&= \int_{\RR^3} \frac{1}{h_1^2h_2^2}K\parr*{\frac{v-u}{h_1}}K\parr*{\frac{\frac{u+v}{2} - x^*}{h_2}}K\parr*{\frac{w-u}{h_1}}K\parr*{\frac{\frac{u+w}{2} - x^*}{h_2}}p_X(u)p_X(v)p_X(w)dudvdw\\
&= \frac{1}{h_2}\int_{\RR^3} K(\td{u})K(\td{v})K(\td{w})K\parr*{\frac{\td{u}h_2 + h_1(\td{w}-\td{v})/2}{h_2}}p_X(s_1)p_X(s_2)p_X(s_3)d\td{u}d\td{v}d\td{w}\\
&\lesssim 1/h_2,
\end{align*}
where we again invoke Condition (b) in $\cal{P}_{\tiny{\text{vf}},(X,\varepsilon)}$ and the compact support of $K(\cdot)$, and
\begin{align*}
s_1 = \td{u}h_2+x^*-\frac{h_1\td{v}}{2}, \quad s_2 = \td{u}h_2+x^*+\frac{h_1\td{v}}{2}, \quad s_3 = \td{u}h_2 + x^* + \td{w}h_1 - \frac{\td{v}h_1}{2}.
\end{align*}
This, along with the fact that $\|\mbf{\theta}\| = O(1)$, concludes that 
\begin{align*}
\E\bbrace*{(\mbf{g}(X_i,X_j) - \mbf{\theta})^\top(\mbf{g}(X_{i^\prime},X_{j^\prime}) - \mbf{\theta})}\lesssim h_2^{-1}
\end{align*}
when $i,j,i^\prime, j^\prime$ take three different values. Similarly, one can prove that when $i,j,i^\prime, j^\prime$ take two different values, that is $i = i^\prime$ and $j = j^\prime$,
\begin{align*}
\E\bbrace*{(\mbf{g}(X_i,X_j) - \mbf{\theta})^\top(\mbf{g}(X_{i^\prime},X_{j^\prime}) - \mbf{\theta})}\lesssim (h_1h_2)^{-1}.
\end{align*}
Putting together the pieces, we obtain that 
\begin{align*}
\var(Z) \lesssim \frac{n^3h_2^{-1} + n^2(h_1h_2)^{-1}}{n^4} = n^{-1}h_2^{-1} + n^{-2}h_1^{-1}h_2^{-1}.
\end{align*}
This completes the proof.
\end{proof}

\begin{lemma}
\label{lemma:B_concentration}
Suppose Condition (b) in $\cal{P}_{\tiny{\text{vf}},(X,\varepsilon)}$ holds. Assume $n^2h_1h_2\rightarrow\infty$, $nh_2\rightarrow \infty$, and $h_1/h_2\rightarrow 0$. Then, for any $u,v > 0$, the matrices $\mathbf{B}_n$ and $\mathbf{B}$ defined in the proof of Theorem \ref{thm:v_function_upper} satisfy
\begin{align*}
\P\parr*{\|\mathbf{B}_n - \mathbf{B}\|\geq C(v^{1/2}n^{-1/2}h_2^{-1/2} + u^{1/2}n^{-1}h_1^{-1/2}h_2^{-1/2})} \leq C\parr*{\exp(-u) + \exp(-v)}
\end{align*}
for some fixed positive constant $C$.
\end{lemma}
\begin{proof}
Using a standard entropy argument (see, for example, Lemma 5.3 in \cite{vershynin2010introduction}), it holds that for any $t>0$,
\begin{align*}
\P(\|\mathbf{B}_n - \mathbf{B}\| \geq t) \leq N\max_{1\leq i\leq N} \P\parr*{\abs*{\mbf{a}_i^\top(\mathbf{B}_n - \mathbf{B})\mbf{a}_i}\geq t/2},
\end{align*} 
where $N \define 5^{\ell+1}$ and $\{\mbf{a}_i\}_{i=1}^{N}$ is a $1/2$-net on the unit sphere $\mathbb{S}^{\ell}$. We now upper bound $\P\parr*{\abs*{\mbf{a}^\top(\mathbf{B}_n - \mathbf{B})\mbf{a}}\geq t/2}$ for an arbitrary $\mbf{a}\in\mathbb{S}^{\ell}$ with the help of Lemma \ref{lemma:Ubernstein}. For this, we upper bound the five quantities $B_1,B_2,B_3,\nu_1^2,\nu_2^2$ therein. Denote the kernel of $\mbf{a}^\top\mathbf{B}_n\mbf{a}$ as $g$ and its linear part as $g_1$, that is,
\begin{align*}
g(X_i, X_j) \define \parr*{\mbf{a}^\top \mbf{q}\parr*{\frac{X_{ij}-x^*}{h_2}}}^2 K_{ij} \text{ and } g_1(X_i) = \E\parr*{g(X_i,X_j)\mid X_i}.
\end{align*}
Then, we have
\begin{align*}
g_1(X_i) &= \int \parr*{\mbf{a}^\top \mbf{q}\parr*{\frac{(u+X_i)/2 - x^*}{h_2}}}^2\frac{1}{h_1}K\parr*{\frac{u - X_i}{h_1}}\frac{1}{h_2}K\parr*{\frac{(u+X_i)/2 - x^*}{h_2}}p_X(u)du\\
&\lesssim \frac{1}{h_2}\int K(\td{u})K\parr*{\frac{X_i + h_1\td{u}/2 - x^*}{h_2}}p_X(X_i + \td{u}h_1)d\td{u} \lesssim h_2^{-1},
\end{align*}
where in the first inequality we use the fact that $\mbf{q}\parr*{\frac{(u+X_i)/2 - x^*}{h_2}}$ has bounded $\ell_2$ norm due to the compact support of $K(\cdot)$, and in the last inequality we apply Condition (b) in $\cal{P}_{\tiny{\text{vf}},(X,\varepsilon)}$. Therefore $B_1\lesssim h_2^{-1}$ and one can similarly show that $\nu_1^2\lesssim h_2^{-1}$.

For $B_2$, we have
\begin{align*}
B_2^2 &= n\sup_{X_i}\E\bbrace*{\parr*{\mbf{a}^\top \mbf{q}\parr*{\frac{X_{ij}-x^*}{h_2}}}^4 \frac{1}{h_1^2}K^2\parr*{\frac{X_i-X_j}{h_1}}\frac{1}{h_2^2}K^2\parr*{\frac{X_{ij}-x^*}{h_2}}\mid X_i}\\
&\lesssim \frac{n}{h_1h_2^2}\sup_{X_i}\E\bbrace*{\frac{1}{h_1}K\parr*{\frac{X_i - X_j}{h_1}}\mid X_i}\\
&= \frac{n}{h_1h_2^2}\sup_{X_i}\int K(u)p_X(X_i + uh_1)du \lesssim nh_1^{-1}h_2^{-2}.
\end{align*}
This concludes that $B_2\lesssim n^{1/2}h_1^{-1/2}h_2^{-1}$. Moreover, one can easily show that $\nu_2^2 \lesssim (h_1h_2)^{-1}$ and $B_3\lesssim (h_1h_2)^{-1}$. Putting together the pieces and applying Lemma \ref{lemma:Ubernstein}, we obtain that for any $u,v>0$,
\begin{align*}
\P\parr*{\abs*{\mbf{a}^\top(\mathbf{B}_n - \mathbf{B})\mbf{a}} \geq a_1v^{1/2} + a_2v + b_1u^{1/2} + b_2u + b_3u^{3/2} + b_4u^2} \leq C\parr*{\exp(-v) + \exp(-u)}, 
\end{align*}
where $a_1\lesssim n^{-1/2}h_2^{-1/2}, a_2\lesssim n^{-1}h_2^{-1}$ and $b_1\lesssim n^{-1}h_1^{-1/2}h_2^{-1/2}, b_2\lesssim n^{-1}h_2^{-1}, b_3\lesssim n^{-3/2}h_1^{-1/2}h_2^{-1}, b_4\lesssim n^{-2}h_1^{-1}h_2^{-1}$. Under the conditions $n^2h_1h_2\rightarrow \infty$, $nh_2\rightarrow \infty$ and $h_1/h_2\rightarrow 0$ as $n\rightarrow \infty$, the dominant terms are $a_1$ and $b_1$, that is,
\begin{align*}
n^{-1/2}h_2^{-1/2} \vee n^{-1}h_1^{-1/2}h_2^{-1/2}.
\end{align*}
This completes the proof.
\end{proof}

\begin{lemma}
\label{lemma:chi_match}
Under the setting and conditions of Theorem \ref{thm:v_lower_int}, for any positive $\eta > 0$, there exists an i.i.d. sequence $\{r_{i,k}\}_{i\in[m_j],k\in[N_1]}$ (with $m_j, N_1$ defined in Theorem \ref{thm:v_lower_int}) with range contained in $[-B, B]$ for some $B$ only depending on $\alpha,\beta,\eta$, such that the probability measures $\P_j$ and $\td{\P}_j$ defined therein satisfy that
\begin{align*}
\chi^2\parr*{\P_j, \td{\P}_j; \Omega_n} \lesssim n^{-\eta},
\end{align*}
where, for any measurable subset $\cal{E}$ and two probability measures $\P$ and $\Q$, $\chi^2(\P,\Q; \cal{E})$ is the conditional $\chi^2$-distance defined as
\begin{align*}
\chi^2(\P,\Q; \cal{E}) \define \int_{\cal{E}}\frac{(p-q)^2}{q}
\end{align*}
with $p,q$ being the densities of $\P$ and $\Q$ with respect to some common dominating measure.
\end{lemma}
\begin{proof}
Without loss of generality, let $j = 1$. First note that the conditional $\chi^2$-distance can be written as
\begin{align*}
\chi^2\parr*{\P_1, \td{\P}_1; \Omega_n} &= \int \mathbbm{1}\{\Omega_n\}\frac{\parr*{p_1(\mbf{x}, \mbf{y}) - \td{p}_1(\mbf{x}, \mbf{y})}^2}{\td{p}_1(\mbf{x}, \mbf{y})} d\mbf{x}d\mbf{y}\\
&= \int \mathbbm{1}\{\Omega_n\}p(\mbf{x})d\mbf{x}\int \frac{\parr*{p_1(\mbf{y}\mid\mbf{x}) - \td{p}_1(\mbf{y}\mid \mbf{x})}^2}{\td{p}_1(\mbf{y}\mid \mbf{x})}d\mbf{y}\\
&= \E\bbrace*{\mathbbm{1}\{\Omega_n\}\chi^2\parr*{\P_1(\mbf{y}\mid \mbf{x}), \td{\P}_1(\mbf{y}\mid \mbf{x})}},
\end{align*}
where we have used $p(\cdot)$ to represent the density of $\{X_i\}_{i=1}^n$ under both $\P_1$ and $\td{\P}_1$.

Recall the definition of the location index sequence $\{b_i\}_{i=1}^n$ in Theorem \ref{thm:v_lower_int}. Consider any realization of $\{X_i\}_{i=1}^n$ in $\Omega_n$, and assume that based on their location indices $\{b_i\}_{i=1}^n$, $\{X_i\}_{i=1}^n$ is partitioned into $L_1 + L_2$ clusters with cardinality $s_\ell$ such that the $X_i$'s in the same cluster have the same value $b_i$. Moreover, for those data points in the first $L_1$ clusters, the location index $b_i = (k_i,\ell_i)$ satisfies that $\Delta_{1,k_i} = 1$ while for the data points in the last $L_2$ clusters, it holds that $\Delta_{1,k_i} = 0$ (recall the definition of $\mbf{\Delta}_1 = (\Delta_{1,1},\ldots, \Delta_{1,N_2})$ in the lower bound design of Theorem \ref{thm:v_lower_int}. Apparently, we have the relations $1\leq L_1 + L_2 \leq n$, $\sum_{\ell=1}^{L_1+L_2}s_\ell = n$ and $1\leq s_\ell \leq K$ for $\ell\in[L_1 + L_2]$ (recall the definition of $K$ in Theorem \ref{thm:v_lower_int}). Moreover, denoting $\P_{1,\pi_\ell}$ and $\td{\P}_{1,\pi_\ell}$ (resp. $p_{1,\pi_\ell}$ and $\td{p}_{1,\pi_\ell}$) for each $\ell\in[L_1 + L_2]$ as the joint distribution (resp. density) of those $Y_i$'s in the $\ell$th cluster conditioning on the given realization $\{X_i\}_{i=1}^n$ under $\P_1$ and $\td{\P}_1$, we have 
\begin{align*}
\chi^2\parr*{\P_1(\mbf{y}\mid \mbf{x}), \td{\P}_1(\mbf{y}\mid \mbf{x})} &= \prod_{\ell=1}^{L_1+L_2}\parr*{1 + \chi^2\parr*{\P_{1,\pi_\ell}, \td{\P}_{1,\pi_\ell}}} - 1\\
&=  \prod_{\ell=1}^{L_1}\parr*{1 + \chi^2\parr*{\P_{1,\pi_\ell}, \td{\P}_{1,\pi_\ell}}} - 1\\
&\leq \exp\bbrace*{\sum_{\ell=1}^{L_1}\chi^2\parr*{\P_{1,\pi_\ell}, \td{\P}_{1,\pi_\ell}}} - 1
\end{align*}  
where the first equality follows by mutual independence of data points in each cluster, the second inequality follows from the fact that for those data points $Y_i$'s in the latter $L_2$ clusters, each $Y_i\mid X_i\sim\cal{N}(0,1)$ under both $\P_1$ and $\td{\P}_1$. Since $L_1\leq n$, it suffices to show that for any realization of $\{X_i\}_{i=1}^n$ in $\Omega_n$, by matching enough moments, we have $\chi^2\parr*{\P_{1,\pi_\ell}, \td{\P}_{1,\pi_\ell}} \leq n^{-\eta}$ for any $\eta > 0$.

For each $\ell\in[L_1]$, $\abs*{p_{1,\pi_\ell} - \td{p}_{1,\pi_\ell}}$ only depends on the $\ell$th cluster via its cardinality, which we now control for a general cluster size $1\leq d\leq K$. Without loss of generality, we assume that $\ell=1$ and the $y_i$'s in this cluster are $\{y_1,\ldots, y_d\}$ with common location index $b_i = (1,1)$. In this case, under the choice of $\theta_n^2$ and $h_1$ given in Theorem \ref{thm:v_lower_int}, we clearly have $Y_i = \theta_nr_{1,1} + (1-h_2^\beta)^{1/2}\varepsilon_i$ under $\P_1$ and $Y_i = \theta_n\td{r}_{1,1} + (1-h_2^\beta)^{1/2}\varepsilon_i$ under $\td{\P}_1$ for $i\in[d]$, where the sequence $\{\varepsilon_i\}_{i=1}^d$ follows the standard normal distribution under both $\P_1$ and $\td{\P}_1$. Define $\sigma^2 \define 1-h_2^\beta = 1-\theta_n^2$. Then, it holds that
\begin{align*}
p_{1,\pi_\ell}(y_1,\ldots,y_d) &= \int_{-\infty}^\infty \frac{1}{\sigma}\varphi\parr*{\frac{y_1-\theta_nv}{\sigma}}\ldots\frac{1}{\sigma}\varphi\parr*{\frac{y_d - \theta_nv}{\sigma}}\G(dv),\\
\td{p}_{1,\pi_\ell}(y_1,\ldots,y_d) &= \int_{-\infty}^\infty \frac{1}{\sigma}\varphi\parr*{\frac{y_1-\theta_nv}{\sigma}}\ldots\frac{1}{\sigma}\varphi\parr*{\frac{y_d - \theta_nv}{\sigma}}\varphi(v)dv,
\end{align*}
where $\G$ is the distribution of $\{r_{i,k}\}_{i\in[m_j],k\in[N_1]}$. Using the well-known equality $\varphi(t - \theta_nv) = \varphi(t)\parr*{\sum_{k=0}^\infty v^k\theta_n^kH_k(t)/k!}$ for any $t,v$, where $H_k$ is the $k$th order Hermite polynomial, it holds that
\begin{align*}
&\ms\varphi\parr*{\frac{y_1 -\theta_nv}{\sigma}}\ldots\varphi\parr*{\frac{y_d -\theta_nv}{\sigma}}\\
 &= \varphi\parr*{\frac{y_1}{\sigma}}\ldots\varphi\parr*{\frac{y_d}{\sigma}}\sum_{k_1,\ldots,k_d = 0}^\infty \parr*{\frac{\theta_nv}{\sigma}}^{k_1+\ldots+k_d}\frac{H_{k_1}(y_1/\sigma)}{k_1!}\ldots\frac{H_{k_d}(y_d/\sigma)}{k_d!}\\
&= \varphi\parr*{\frac{y_1}{\sigma}}\ldots\varphi\parr*{\frac{y_d}{\sigma}}\sum_{k=0}^\infty \parr*{\frac{\theta_nv}{\sigma}}^k\sum_{k_1+\ldots+k_d = k}^\infty \parr*{\frac{\theta_nv}{\sigma}}^k\frac{H_{k_1}(y_1/\sigma)}{k_1!}\ldots\frac{H_{k_d}(y_d/\sigma)}{k_d!},
\end{align*}
and therefore by the symmetry of $\G$ and by matching the moments of $\G$ and the standard normal distribution up to order $2p$ for some positive integer $p$ to be chosen later, we obtain
\begin{align*}
&\ms p_{1,\pi_\ell}(y_1,\ldots,y_d) - \td{p}_{1,\pi_\ell}(y_1,\ldots,y_d)\\
&= \varphi\parr*{\frac{y_1}{\sigma}}\ldots\varphi\parr*{\frac{y_d}{\sigma}}\sum_{k=0}^\infty \parr*{\frac{\theta_n}{\sigma}}^k \sum_{k_1+\ldots+k_d=k}\frac{H_{k_1}(y_1/\sigma)}{k_1!}\ldots\frac{H_{k_d}(y_d/\sigma)}{k_d!}\int v^k(\G - \Phi)(dv)\\
&= \varphi\parr*{\frac{y_1}{\sigma}}\ldots\varphi\parr*{\frac{y_d}{\sigma}}\sum_{k=p}^\infty \parr*{\frac{\theta_n}{\sigma}}^{2k} \sum_{k_1+\ldots+k_d=2k}\frac{H_{k_1}(y_1/\sigma)}{k_1!}\ldots\frac{H_{k_d}(y_d/\sigma)}{k_d!}\int v^{2k}(\G - \Phi)(dv).
\end{align*}
Define $\delta_{2k} \define \int v^{2k}(\G - \Phi)(dv)$. Then, the above inequality further implies that
\begin{equation}
\begin{aligned}
\label{eq:chi_1}
&\ms\parr*{p_{1,\pi_\ell}(y_1,\ldots,y_d) - \td{p}_{1,\pi_\ell}(y_1,\ldots,y_d)}^2\\
& = \varphi^2\parr*{\frac{y_1}{\sigma}}\ldots\varphi^2\parr*{\frac{y_d}{\sigma}}\sum_{k,\ell=p}^\infty \parr*{\frac{\theta_n}{\sigma}}^{2k+2\ell}\sum_{\substack{k_1+\ldots+k_d=2k\\ \ell_1+\ldots+\ell_d = 2\ell}}\frac{H_{k_1}(y_1/\sigma)}{k_1!}\frac{H_{\ell_1}(y_1/\sigma)}{\ell_1!}\ldots\frac{H_{k_d}(y_d/\sigma)}{k_d!}\frac{H_{\ell_d}(y_d/\sigma)}{\ell_d!}
\end{aligned}
\end{equation}
On the other hand, letting $Z\sim \G$, we have
\begin{align}
\td{p}_{1,\pi_\ell}(y_1,\ldots,y_d) &= \int \frac{1}{\sigma}\varphi\parr*{\frac{y_1-\theta_nv}{\sigma}}\ldots\frac{1}{\sigma}\varphi\parr*{\frac{y_d-\theta_nv}{\sigma}}\G(dv)\nonumber\\
&= \frac{1}{\sigma^d}\varphi\parr*{\frac{y_1}{\sigma}}\ldots\varphi\parr*{\frac{y_d}{\sigma}}\int \exp\bbrace*{-\frac{d}{2\sigma^2}(\theta_nv)^2 +\frac{\sum_{i=1}^d y_i\theta_nv}{\sigma^2}}\G(dv)\nonumber\\
&= \frac{1}{\sigma^d}\varphi\parr*{\frac{y_1}{\sigma}}\ldots\varphi\parr*{\frac{y_d}{\sigma}}\E\bbrace*{\exp\bbrace*{-\frac{d}{2\sigma^2}\theta_n^2Z^2+ \frac{\sum_{i=1}^d y_i\theta_n}{\sigma^2}Z}}\nonumber\\
&\geq \frac{1}{\sigma^d}\varphi\parr*{\frac{y_1}{\sigma}}\ldots\varphi\parr*{\frac{y_d}{\sigma}}\exp\parr*{-\frac{d\theta_n^2}{2\sigma^2}}\label{eq:chi_2},
\end{align}
where the last inequality follows from Jensen's inequality and the fact $\E Z = 0$, $\E Z^2 = 1$ from moment matching. Combining \eqref{eq:chi_1} and \eqref{eq:chi_2}, we obtain that
\begin{align*}
\chi^2(\P_{1,\pi_\ell}, \td{\P}_{1,\pi_\ell}) &= \int \frac{\parr*{p_{1,\pi_\ell}(y_1,\ldots,y_d) - \td{p}_{1,\pi_\ell}(y_1,\ldots,y_d)}^2}{\td{p}_{1,\pi_\ell}(y_1,\ldots,y_d)}dy_1\ldots dy_d\\
&\leq \sigma^d\sum_{k,\ell=p}^\infty\delta_{2k}\delta_{2\ell}\sum_{\substack{k_1+\ldots+k_d=2k\\ \ell_1+\ldots+\ell_d = 2\ell}} \parr*{\frac{\theta_n}{\sigma}}^{2k+2\ell}\prod_{j=1}^d\int \varphi\parr*{\frac{y_j}{\sigma}}\frac{H_{k_j}(y_j/\sigma)}{k_j!}\frac{H_{\ell_j}(y_j/\sigma)}{\ell_j!}dy_j\\
&= \sigma^{2d} \sum_{k,\ell=p}^\infty\delta_{2k}\delta_{2\ell}\sum_{\substack{k_1+\ldots+k_d=2k\\ \ell_1+\ldots+\ell_d = 2\ell}} \parr*{\frac{\theta_n}{\sigma}}^{2k+2\ell}\prod_{j=1}^d\int \varphi\parr*{y_j}\frac{H_{k_j}(y_j)}{k_j!}\frac{H_{\ell_j}(y_j)}{\ell_j!}dy_j\\
&\leq \sum_{k=p}^\infty \parr*{\frac{\theta_n}{\sigma}}^{4k}\delta_{2k}^2\sum_{k_1+\ldots+k_d=2k}\frac{1}{k_1!}\ldots\frac{1}{k_d!},
\end{align*} 
where the last inequality follows from the fact that $\sigma \leq 1$ and $\int \varphi(t)H_k(t)H_\ell(t)dt = k!\mathbbm{1}\parr*{k = \ell}$. Now, using the multinomial identity 
\begin{align*}
\sum_{k_1+\ldots+k_d = 2k}\frac{(2k)!}{k_1!\ldots k_d!}\parr*{\frac{1}{d}}^{2k} = 1,
\end{align*}
we obtain that $\sum_{k_1+\ldots+k_d=2k}1/(k_1!\ldots k_d!) = d^{2k}/(2k)!$, therefore it holds that
\begin{align*}
\chi^2(\P_{1,\pi_\ell}, \td{\P}_{1,\pi_\ell}) \leq \sum_{k=p}^\infty \parr*{\frac{\theta_n\sqrt{d}}{\sigma}}^{4k}\frac{1}{(2k)!}\delta_{2k}^2.
\end{align*}
Now, by Lemma \ref{lemma:moment_matching}, for any positive integer $p$,  we can find a symmetric distribution $\G$ that has the same first $p$ moments as the standard normal distribution and is compactly supported on $[-B,B]$ for some $B$ that only depends on $p$. This combined with the fact that $\int t^{2k}\varphi(t)dt = (2k-1)!!$ implies that $\delta_{2k}^2\lesssim B^{4k} + (2k)!$. We therefore obtain
\begin{align*}
\chi^2(\P_{1,\pi_\ell}, \td{\P}_{1,\pi_\ell}) \lesssim \sum_{k=p}^\infty \parr*{\frac{\theta_n\sqrt{d}}{\sigma}}^{4k} \lesssim n^{-\eta}
\end{align*}
by choosing a sufficiently large $p$ that only depends on $\alpha,\beta$ and $\eta$, where we also use the fact that $d$ is bounded by an absolute constant and for sufficiently large $n$, it holds that $\sigma > 1/2$. This completes the proof.
\end{proof}

\begin{lemma}
\label{lemma:cond_kull}
For some $M\geq 2$, let $\P_0,\P_1,\ldots, \P_M$ be $M+1$ hypotheses on some measurable space $(\cal{X}, \cal{A})$ such that for each $0 \leq i\neq j\leq M$, $\P_i$ and $\P_j$ are mutually absolutely continuous. Let $\Omega$ be a measurable subset of $\cal{X}$ such that $\P_j(\Omega)$ is identical for all $0\leq j\leq M$. Define the ``conditional" version of Kullback divergence as
\begin{align}
\label{eq:kull_cond}
K(\P,\Q; \Omega) \define \int_{\Omega} \log\parr*{\frac{d\P}{d\Q}}d\P.
\end{align}
Then, if
\begin{align*}
\frac{1}{M}\sum_{j=1}^MK(\P_j, \P_0; \Omega) \leq c^*\log(M)
\end{align*}
for some $0 < c^* < 1/8$, the following statement holds
\begin{align*}
\inf_{\psi}p_{e,M}(\psi) \geq \frac{\sqrt{M}}{1+\sqrt{M}}\parr*{\P_0(\Omega) - 2c^* - \sqrt{\frac{2c^*}{\log(M)}}},
\end{align*}  
where the infimum ranges over all tests taking values in $\{0,1,\ldots, M\}$ and $p_{e,M}(\psi)\define \max_{0\leq j\leq M}\P_j(\psi \neq j)$.
\end{lemma}
\begin{proof}
This is exactly the conditional version of Theorem 2.5 in \cite{tsybakov2009introduction}. We first show that subject to the condition
\begin{align*}
\frac{1}{M}\sum_{j=1}^MK(\P_j, \P_0; \Omega) \leq c
\end{align*}
for some $c > 0$, for all $0 < \tau < 1$, it holds that
\begin{align}
\label{eq:kull_condition}
\frac{1}{M}\sum_{j=1}^M\P_j\parr*{\frac{d\P_0}{d\P_j}\geq \tau} \geq \P_0(\Omega)- c^\prime,
\end{align}
where $c^\prime\define -(c+\sqrt{c/2})/\log(\tau)$. For this, we have for each $j\in[M]$
\begin{align*}
\P_j\parr*{\frac{d\P_0}{d\P_j}\geq \tau} &= \P_j\parr*{\frac{d\P_j}{d\P_0}\leq \frac{1}{\tau}}\\
&=1 - \bbrace*{\P_j\parr*{\bbrace*{\frac{d\P_j}{d\P_0}\geq \frac{1}{\tau}} \bigcap \Omega} + \P_j\parr*{\bbrace*{\frac{d\P_j}{d\P_0}\geq \frac{1}{\tau}} \bigcap \Omega^c}}\\
&\geq \P_0(\Omega) - \P_j\parr*{\bbrace*{\frac{d\P_j}{d\P_0}\geq \frac{1}{\tau}} \bigcap \Omega}\\
&= \P_0(\Omega) - \P_j\parr*{\bbrace*{\log\parr*{\frac{d\P_j}{d\P_0}}\geq \log\parr*{\frac{1}{\tau}}} \bigcap \Omega}\\
&\geq \P_0(\Omega) - (\log(1/\tau))^{-1}\E_{\P_j}\parr*{\log\parr*{\frac{d\P_j}{d\P_0}}_+\mathbbm{1}\bbrace*{\Omega}},
\end{align*}
where in the third line we use the fact that $\P_j(\Omega) = \P_0(\Omega)$, and for a real number $a$, $a_+ \define \max\{0,a\}$. Let $p_0$ and $p_j$ be the densities of $\P_0$ and $\P_j$ with respect to some common dominating measure. Then, by definition of the conditional Kullback divergence, Lemma \ref{lemma:upper_total}, and Lemma \ref{lemma:cond_pinsker}, it holds that
\begin{align*}
\E_{\P_j}\parr*{\log\parr*{\frac{d\P_j}{d\P_0}}_+\mathbbm{1}\bbrace*{\Omega}} &= \int_{\Omega} p_j\parr*{\log\frac{p_j}{p_0}}_+\\
&= K(\P_j, \P_0; \Omega) + \int_{\Omega} p_j\parr*{\log\frac{p_j}{p_0}}_-\\
&\leq K(\P_j, \P_0; \Omega) + \TV(\P_j, \P_0; \Omega)\\
&\leq K(\P_j, \P_0; \Omega) + \sqrt{K(\P_j, \P_0; \Omega)/2}.
\end{align*}
Now, by the condition $\sum_{j=1}^MK(\P_j, \P_0; \Omega)/M \leq c$ and Cauchy's inequality, it holds that
\begin{align*}
\frac{1}{M}\sum_{j=1}^M \sqrt{K(\P_j, \P_0; \Omega)}\leq \bbrace*{\frac{1}{M}\sum_{j=1}^M K(\P_j, \P_0; \Omega)}^{1/2} \leq \sqrt{c}.
\end{align*}
We therefore conclude that \eqref{eq:kull_condition} is true.
Next, by Proposition 2.2 in \cite{tsybakov2009introduction}, we obtain that
\begin{align*}
\inf_{\psi}p_{e,M}(\psi) \geq \sup_{0<\tau<1}\frac{\tau M}{\tau M+ 1}\bbrace*{\frac{1}{M}\sum_{j=1}^M \P_j\parr*{\frac{d\P_0}{d\P_j}\geq \tau}} \geq \sup_{0<\tau<1}\frac{\tau M}{\tau M+ 1}\parr*{\P_0(\Omega) + \frac{c + \sqrt{c/2}}{\log \tau}}.
\end{align*}
Lastly, by choosing $c = c^*\log M$ and $\tau = 1/\sqrt{M}$, we obtain
\begin{align*}
\inf_{\psi}p_{e,M}(\psi) &\geq \sup_{0<\tau<1}\frac{\tau M}{\tau M+ 1}\bbrace*{\frac{1}{M}\sum_{j=1}^M \P_j\parr*{\frac{d\P_0}{d\P_j}\geq \tau}}\\
 &\geq \frac{\sqrt{M}}{1+\sqrt{M}}\parr*{\P_0(\Omega)-2c^* - \sqrt{\frac{2c^*}{\log M}}}.
\end{align*}
This completes the proof.
\end{proof}

\begin{lemma}
\label{lemma:upper_total}
Let $\P$ and $\Q$ be two probability measures on a measurable space $(\cal{X},\cal{A})$ such that $\P\ll \Q$, and $\Omega$ be a measurable subset of $\cal{X}$. Define the conditional version of the total variation distance as follows
\begin{align}
\label{eq:tv_cond}
\TV(\P, \Q; \Omega) \define \sup_{A\in\cal{A}}\abs*{\P(A\bigcap \Omega) - \Q(A\bigcap \Omega)}.
\end{align}
Then, it holds that
\begin{align*}
\int_{\Omega} \parr*{\log\parr*{\frac{d\P}{d\Q}}}_- \leq \TV(\P,\Q; \Omega),
\end{align*}
where $a_- \define \max\{0, -a\}$.
\end{lemma}
\begin{proof}
Let $p$ and $q$ be the densities of $\P$ and $\Q$ with respect to some common dominating measure, and define $A\define \{q\geq p > 0\}$. Then, it holds that
\begin{align*}
\int_{\Omega}  \parr*{\log\parr*{\frac{d\P}{d\Q}}}_-d\P &= \int_{\Omega \bigcap \{p,q>0\}} p\parr*{\log\frac{p}{q}}_- = \int_{A\bigcap\Omega} p\log\frac{q}{p} \leq \int_{A\bigcap\Omega} q - p \leq \TV(\P,\Q; \Omega).
\end{align*}
This completes the proof.
\end{proof}

\begin{lemma}
\label{lemma:cond_pinsker}
Let $\P$ and $\Q$ be two probability measures on a measurable space $(\cal{X},\cal{A})$ such that $\P\ll \Q$, and $\Omega$ be a measurable subset of $\cal{X}$ such that $\P(\Omega) = \Q(\Omega)$. For the conditional version of the Kullback divergence (defined in \eqref{eq:kull_cond}) and total variation distance (defined in \eqref{eq:tv_cond}),
it holds that
\begin{align*}
\TV(\P,\Q; \Omega) \leq \sqrt{K(\P,\Q;\Omega)/2}.
\end{align*}
\end{lemma}
\begin{proof}
Firstly, using the condition $\P(\Omega) = \Q(\Omega)$, it can be readily verified that the conditional total variation distance can be equivalently written as
\begin{align*}
\TV(\P,\Q; \Omega) = \frac{1}{2}\int_{\Omega}|p - q|,
\end{align*}
where $p$ and $q$ are the densities of $\P$ and $\Q$ with respect to some common dominating measure (cf. Lemma 2.1 in \cite{tsybakov2009introduction}). Then, following the proof of the first Pinsker's inequality in Lemma 2.5 in \cite{tsybakov2009introduction}, it holds that
\begin{align*}
\TV(\P,\Q;\Omega) &= \frac{1}{2}\int_{\Omega} \abs*{p-q}\\
&= \frac{1}{2}\int_{\Omega\bigcap\{q>0\}} \abs*{\frac{p}{q}-1}q\\
&\leq \frac{1}{2}\int_{\Omega\bigcap\{q>0\}} q\sqrt{\parr*{\frac{4}{3} + \frac{2p}{3q}}\psi(\frac{p}{q})}\\
&\leq \frac{1}{2}\bbrace*{\int_{\Omega}\parr*{\frac{4q}{3}+\frac{2p}{3}}}^{1/2}\bbrace*{\int_{\Omega} q\parr*{\frac{p}{q}\log\frac{p}{q} - \frac{p}{q} + 1}}^{1/2}\\
&= \sqrt{\frac{\P(\Omega)}{2}}K^{1/2}(\P,\Q; \Omega)\\
&\leq \sqrt{K(\P,\Q; \Omega)/2},
\end{align*}
where $\psi(x) \define x\log x - x + 1$. This completes the proof.
\end{proof}

\section{Proofs of results in Section \ref{sec:discussion}}
We only provide the proofs for Propositions \ref{prop:multivariate_upper}, \ref{prop:multivariate_lower}, \ref{prop:additive_fixed_GD}, \ref{prop:additive_random_independent}, \ref{prop:additive_random_known}-\ref{prop:cv_lower}. The proofs of Propositions \ref{prop:multivariate_fixed_DD}, \ref{prop:additive_fixed_DD}, \ref{prop:additive_random_lower}, and \ref{prop:quad_est} are straightforward, and the proof of Proposition \ref{prop:multivariate_random_known} is similar to that of Proposition \ref{prop:additive_random_known}.

\subsection{Proof of Proposition \ref{prop:multivariate_upper}}
\begin{proof}

Given the proof of Theorem \ref{thm:nonpar_upper} and its supporting lemmas, the proof here is relatively straightforward. We only provide here a sketched version for completeness. For simplicity, we only prove the case with $d = 2$, and we will show that the desired upper bound can be achieved with the bandwidths choices 
\begin{align*}
h_1\asymp n^{-2\alpha_2/(4\alpha_1\alpha_2 + \alpha_1 + \alpha_2)} \text{ and } h_2\asymp n^{-2\alpha_1/(4\alpha_1\alpha_2 + \alpha_1 + \alpha_2)}.
\end{align*}
$C$ and $c$ still represent two generic fixed positive constants whose values may change at each occurrence. 

Define $U_1,U_2, \theta_1, \theta_2$ and the ``good" event $\cal{E}$ the same way as in Theorem \ref{thm:nonpar_upper}. Following its proof, we now lower bound $\theta_2$ and upper bound the term $\abs*{\theta_1 - \theta_2\sigma^2}$. For $\theta_2$, we have
\begin{align*}
\theta_2 &= \E\bbrace*{K_{h_1}(X_{i,1}-X_{j,1})K_{h_2}(X_{i,2}-X_{j,2})}\\
&= \int_{\RR^2} \frac{1}{h_1h_2}K\parr*{\frac{u}{h_1}}K\parr*{\frac{v}{h_2}}p_{\bXX}(u,v)dudv\\
&= \int_{\RR^2} K(u)K(v)p_{\bXX}(uh_1, vh_2)dudv\\
&= \int_{-1}^1 \int_{-1}^1 K(u)K(v)p_{\bXX}(uh_1, vh_2)dudv\\
&\geq \int_{\cal{U}_{(h_1,h_2)}} K(u)K(v)p_{\bXX}(uh_1, vh_2)dudv\\
&\geq \mbf{\lambda}(\cal{U}_{(h_1,h_2)})\inf_{\mbf{u}\in\cal{U}_{(h_1,h_2)}}p_{\mbf{\XX}}(u_1h_1,u_2h_2)\under{M}_K\\
&\geq c_0^2\under{M}_K.
\end{align*}
Here, the fourth equality follows from the kernel condition that $K(\cdot)$ is supported in $[-1,1]$, and the set $\cal{U}_{(h_1,h_2)}$ starting from the first inequality follows from Condition (b) in $\cal{P}_{\text{\tiny{mcv}},(\mbf{X},\varepsilon)}$ since for any fixed $\delta_0 > 0$ chosen therein, $\|\mbf{\delta}\|_\infty\define \|(h_1, h_2)\|_\infty\leq \delta_0$ for sufficiently large $n$. For $\abs*{\theta_1 - \theta_2\sigma^2}$, using the condition $\alpha_i\in(0,1]$, $i=1,2$, it holds that
\begin{align*}
&\ms\abs*{\theta_1 - \theta_2\sigma^2}\\
&= \E\bbrace*{K_{h_1}(\td{X}_{ij,1})K_{h_2}(\td{X}_{ij,2})(f(\mbf{X}_i)-f(\mbf{X}_j))^2/2}\\
& \lesssim\E\bbrace*{\frac{1}{h_1h_2}K\parr*{\frac{\td{X}_{ij,1}}{h_1}}K\parr*{\frac{\td{X}_{ij,2}}{h_2}}\parr*{\abs*{\td{X}_{ij,1}}^{2\alpha_1} + \abs*{\td{X}_{ij,2}}^{2\alpha_2}}}\\
&= \int \frac{1}{h_1h_2}K\parr*{\frac{u}{h_1}}K\parr*{\frac{v}{h_1}}\parr*{\abs*{u}^{2\alpha_1} + \abs*{v}^{2\alpha_2}}p_{\bXX}(u,v)dudv\\
&= h_1^{2\alpha_1}\int K(u)K(v)|u|^{2\alpha_1}p_{\bXX}(uh_1,vh_2)dudv + \\
&\ms h_2^{2\alpha_2}\int K(u)K(v)|v|^{2\alpha_2}p_{\bXX}(uh_1,vh_2)dudv\\
&\leq C_0\parr*{h_1^{2\alpha_1}\int K(u)K(v)|u|^{2\alpha_1}dudv + h_2^{2\alpha_2}\int K(u)K(v)|v|^{2\alpha_2}dudv}\\
&\leq C(h_1^{2\alpha_1} + h_2^{2\alpha_2}),
\end{align*}
where we have applied Condition (a) in $\cal{P}_{\text{\tiny{mcv}},(\mbf{X},\varepsilon)}$ and the compact support of $K(\cdot)$. Therefore, Lemmas \ref{lemma:d_concentrate_1} and \ref{lemma:d_concentrate_2} and the above estimates imply that 
\begin{align*}
\E\bbrace*{\parr*{\frac{U_1 - U_2\sigma^2}{U_2}}^2\mathbbm{1}\{\cal{E}\}} \lesssim (h_1^{4\alpha_1} + h_2^{4\alpha_2} + n^{-1} + n^{-2}(h_1h_2)^{-1}) \asymp (n^{-\frac{8\alpha_1\alpha_2}{4\alpha_1\alpha_2 + \alpha_1 + \alpha_2}} + n^{-1}).
\end{align*}

Moreover, using the same argument as in the proof of Theorem \ref{thm:nonpar_upper}, it holds that
\begin{align*}
\E\bbrace*{\parr*{\frac{U_1 - U_2\sigma^2}{U_2}}^2\mathbbm{1}\{\cal{E}^c\}} = o(n^{-\frac{8\alpha_1\alpha_2}{4\alpha_1\alpha_2 + \alpha_1 + \alpha_2}} + n^{-1}).
\end{align*}
This completes the proof for $d = 2$.
In the case of general dimension $d$ with heterogeneous smoothness index $\mbf{\alpha} = (\alpha_1, \ldots, \alpha_d)^\top$, the upper bound takes the form
\begin{align*}
\E\parr*{\hat{\sigma}_d^2 - \sigma^2}^2\lesssim n^{-1} + n^{-2}\parr*{\prod_{k=1}^dh_k}^{-1} + \sum_{k=1}^d h_k^{4\alpha_k}
\end{align*}
and we choose $h_k \asymp n^{-2\under{\alpha}/(\alpha_k(4\under{\alpha}+d))}$. This completes the proof.
\end{proof}

\subsection{Proof of Proposition \ref{prop:multivariate_lower}}
\begin{proof}
Given the proof of Theorem \ref{thm:nonpar_lower}, the proof here is relatively straightforward. We will thus only present the construction of the hardest sub-problem. For simplicity, we will only prove the case for $d = 2$. We also only consider the regime of $(\alpha_1,\alpha_2)$ in which the lower bound is sub-parametric: $4\alpha_1\alpha_2  < \alpha_1+ \alpha_2$. Throughout the proof, $C$ represents some generic positive constant and does not depend on $n$, and $c$ represents a generic sufficiently small positive constant which also does not depend on $n$. In particular, $c$ is always taken to be smaller than $1$. Both $C$ and $c$ might have different values for each occurrence.

Introduce the following constants:
\begin{align}
\label{eq:constraint}
\theta_n^2 \define h_1^{2\alpha_1} \define h_2^{2\alpha_2} \define cn^{-\frac{4\alpha_1\alpha_2}{4\alpha_1\alpha_2 + \alpha_1 + \alpha_2}}, \quad N_1 \define 1/(6h_1), \quad N_2 \define 1/(6h_2),
\end{align}
where we tune the constant $c$ in $\theta_n^2$ so that $N_1$ and $N_2$ are both positive integers. We now specify $f(\cdot)$, distribution of $\mbf{X}$, $\sigma^2$ and distribution of $\varepsilon$ in the null and alternative hypotheses, $H_0$ and $H_1$, respectively.

\begin{itemize}
\item[]\emph{Choice of $\sigma^2$}: Under $H_0$, let $\sigma^2 = 1 + \theta_n^2$. Under $H_1$, let $\sigma^2 = 1$.
\item[]\emph{Choice of $\varepsilon$}: Under both $H_0$ and $H_1$, let $\varepsilon\sim\cal{N}(0,1)$.
\item[]\emph{Choice of $\mbf{X}$}: Under both $H_0$ and $H_1$, let $\mbf{X}$ be uniformly distributed on the union of the rectangles $[(6i_1-5)h_1, (6i_1-1)h_1]\times [(6i_2-5)h_2, (6i_2-1)h_2]$ for $i_1\in[N_1]$ and $i_2\in[N_2]$.
\item[]\emph{Choice of $f(\cdot)$}: Under $H_0$, let $f\equiv 0$. Under $H_1$, let $f$ be a smooth bump function that takes value $\theta_nr_{i_1,i_2}$ on the rectangle $[(6i_1-5)h_1, (6i_1-1)h_1]\times [(6i_2-5)h_2, (6i_2-1)h_2]$, and then smoothly decays to 0 on the union of the segments $\{x_1 = 6(i_1-1)h_1, 0 \leq x_2\leq 1\}$ for $i_1\in[N_1]$ and $\{0\leq x_1\leq 1, x_2 = 6(i_2-1)h_2\}$ for $i_2\in[N_2]$. Here, the double indexed sequence $\{r_{i_1,i_2}\}_{i_1\in[N_1], i_2\in[N_2]}$ are $N_1\times N_2$ i.i.d. symmetric and compactly supported random variables with distribution $\G$ satisfying
\begin{align*}
\int_{-\infty}^\infty x^j\G(dx) = \int_{-\infty}^\infty x^j\varphi(x)dx, \quad j=1,\ldots, q,
\end{align*}
where $q$ is some odd integer strictly larger than $1+(\alpha_1+\alpha_2)/(2\alpha_1\alpha_2)$.
\end{itemize}
In the definition of $f(\cdot)$ under $H_1$, the existence of the distribution $G$ is guaranteed by Lemma \ref{lemma:moment_matching} and its range only depends on $\alpha_1$ and $\alpha_2$. The smoothness property of $f(\cdot)$ can be achieved by mollifying an indicator function.

We only verify Condition (c) in $\cal{P}_{\text{\tiny{mcv}},(\mbf{X},\varepsilon)}$, which holds by the convolution formula that for any $0\leq u\leq 1/2$ and $0 \leq v\leq 1/2$
\begin{align*}
p_{\bXX}(u,v) &= \int_{u}^1\int_v^1 p_{\mbf{X}}(t_1,t_2)p_{\mbf{X}}(t_1-u, t_2-v)dt_1dt_2\\
&\geq \sum_{i_1=\ceil{u/(6h_1)}+1}^{N_1}\sum_{i_2=\ceil{v/(6h_2)}+1}^{N_2}\int_{(6i_1-5)h_1}^{(6i_1-1)h_1}\int_{(6i_2-5)h_2}^{(6i_2-1)h_2} p_{\mbf{X}}(t_1,t_2)p_{\mbf{X}}(t_1-u, t_2-v)dt_1dt_2\\
&\geq \sum_{i_1=\ceil{u/(6h_1)}+1}^{N_1}\sum_{i_2=\ceil{v/(6h_2)}+1}^{N_2}(2h_1)(2h_2)\cdot\frac{9}{4}\cdot\frac{9}{4}\\
&\geq \frac{9}{256}
\end{align*}
for sufficiently large $n$. A similar calculation holds for all $|u|\leq 1/2$ and $|v|\leq 1/2$. Therefore, Condition (c) holds with $\delta_0 = 1/2$ and $\cal{U}_{\mbf{\delta}}\equiv[-1,1]^2$.
\end{proof}

\subsection{Proof of Proposition \ref{prop:additive_fixed_GD}}
\begin{proof}
We employ an iterative usage of pairwise difference. Under the regular design, we have $Y_{i_1,\ldots,i_d} = \sum_{k=1}^d f_k(i_k / n^{1/d}) + \sigma\varepsilon_{i_1,\ldots,i_d}$ for $(i_1,\ldots,i_d)\in [n^{1/d}]\times\ldots\times[n^{1/d}]$, where we assume without loss of generality that $n^{1/d}$ is an even integer. Let $m \define n^{1/d}/2$ and define $\cal{I} \define \{(1,2),\ldots, (2m-1, 2m)\}$ with cardinality $m$. For all index pairs $(i^{(1)}_k, i^{(2)}_k)\in \cal{I}$, $k\in[d]$, we have
\begin{align*}
Y_{(i^{(1)}_1,i^{(2)}_1),\ldots,(i^{(1)}_d,i^{(2)}_d)} &\define \sum_{j_k\in\{i_k^{(1)},i_k^{(2)}\},k\in[d]} Y_{j_1,\ldots, j_d}(-1)^{\sum_{k=1}^d \mathbbm{1}\{j_k = i_k^{(1)}\}}\\
& = \sum_{j_k\in\{i_k^{(1)},i_k^{(2)}\},k\in[d]} \sigma\varepsilon_{j_1,\ldots, j_d}(-1)^{\sum_{k=1}^d \mathbbm{1}\{j_k = i_k^{(1)}\}}.
\end{align*}
Clearly, we have $\E\parr*{Y_{(i^{(1)}_1,i^{(2)}_1),\ldots,(i^{(1)}_d,i^{(2)}_d)}} = 0$ and $\var\parr*{Y_{(i^{(1)}_1,i^{(2)}_1),\ldots,(i^{(1)}_d,i^{(2)}_d)}} = 2^d\sigma^2$. More importantly, the newly formed data sequence $\{Y_{(i^{(1)}_1,i^{(2)}_1),\ldots,(i^{(1)}_d,i^{(2)}_d)}\}_{(i_k^{(1)},i^{(2)}_k)\in\cal{I}, k\in[d]}$ with cardinality $m^d = n/2^d$ is i.i.d. with mean $0$ and variance $2^d\sigma^2$. Therefore, by defining $\bar{Y}$ to be the average of this newly formed data sequence and $\widehat{\sigma}^2$ to be 
\begin{align*}
\widehat{\sigma}^2 \define \frac{1}{n}\sum_{(i^{(1)}_k, i^{(2)}_k)\in\cal{I}, k\in[d]} \parr*{Y_{(i^{(1)}_1,i^{(2)}_1),\ldots,(i^{(1)}_d,i^{(2)}_d)} - \bar{Y}}^2,
\end{align*}
we clearly have $\E\parr*{\widehat{\sigma}^2 - \sigma^2}^2 \lesssim n^{-1}$ for some absolute positive constant $C$ under a finite fourth moment assumption, which is clearly not improvable. Thus the proof is complete.
\end{proof}

\subsection{Proof of Proposition \ref{prop:additive_random_independent}}
\begin{proof}
Throughout the proof, $C$ represents a positive constant that only depends $\mbf{\alpha}$ and $C_0$. Following the argument before the statement of Proposition \ref{prop:additive_random_independent}, define 
\begin{align*}
\varepsilon_i^{(\ell)} \define \sum_{k\in[d], k\neq \ell}f_k(X_{i,k}) + \varepsilon_i
\end{align*}
and its variance
\begin{align*}
\sigma^2_{(\ell)}\define \sum_{k\in[d],k\neq \ell}\E f^2_k(X_{i,k}) + \sigma^2
\end{align*}
for all $\ell\in[d]$. Clearly, under the mutual independence of the components of $(X_{i,1},\ldots, X_{i,d})$, it holds that $\E\varepsilon_i^{(\ell)} = 0$ and $\varepsilon_i^{(\ell)}$ is independent of $f_\ell(X_{i,\ell})$. For each $\ell\in[d]$, by viewing the model equivalently as $Y_i = f_{\ell}(X_{i,\ell}) + \varepsilon^{(\ell)}_i$ for $i\in[n]$ and then applying the univariate kernel smoother defined in \eqref{eq:nonpar_estimator}, which renders an estimator which we denote as $\widehat{\sigma}^2_{(\ell)}$, we obtain by Theorem \ref{thm:nonpar_upper}
\begin{align*}
\E\parr*{\widehat{\sigma}^2_{(\ell)} - \sigma^2_{(\ell)}}^2 = \E\parr*{\widehat{\sigma}^2_{(\ell)} - \sum_{k\in[d],k\neq \ell}\E f_k^2(W_{i,k}) - \sigma^2}^2 \leq Cn^{-\frac{8\alpha_\ell}{4\alpha_\ell+1}}.
\end{align*}
Moreover, letting $\bar{Y}$ be the average of $\{Y_i\}_{i=1}^n$ and $\widehat{\sigma}^2_Y$ be the sample variance estimator, that is, $\widehat{\sigma}^2_Y \define \sum_{i=1}^n (Y_i - \bar{Y})^2/n$, it holds that
\begin{align*}
\E\parr*{\widehat{\sigma}^2_Y - \var(Y)}^2 = \E\parr*{\widehat{\sigma}^2_Y - \sum_{k\in[d]}\E f_k^2(X_{i,k}) - \sigma^2}^2 \leq Cn^{-1}.
\end{align*}
Since $\sigma^2 = \sum_{\ell=1}^d \sigma^2_{(\ell)} - (d-1)\var(Y)$, thus by defining $\hat{\sigma}^2 \define \sum_{\ell=1}^d \widehat{\sigma}^2_{(\ell)} - (d-1)\widehat{\sigma}^2_Y$, we obtain that
\begin{align*}
\E\parr*{\widehat{\sigma}^2 - \sigma^2}^2 &= \E\bbrace*{\parr*{\widehat{\sigma}^2_{(1)} - \sigma^2_{(1)}} + \ldots + \parr*{\widehat{\sigma}^2_{(d)} - \sigma^2_{(d)}} + (d-1)\parr*{\widehat{\sigma}^2_Y - \var(Y)}}^2\\
&\leq C\parr*{n^{-\frac{8\alpha_{\min}}{4\alpha_{\min}+1}} + n^{-1}}.
\end{align*}
This completes the proof.
\end{proof}


\subsection{Proof of Proposition \ref{prop:additive_random_known}}
\begin{proof}
For simplicity, we only prove the case with two additive components $f(X)$ and $g(W)$ which are $\alpha$- and $\beta$-H\"older smooth, respectively. Throughout the proof, $C$ represents a generic fixed positive constant that only depends on $\alpha,\beta, C_0$ and the joint distribution of $(X,W)$. Denote the marginal distribution of $X$ and $W$ as $F_X$ and $F_W$. Since the transition boundary for both $\alpha$ and $\beta$ is $1/4$, we may assume without of loss of generality that $0 < \alpha, \beta < 1$. As a result, since $F^{-1}_X$ and $F^{-1}_W$ are both Lipschitz with fixed positive constants, $\bar{f}\define f\circ F^{-1}_X$ and $\bar{g}\define g\circ F^{-1}_W$ are still $\alpha$- and $\beta$-H\"older smooth. With a standard wavelet expansion (cf. Proposition 2.5 in \cite{meyer1990ondelettes}), we can write the model equivalently as
\begin{align*}
Y_i = \bar{f}_1(U_{1,i}) + \sum_{j=1}^{2^{J_1}} \psi_j(U_{1,i})\gamma_{1,j} + \bar{g}_1(U_{2,i}) + \sum_{j=1}^{2^{J_2}} \varphi_j(U_{2,i})\gamma_{2,j} + \sigma\varepsilon_i,
\end{align*}
where $\{\psi_j\}_{j=1}^{\infty}$ and $\{\varphi_j\}_{j=1}^{\infty}$ are two sets of orthonormal wavelet basis (with respect to the Lebesgue measure on $[0,1]$), $\{U_{1,i}\}_{i=1}^n = \{F_X(X_i)\}_{i=1}^n$ and $\{U_{2,i}\}_{i-=1}^n = \{F_W(W_i)\}_{i=1}^n$ are two uniform $[0,1]$ sequences, and $\|\bar{f}_1\|_\infty \leq C(2^{-\alpha J_1})$ and $\|\bar{g}_1\|_\infty \leq C(2^{-\beta J_2})$. Define $\mbf{U}_i \define (\psi_1(U_{1,i}), \ldots, \psi_{2^{J_1}}(U_{1,i}), \varphi_1(U_{2,i}),\ldots, \varphi_{2^{J_2}}(U_{2,i}))$ as the new feature vector of length $2^{J_1} + 2^{J_2}$. Without loss of generality, we assume $\E\mbf{U}_i = 0$ (a mean shift does not affect the estimation of variance) and let $\mbf{\Sigma} \define \cov(\mbf{U}_i)$. Without loss of generality, we can assume $\mbf{\Sigma}$ is strictly positive definite (otherwise we can orthogonalize with respect to the linear span of $(\psi_1(U_{1,i}), \ldots, \psi_{2^{J_1}}(U_{2,i}), \varphi_1(U_{1,i}),\ldots, \varphi_{2^{J_2}}(U_{2,i}))$ in \eqref{eq:additive_new_model} below), and thus it holds that 
\begin{align}
\label{eq:additive_new_model}
Y_i = \mbf{V}_i^\top \mbf{\gamma} + \bar{f}_1(U_{1,i}) + \bar{g}_1(U_{2,i}) + \sigma\varepsilon_i,
\end{align}
where $\mbf{\gamma} \define \mbf{\Sigma}^{1/2}(\gamma_{1,1},\ldots,\gamma_{1,2^{J_1}},\gamma_{2,1},\ldots,\gamma_{2,2^{J_2}})$, and $\mbf{V}_i\define \mbf{\Sigma}^{-1/2}\mbf{U}_i$.

We now calculate the bias and variance of the estimator $\widehat{\sigma}^2_{\tiny{\text{proj}},\tiny{\text{add}}}$ defined as
\begin{align*}
\widehat{\sigma}^2_{\tiny{\text{proj}},\tiny{\text{add}}} \define \frac{1}{n-1}\sum_{i=1}^n(Y_i - \bar{Y})^2 - {n\choose 2}^{-1}\sum_{i<j}Y_iY_j\mbf{V}_i^\top\mbf{V}_i.
\end{align*}
Direct calculation shows that
\begin{align*}
\abs*{\widehat{\sigma}^2_{\tiny{\text{proj}},\tiny{\text{add}}} - \sigma^2} &= \abs*{\E \bar{f}_1^2(U_1) + \E \bar{g}_1^2(U_2) + 2\E\parr*{\bar{f}_1(U_1)\bar{g}_1(U_2)} - \nm*{\E\parr*{(\bar{f}_1(U_1) - \bar{g}_1(U_2))\mbf{V}}}_2^2}\\
&\lesssim \E \bar{f}_1^2(U_1) + \E \bar{g}_1^2(U_2) + \nm*{\E\parr*{\bar{f}_1(U_1)\mbf{V}}}_2^2 + \nm*{\E\parr*{\bar{g}_1(U_2)\mbf{V}}}_2^2\\
&\lesssim 2^{-2\alpha J_1} + 2^{-2\beta J_2} +  \nm*{\E\parr*{\bar{f}_1(U_1)\mbf{V}}}_2^2 + \nm*{\E\parr*{\bar{g}_1(U_2)\mbf{V}}}_2^2.
\end{align*} 
Moreover, we have
\begin{align*}
\nm*{\E\parr*{\bar{f}_1(U_1)\mbf{V}}}_2^2 = \sup_{\|\mbf{a}\|\leq 1} \parr*{\E\parr*{\bar{f}_1(U_1)\mbf{V}^\top \mbf{a}}}^2 \leq \E \bar{f}_1^2(U_1)\cdot\sup_{\|\mbf{a}\|\leq 1} \bbrace*{\mbf{a}^\top\E(\mbf{V}\mbf{V}^\top)\mbf{a}} \lesssim 2^{-2\alpha J_1},
\end{align*}
where the last inequality again follows by the identity covariance of $\mbf{V}$. Similarly, it holds that $\nm*{\E\parr*{\bar{g}_1(U_2)\mbf{V}}}_2^2 \lesssim 2^{-2\beta J_2}$. We therefore conclude that the bias of $\widehat{\sigma}^2_{\tiny{\text{proj}},\tiny{\text{add}}}$ is smaller than the order $2^{-2\alpha J_1} + 2^{-2\beta J_2}$. 

Next, we calculate the variance of $\widehat{\sigma}^2_{\tiny{\text{proj}},\tiny{\text{add}}}$. For this, it suffices to upper bound the variance of $\sum_{i=1}^n (Y_i - \bar{Y})^2/(n-1)$ and ${n\choose 2}^{-1}\sum_{i<j}Y_iY_j\mbf{V}_i^\top\mbf{V}_j$. The first variance is clearly of the order $n^{-1}$ under the boundedness of $f(\cdot)$ and $g(\cdot)$ and the fact that $\E\varepsilon_i^4\leq C_\varepsilon$, so we focus on the second variance. Direct calculation shows that
\begin{align}
\label{eq:known_variance}
&\ms\var\parr*{\sum_{i<j}Y_iY_j\mbf{V}_i^\top\mbf{V}_j}\notag\\
&= \sum_{i<j,i^\prime<j^\prime} \parr*{\E\parr*{Y_iY_jY_{i^\prime}Y_{j^\prime}(\mbf{V}_i^\top\mbf{V}_j)(\mbf{V}_{i^\prime}^\top\mbf{V}_{j^\prime})} - \E\parr*{Y_iY_j\mbf{V}_i^\top\mbf{V}_j}\E\parr*{Y_{i^\prime}Y_{j^\prime}\mbf{V}_{i^\prime}^\top\mbf{V}_{j^\prime}}}.
\end{align} 
When $i,j,i^\prime,j^\prime$ take four different values, the above summand is clearly $0$. When they take $3$ values ($i = i^\prime < j < j^\prime$), denoting $z_i \define \bar{f}_1(U_{1,i}) + \bar{g}_1(U_{2,i})$, $i\in[n]$, we have
\begin{align*}
&\ms\E\parr*{Y_i^2Y_jY_{j^\prime}(\mbf{V}_i^\top\mbf{V}_j)(\mbf{V}_i^\top\mbf{V}_{j^\prime})}\\
& = \E\parr*{(z_i + \mbf{V}_i^\top\mbf{\gamma} + \sigma\varepsilon_i)^2(z_j + \mbf{V}_j^\top\mbf{\gamma} + \sigma\varepsilon_j)(z_{j^\prime} + \mbf{V}_{j^\prime}^\top\mbf{\gamma} + \sigma\varepsilon_{j^\prime})(\mbf{V}_i^\top\mbf{V}_j)(\mbf{V}_i^\top\mbf{V}_{j^\prime})}.
\end{align*}
Next, we expand the above display and upper bound each term individually. For simplicity, we only show the calculation for the following dominating term and the other terms follow similarly. 
\begin{align*}
&\ms\E\parr*{(\mbf{V}_i^\top\gamma)^2(\mbf{V}_j^\top\mbf{\gamma})(\mbf{V}_{j^\prime}^\top\mbf{\gamma})(\mbf{V}_i^\top\mbf{V}_j)(\mbf{V}_i^\top\mbf{V}_{j^\prime})}\\
&= \E\parr*{(\mbf{V}_j^\top\mbf{\gamma})\mbf{V}_j^\top\parr*{(\mbf{V}_i^\top\mbf{\gamma})^2\mbf{V}_i\mbf{V}_i^\top}\mbf{V}_{j^\prime}(\mbf{V}_{j^\prime}^\top\mbf{\gamma})}\\
&= \E\parr*{\mbf{\gamma}^\top\mbf{V}_j\mbf{V}_{j}^\top}\E\parr*{(\mbf{V}_i^\top\mbf{\gamma})^2\mbf{V}_i\mbf{V}_i^\top}\E\parr*{\mbf{V}_{j^\prime}\mbf{V}_{j^\prime}^\top\mbf{\gamma}}\\
&= \mbf{\gamma}^\top\E\parr*{(\mbf{V}_i^\top\mbf{\gamma})^2\mbf{V}_i\mbf{V}_i^\top}\mbf{\gamma}\\
&= \E (\mbf{V}_i^\top\mbf{\gamma})^4,
\end{align*}
where in the second equality we use the independence of $\mbf{V}_i,\mbf{V}_j,\mbf{V}_{j^\prime}$ and in the third equality we use the fact that $\E(\mbf{V}_i^\top\mbf{V}_i) = \cov(\mbf{V}_i) = \mathbf{I}_L$ since $\E\mbf{V}_i = \mbf{\Sigma}^{-1/2}\E\mbf{U}_i = 0$, where $L\define 2^{J_1} + 2^{J_2}$. Moreover, by definition, it holds that $\abs*{\mbf{V}_i^\top\mbf{\gamma}} = \abs*{f(X_i) + g(W_i) - z_i} \leq \|f\|_\infty + \|g\|_\infty + |z_i| \leq C$ due to the boundedness of $f(\cdot),g(\cdot)$ and $|z_i|$. This concludes that when $i,j,i^\prime,j^\prime$ take three different values, the summand in \eqref{eq:known_variance} can be upper bounded by a fixed constant.
When they take two different values ($i = i^\prime, j = j^\prime$), we have
\begin{align*}
\E\parr*{Y_i^2Y_j^2(\mbf{V}_i^\top\mbf{V}_j)^2} = \E\parr*{(z_i + \mbf{V}_i^\top\mbf{\gamma} + \sigma\varepsilon_i)^2(z_j + \mbf{V}_j^\top\mbf{\gamma} + \sigma\varepsilon_j)^2(\mbf{V}_i^\top\mbf{V}_j)^2}.
\end{align*}
We again upper bound the dominating term in the expansion of the above display.
\begin{align*}
\E\parr*{(\mbf{V}_i^\top\mbf{\gamma})^2(\mbf{V}_j^\top\mbf{\gamma})^2(\mbf{V}_i^\top\mbf{V}_j)^2} \leq C\E(\mbf{V}_i^\top\mbf{V}_j)^2 = C\tr(\mathbf{I}_L) = CL,
\end{align*}
where we again use the fact that $(\mbf{V}_i^\top\mbf{\gamma})$ and $(\mbf{V}_j^\top\mbf{\gamma})$ are bounded by a fixed constant. Putting together the pieces, we obtain that
\begin{align*}
\var\parr*{\widehat{\sigma}^2_{\tiny{\text{proj}},\tiny{\text{add}}}} \leq C(n^{-1} + n^{-4}(n^3 + n^2L)) = C\frac{n + 2^{J_1} + 2^{J_2}}{n^2}.
\end{align*}
Optimal choice of $2^{J_1} \asymp n^{2/(4\alpha+1)}$ and $2^{J_2}\asymp n^{2/(4\beta +1)}$ then gives the desired error bound.
\end{proof}

\subsection{Proof of Proposition \ref{prop:cv_upper}}
\begin{proof}
We use $\cal{H}_*$ as a shorthand for $\cal{H}_{\delta^*}$. Fix any $\alpha\geq \alpha_*$. Define the oracle bandwidth 
\begin{align*}
h^*\define
\begin{cases}
n^{-1}, & \alpha > 1/4,\\
\max\bbrace*{h\in\cal{H}_*: h^{2\alpha}\leq c\sqrt{\log n/(n^2h)}}, & 0<\alpha\leq 1/4,
\end{cases}
\end{align*}
 for some positive constant $c$ to be specified later. When $0<\alpha\leq 1/4$, $h^*$ is taken to be $n^{-2/(4\alpha+1)}$ if the set being maximized is empty. If not, then it holds that $(2h^*)^{2\alpha}> c\sqrt{\log n/(2n^2h^*)}$, and thus $(h^*)^{2\alpha}\asymp \sqrt{\log n/(n^2h^*)}$, or $h^*\asymp (\log n/n^2)^{1/(4\alpha+1)}$.

We first prove that with high probability, we have $\widehat{h}_{\delta_*}\geq h^*$. For this, we have
\begin{align*}
\P\parr*{\widehat{h}_{\delta_*}<h^*} &\leq \P\parr*{\exists h\in\cal{H}_*, h\leq h^*, \abs*{\wh{\sigma}^2(h) - \wh{\sigma}^2(h^*)}\geq \tau\sqrt{\log n/(n^2h)}}\\
&\leq \sum_{h\in\cal{H}_*,h\leq h^*} \P\parr*{\abs*{\wh{\sigma}^2(h) - \wh{\sigma}^2(h^*)}\geq \tau\sqrt{\log n/(n^2h)}}\\
&\leq \sum_{h\in\cal{H}_*,h\leq h^*}\P\parr*{\abs*{\wh{\sigma}^2(h) - \sigma^2}\geq \frac{\tau}{2}\sqrt{\log n/(n^2h)}} + \abs*{\cal{H}_*}\cdot\P\parr*{\abs*{\wh{\sigma}^2(h^*) - \sigma^2}\geq \frac{\tau}{2}\sqrt{\log n/(n^2h^*)}}.
\end{align*}
We now upper bound each probability in the above summation for any $h\leq h^*$. As in the proof of Theorem \ref{thm:nonpar_upper}, denote the two U-statistics on the numerator and denominator of $\wh{\sigma}^2(h)$ as $U_1$,$U_2$, with corresponding mean values $\theta_1,\theta_2$. That is, 
\begin{align*}
\theta_1\define \E\bbrace*{K_h(X_i - X_j)(Y_i-Y_j)^2/2} \text{ and } \theta_2 \define \E K_h(X_i-X_j). 
\end{align*}
Define the ``good" event $\cal{E}\define \bbrace*{U_2\geq \theta_2/2}$ and $\cal{E}^c$ as its complement, then it holds that
\begin{align*}
&\ms\P\parr*{\abs*{\wh{\sigma}^2(h) - \sigma^2}\geq \frac{\tau}{2}\sqrt{\log n/(n^2h)}\bigcap\cal{E}}\\
&\leq \P\parr*{\abs*{U_1 - U_2\sigma^2}\geq \frac{\tau}{\theta_2}\sqrt{\log n/(n^2h)}}\\
&= \P\parr*{\abs*{(U_1-\theta_1) + (\theta_1 - \theta_2\sigma^2) + (U_2-\theta_2)}\geq \frac{\tau}{\theta_2}\sqrt{\log n/(n^2h)}}\\
&\leq \P\parr*{\abs*{(U_1 - \theta_1)}\geq \frac{\tau}{4\theta_2}\sqrt{\log n/(n^2h)}} + \P\parr*{\abs*{(U_2 - \theta_2)}\geq \frac{\tau}{4\theta_2}\sqrt{\log n/(n^2h)}},
\end{align*}
where the last inequality follows from the fact $h\leq h^*$ and the bound $\abs*{\theta_1 - \theta_2\sigma^2}\lesssim h^{2(\alpha\wedge 1)}$ calculated in the proof Theorem \ref{thm:nonpar_upper}. By choose $u\asymp \log n$ and $v\asymp \log n/(nh)$ in Lemma \ref{lemma:concentrate_U2} and $\tau$ to be sufficiently large, it holds that
\begin{align*}
\P\parr*{\abs*{U_2 - \theta_2}\geq \frac{\tau}{4\theta_2}\sqrt{\log n/(n^2h)}} \lesssim n^{-C}
\end{align*}
for arbitrarily large $C$. Furthermore, for sufficiently large $\tau$ and $\eta$ in Lemma \ref{lemma:exp_U1} below, choosing the same $u$ and $v$ yields that
\begin{align*}
\P\parr*{\abs*{U_1 - \theta_1}\geq \frac{\tau}{4\theta_2}\sqrt{\log n/(n^2h)}} \lesssim n^{-C}
\end{align*}
for arbitrarily large $C$. This, combined with the calculation 
\begin{align*}
\P\parr*{\cal{E}^c}\lesssim \exp(-\theta_2^2n/16)+\exp(-\theta_2^2n^2h/16)
\end{align*}
in the proof of Theorem \ref{thm:nonpar_upper}, concludes that $\P(\td{\cal{E}}^c)\lesssim n^{-C}$, where $\td{\cal{E}}\define \bbrace*{\widehat{h}_{\delta_*}<h^*}$. Therefore, we have
\begin{align*}
\E\parr*{\widehat{\sigma}_{\text{\tiny{adapt}}}^2 - \sigma^2}^2 &\lesssim \E\bbrace*{\parr*{\widehat{\sigma}_{\text{\tiny{adapt}}}^2 - \sigma^2}^2\mathbbm{1}\bbrace*{\td{\cal{E}}^c}} + \E\bbrace*{\parr*{\widehat{\sigma}_{\text{\tiny{adapt}}}^2 - \wh{\sigma}^2(h^*)}^2\mathbbm{1}\bbrace*{\td{\cal{E}}}} + \E\bbrace*{\parr*{\wh{\sigma}^2(h^*) - \sigma^2}^2}\\
&\lesssim n^{-C} +  \frac{\log n}{n^2h^*} + \parr*{n^{-1} + (h^*)^{4(\alpha\wedge 1)}+ (n^2h^*)^{-1}}\\
&\lesssim \parr*{\frac{\log n}{n^2}}^{4\alpha/(4\alpha+1)} + n^{-1}.
\end{align*}
This completes the proof.
\end{proof}

\subsection{Proof of Proposition \ref{prop:cv_lower}}
\begin{proof}

Note that the desired result is equivalent to the following statement:
\begin{align*}
\inf_{\td{\sigma}^2}\max\bbrace*{\sup_{\substack{f\in\Lambda_{\alpha_1}(C_\cal{F})\\\sigma^2\leq C_\sigma,\P_{(X,\varepsilon)}\in\cal{P}^{\text{\tiny{adapt}}}_{\text{\tiny{cv}}, (X,\varepsilon)}}}\E\parr*{(\td{\sigma}^2 - \sigma^2)/\phi_{n,\alpha_1}}^2,\sup_{\substack{f\in\Lambda_{\alpha_2}(C_\cal{F})\\\sigma^2\leq C_\sigma,\P_{(X,\varepsilon)}\in\cal{P}^{\text{\tiny{adapt}}}_{\text{\tiny{cv}}, (X,\varepsilon)}}}\E\parr*{(\td{\sigma}^2 - \sigma^2)/\phi_{n,\alpha_2}}^2}\geq c
\end{align*}
for sufficiently large $n$ and sufficiently small $c$. By applying Lemma \ref{lemma:poisson} with $\cal{A}=\{\alpha_1,\alpha_2\}$ with $\alpha_*\leq \alpha_1<\alpha_2$, it suffices to lower bound the adaptive minimax rate under measure $\td{\P}$ defined therein. More precisely, we will prove that for $n\geq n_0$ with some sufficiently large $n_0$,
\begin{align*}
\inf_{\td{\sigma}^2}\max\bbrace*{\sup_{\substack{f\in\Lambda_{\alpha_1}(C_{\cal{F}})\\\sigma^2\leq C_\sigma,\P_{(X,\varepsilon)}\in\cal{P}^{\text{\tiny{adapt}}}_{\text{\tiny{cv}}, (X,\varepsilon)}}}\E_{\td{\P}}\parr*{(\td{\sigma}^2 - \sigma^2)/\phi_{n,\alpha_1}}^2,\sup_{\substack{f\in\Lambda_{\alpha_2}(C_{\cal{F}})\\\sigma^2\leq C_\sigma,\P_{(X,\varepsilon)}\in\cal{P}^{\text{\tiny{adapt}}}_{\text{\tiny{cv}}, (X,\varepsilon)}}}\E_{\td{\P}}\parr*{(\td{\sigma}^2 - \sigma^2)/\phi_{n,\alpha_2}}^2} > c
\end{align*} 
for some sufficiently small positive $c$. In order to show this, we will prove that, for any $n\geq n_0$ and any estimator $\td{\sigma}^2$, if 
\begin{align}
\label{eq:fast}
\sup_{\sigma^2\leq C_\sigma}\sup_{f\in\Lambda_{\alpha_2}(C_{\cal{F}})}\sup_{\P_{(X,\varepsilon)}\in\cal{P}^{\text{\tiny{adapt}}}_{\text{\tiny{cv}}, (X,\varepsilon)}}\E_{\td{\P}}\parr*{(\td{\sigma}^2 - \sigma^2)/\phi_{n,\alpha_2}}^2 \leq c,
\end{align}
then it holds that
\begin{align}
\label{eq:slow}
\sup_{\sigma^2\leq C_\sigma}\sup_{f\in\Lambda_{\alpha_1}(C_{\cal{F}})}\sup_{\P_{(X,\varepsilon)}\in\cal{P}^{\text{\tiny{adapt}}}_{\text{\tiny{cv}}, (X,\varepsilon)}}\E_{\td{\P}}\parr*{(\td{\sigma}^2 - \sigma^2)/\phi_{n,\alpha_1}}^2 > c.
\end{align}
If $\alpha_2>1/4$, then $\phi_{n,\alpha_2}\asymp n^{-1/2}$, and we can choose a sufficiently small $c$ such that \eqref{eq:fast} never holds. Therefore, in what follows, we will assume $\alpha_*\leq\alpha_1<\alpha_2\leq 1/4$, in which case $\phi_{n,\alpha_i}\asymp (\log n/n^2)^{2\alpha_i/(4\alpha_i+1)}$ for $i=1,2$. 

We will now apply Lemma \ref{lemma:constrained_risk}. To this end, we adopt a two-point method and introduce the two probability measures $\td{\P}_0$ and $\td{\P}_1$ (conditioning on a fixed realization of $m\sim \text{Poi}(2n)$) as follows. Introduce
\begin{align*}
h_n = c\parr*{\frac{\log n}{n^2}}^{1/(4\alpha_1+1)} \text{ and } N\define N_n\define h_n^{-1},
\end{align*}
where $c$ is some sufficiently small constant tuned such that $N$ is a positive integer. Let $Q$ be a discrete distribution that takes value $0$ with probability $1/2$, $-1$ with probability $1/4$ and $1$ with probability $1/4$. It then can be readily checked that $m_1(Q) = 0$ and $m_2(Q) = 1/2$.

\begin{itemize}
\item[] Choice of $\varepsilon$: Under $H_0$, let $\varepsilon\sim (1+h_n^{2\alpha_1}/2)^{-1/2}((h_n^{\alpha_1} Q)*\cal{N}(0,1))$. Under $H_1$, let $\varepsilon\sim \cal{N}(0,1)$. 
\item[] Choice of $\sigma^2$: Under $H_0$, let $\sigma^2 = 1 + h_n^{2\alpha_1}/2$. Under $H_1$, let $\sigma^2 = 1$.
\item[]\textit{Choice of $X$}: Under both $H_0$ and $H_1$, let $X$ be uniformly distributed on the union of the intervals $[(6i-5)h_n, (6i-1)h_n]$ for $i\in[N]$.
\item[] Choice of $f(\cdot)$: Under $H_0$, let $f\equiv 0$. Under $H_1$, let $f$ take the value $h_n^{\alpha_1} r_i$ on $[(6i-5)h_n, (6i-1)h_n]$, where $\{r_i\}_{i=1}^N$ are $N$ i.i.d. variables with law $Q$. 
\end{itemize}
Clearly, by the boundedness of $Q$, $H_0$ belongs to the model class indexed by the smoothness index $\alpha_2$, and $H_1$ belongs to the model class indexed by $\alpha_1$. Moreover, the absolute difference in $\sigma^2$ under $H_0$ and $H_1$ is lower bounded by the order $h_n^{2\alpha_1}\asymp (\log n/n^2)^{2\alpha_1/(4\alpha_1+1)}$.

Denote $\td{p}_0$ and $\td{p}_1$ as the densities of $\td{\P}_0$ and $\td{\P}_1$. Define $f_{\max}\define \ceil{2/(1-4\alpha_1)}+1$ and let $d_i$ be the number of $X$'s that fall into $[(6i-5)h_n, (6i-1)h_n]$ for each $i\in [N]$. Consider the following event 
\begin{align*}
\cal{E}\define \bbrace*{\{m\leq 3n\}\bigcap \{\max_{1\leq i\leq N}d_i \leq f_{\max}\}}.
\end{align*}
Note that under both $\td{\P}_0$ and $\td{\P}_1$, the sequence $\{d_i\}_{i=1}^N$ are i.i.d. Poisson variables with mean $2n/N$. Thus, a standard Poisson tail estimate and Lemma \ref{lemma:finite} imply that the event $\cal{E}$ has asymptotic probability $1$ under both $\td{\P}_0$ and $\td{\P}_1$. Next, we calculate the $\chi^2$ distance between $\td{\P}_0$ and $\td{\P}_1$ conditioning on the event $\cal{E}$. First, we have
\begin{align*}
\int \frac{\td{p}_1^2}{\td{p}_0}\mathbbm{1}\bbrace*{\cal{E}} &= \int p(m)p(x_1,\ldots,x_m)\mathbbm{1}\bbrace*{\cal{E}}\int \prod_{j=1}^N\frac{p_{1,j}^2}{p_{0,j}}\\
&= \int p(m)p(x_1,\ldots,x_m)\mathbbm{1}\bbrace*{\cal{E}}\int \prod_{j=1}^N \parr*{1 + \chi^2\parr*{p_{1,j}, p_{0,j}}},
\end{align*}
where $p(m)$ and $p(x_1,\ldots,x_m)$ are the pmf and pdf of $m$ and $\{X_i\}_{i=1}^m$ under both $\td{\P}_0$ and $\td{\P}_1$, and $p_{0,j}$ is the conditional density of all those $Y_i$'s with corresponding $X_i\in[(6j-5)h_n, (6j-1)h_n]$ and similarly for $p_{1,j}$. 

We now upper bound each $\chi^2(p_{0,j}, p_{1,j})$, where we assume that there are $d_j$ $X_i$'s that belong to $[(6j-5)h_n, (6j-1)h_n]$. Write $d$ instead of $d_j$ for short. Here, $d$ is a random variable that only depends on $m$ and $\{X_i\}_{i=1}^m$, and on the event $\cal{E}$, we have $d\leq f_{\max}$. Clearly, if $d = 0$ or $1$, then $\chi^2(p_{0,j}, p_{1,j}) = 0$. Assume $d\geq 2$. Assume for simplicity that the $d$ data points are $y_1,\ldots,y_d$. Then, by definition of $\td{\P}_0$, we have
\begin{align*}
\td{p}_{0,j} \geq (1/2)^{f_{\max}}\varphi(y_1)\ldots\varphi(y_d)\geq c\varphi(y_1)\ldots\varphi(y_d). 
\end{align*}
On the other hand, we have by direct calculation 
\begin{align*}
\int \td{p}_{0,j}^2/(\varphi(y_1)\ldots\varphi(y_d)) &= \E_{\substack{s_1,\ldots,s_d\sim h^{\alpha_1} Q\\ \td{s}_1,\ldots,\td{s}_d\sim h^{\alpha_1} Q}}\exp\parr*{\sum_{i=1}^d s_i\td{s}_i},\\
\int \td{p}_{1,j}^2/(\varphi(y_1)\ldots\varphi(y_d)) &= \E_{t,\td{t}\sim h^{\alpha_1} Q}\exp(dt\td{t}),\\
\int \td{p}_{0,j}\td{p}_{1,j}/(\varphi(y_1)\ldots\varphi(y_d)) &= \E_{t,s_1,\ldots,s_d\sim h^{\alpha_1} Q} \exp\parr*{\sum_{i=1}^d ts_i}.
\end{align*}
We therefore conclude that
\begin{align*}
\chi^2(p_{0,j}, p_{1,j}) &\lesssim \sum_{k=1}^\infty \frac{1}{k!}\sum_{1\leq i_1,\ldots,i_k\leq d}\parr*{\E_{s_1,\ldots,s_d\sim h^{\alpha_1} Q}\parr*{s_{i_1}\ldots s_{i_k}} - \E_{t\sim h^{\alpha_1} Q}t^k}^2\\
&\define \sum_{k=1}^\infty \frac{1}{k!}\sum_{1\leq i_1,\ldots,i_k\leq d}\Delta_{i_1,\ldots,i_k}^2.
\end{align*}
For any $k\geq 1$, if $(i_1,\ldots,i_k)$ are all identical, then $\Delta_{i_1,\ldots,i_k} = 0$. More generally, if $(i_1,\ldots,i_k)$ take $\ell$ different values, where $\ell \leq d\leq f_{\max}$ on the event $\cal{E}$, there are ${d\choose \ell}$ ways of choosing $\ell$ different values among $[d]$, and there are a total of ${k-1\choose \ell-1}$ ways to distribute $\ell$ values in $(i_1,\ldots,i_k)$, thus we obtain the estimate
\begin{align*}
\sum_{k=1}^\infty \frac{1}{k!}\sum_{1\leq i_1,\ldots,i_k\leq d}\Delta_{i_1,\ldots,i_k}^2 &\leq \sum_{k=2}^\infty \frac{1}{k!}\sum_{\ell=2}^k h^{2\alpha_1 k}{d\choose \ell}{k-1\choose \ell-1}\\
&= \sum_{\ell=2}^{f_{\max}}{d\choose \ell}\sum_{k=\ell}^\infty \frac{1}{k!}h^{2\alpha_1 k}{k-1\choose \ell-1}.
\end{align*}
For each $2\leq \ell\leq f_{\max}$, by the Stirling's formula, we have
\begin{align*}
\sum_{k=\ell}^\infty \frac{1}{k!}h^{2\alpha_1 k}{k-1\choose \ell-1} \lesssim \sum_{k=\ell}^\infty \frac{h^{2\alpha_1 k}}{k^{k+1/2}e^{-k}}(\frac{ek}{\ell})^k \lesssim \sum_{k=\ell}^\infty \frac{(e^2h^{2\alpha_1})^k}{k^{1/2}\ell^k}\lesssim h^{2\alpha_1\ell}.
\end{align*}
Next, using the trivial bound ${d\choose \ell}\leq d^2f_{\max}^{\ell-2}$, we obtain that $\chi^2\parr*{\td{p}_{0,j}, \td{p}_{1,j}}\lesssim d^2h^{4\alpha_1}$ if $d\geq 2$. This implies that
\begin{align*}
\int \frac{\td{p}_1^2}{\td{p}_0}\mathbbm{1}\bbrace*{\cal{E}} \lesssim \int p(m)p(x_1,\ldots,x_m)\prod_{j=1}^N \parr*{1+\chi^2_j},
\end{align*}
where $\chi^2_j = 0$ for $d_j = 0,1$ and $\chi^2_j\leq d_j^2h^{4\alpha_1}$ for $d_j\geq 2$. Next, 
\begin{align*}
\prod_{j=1}^N \parr*{1+\chi^2_j} = 1 + \sum_{j=1}^N \chi^2_j + \sum_{1\leq i< j\leq N}\chi^2_i\chi^2_j + \ldots + \sum_{1\leq i_1<\ldots<i_N\leq N}\chi^2_{i_1}\ldots\chi^2_{i_N}.
\end{align*}
Consider the $k$th term in the above display. Note that on the event $\cal{E}$, we have $\sum_{j=1}^N d_j\leq 3n$, thus there are at most $(3n/2)$ $j$'s with $d_j\geq 2$. Then, using the estimate ${n\choose k}\leq n^k/k!$, we have
\begin{align*}
\sum_{1\leq i_1<\ldots<i_k\leq N}\chi^2_{i_1}\ldots\chi^2_{i_k} \leq {(3n/2)\choose k}h^{4k\alpha_1}d^2_{i_1}\ldots d^2_{i_k} \leq \frac{C^kn^k}{k!}h^{4k\alpha_1}d^2_{i_1}\ldots d^2_{i_k}.
\end{align*}
We therefore conclude that
\begin{align*}
\int \frac{\td{p}_1^2}{\td{p}_0}\mathbbm{1}\bbrace*{\cal{E}} \leq \E_{d_1,\ldots,d_N\sim \text{Poi}(2n/N)}\parr*{1+\sum_{k=1}^{N}\frac{C^kn^kh^{4\alpha_1 k}}{k!}d_1^2\ldots d_k^2} \leq 1 + \sum_{k=1}^\infty \frac{(c\log n)^k}{k!} = n^c
\end{align*}
for some sufficiently small constant $c$. Then, by Lemma \ref{lemma:constrained_risk} and the calculation
\begin{align*}
\frac{\varepsilon\sqrt{I_\cal{E}}}{\Delta} \lesssim \frac{n^{c/2}(\log n/n^2)^{4\alpha_2/(4\alpha_2+1)}}{(\log n/n^2)^{4\alpha_1/(4\alpha_1+1)}}\rightarrow 0,
\end{align*}
where $\Delta,I,\varepsilon$ are defined as in Lemma \ref{lemma:constrained_risk}, we conclude that for sufficiently large $n$ and any considered estimator $\td{\sigma}^2$, if \eqref{eq:fast} holds, then \eqref{eq:slow} will follow. This shows that, even over two smooth classes $\alpha\in\{\alpha_1,\alpha_2\}$, the adaptive minimax rate can be no faster than $\phi_{n,\alpha}$.
\end{proof}

\subsection{Supporting Lemmas}
\begin{lemma}
\label{lemma:d_concentrate_1}
Suppose $f\in\Lambda_{\mbf{\alpha}}(C_{\cal{F}})$ and $\sigma^2\leq C_\sigma$for some fixed constants $C_{\cal{F}},C_\sigma$, and the joint distribution of $(\mbf{X},\varepsilon)$ satisfies the conditions in $\cal{P}_{\text{\tiny{mcv}},(\mbf{X},\varepsilon)}$. Then, the U-statistic $U_1$ defined in the proof of Proposition \ref{prop:multivariate_upper} satisfies 
\begin{align*}
\E\parr*{U_1 - \theta_1}^2 \leq C(n^{-1} \vee n^{-2}(h_1h_2)^{-1})
\end{align*}
for some positive constant $C$ that only depends on $\overline{M}_K,\under{M}_K,\mbf{\alpha},C_{\cal{F}},C_\sigma,C_0,C_\varepsilon$.
\end{lemma}
\begin{proof}
Denote $g$ as the kernel of $U_1$, that is,
\begin{align*}
g(\mbf{D}_i, \mbf{D}_j) \define K_{h_1}(X_{i,1}- X_{j,1})K_{h_2}(X_{i,2}-X_{j,2})(Y_i - Y_j)^2/2, \quad \mbf{D}_i \define (\mbf{X}_i, \varepsilon_i)^\top.
\end{align*}
Then, it holds that
\begin{align*}
\var(U_1) = {n\choose 2}^{-1} \sum_{i<j,i^\prime < j^\prime} \E\bbrace*{\parr*{g(\mbf{D}_i, \mbf{D}_j) - \theta_1}\parr*{g(\mbf{D}_{i^\prime}, \mbf{D}_{j^\prime}) - \theta_1}}.
\end{align*}
When $i,j,i^\prime, j^\prime$ take four different values, the expectation is zero. Using a similar argument as in the proof of Lemma \ref{lemma:concentrate_U1}, when $i,j,i^\prime, j^\prime$ take three different values, it holds that
\begin{align*}
\E\bbrace*{\parr*{g(\mbf{D}_i, \mbf{D}_j) - \theta_1}\parr*{g(\mbf{D}_{i^\prime}, \mbf{D}_{j^\prime}) - \theta_1}} = O(1).
\end{align*}
When they take two different values,
\begin{align*}
\E\bbrace*{\parr*{g(\mbf{D}_i, \mbf{D}_j) - \theta_1}\parr*{g(\mbf{D}_{i^\prime}, \mbf{D}_{j^\prime}) - \theta_1}} = O((h_1h_2)^{-1}).
\end{align*}
We therefore conclude that
\begin{align*}
\var(U_1) \lesssim \frac{n^3 + n^2(h_1h_2)^{-1}}{n^4} \asymp n^{-1} + n^{-2}(h_1h_2)^{-1}.
\end{align*}
This completes the proof.
\end{proof}

\begin{lemma}
\label{lemma:d_concentrate_2}
Suppose $h_1h_2 \gtrsim n^{-(2-\delta)}$ for some $0 < \delta < 2$, and the joint distribution of $(\mbf{X},\varepsilon)$ satisfies the conditions in $\cal{P}_{\text{\tiny{mcv}},(\mbf{X},\varepsilon)}$ Then, for any $u, v> 0$, the U-statistic $U_2$ defined in the proof of Theorem \ref{thm:nonpar_upper} satisfies 
\begin{align*}
\P\parr*{\abs*{U_2 - \theta_2}\geq C(v^{1/2}n^{-1/2} + u^{1/2}n^{-1}(h_1h_2)^{-1/2})} \leq C(\exp(-u) + \exp(-v))
\end{align*}
for sufficiently large $n$ and
\begin{align*}
\E\parr*{U_2 - \theta_2}^2 \leq C(n^{-1} \vee n^{-2}(h_1h_2)^{-1}),
\end{align*}
where $C$ is some positive constant that only depends on $\overline{M}_K,\under{M}_K,\mbf{\alpha},C_0$. 
\end{lemma}

\begin{proof}
The proof is similar to that of Lemma \ref{lemma:concentrate_U2}. In the application of Lemma \ref{lemma:Ubernstein}, the five quantities are of the order $B_1 \lesssim 1$, $B_2\lesssim (h_1h_2)^{-1}$, $B_3\leq n^{1/2}(h_1h_2)^{-1/2}$, $\nu_1^2\lesssim 1$, $\nu_2^2 \lesssim (h_1h_2)^{-1}$. Therefore, for any $u,v> 0$, it holds that
\begin{align*}
\P(|U_2 - \theta_2| \geq a_1v^{1/2} + a_2v + b_1u^{1/2} + b_2u + b_3u^{3/2} + b_4u^2) \leq C(\exp(-v) + \exp(-u)),
\end{align*}
where $a_1 \lesssim n^{-1/2}, a_2 \lesssim n^{-1}, b_1 \lesssim n^{-1}(h_1h_2)^{-1/2}, b_2\lesssim n^{-1}, b_3 \lesssim n^{-3/2}(h_1h_2)^{-1/2}, b_4 \lesssim n^{-2}(h_1h_2)^{-1}$. Under the condition that $h_1h_2 = \Omega(n^{-(2-\delta)})$ for some $\delta > 0$ and $n$ is sufficiently large, the dominant terms in the above inequality are $a_1$ and $b_1$, that is,
\begin{align*}
n^{-1/2} \vee n^{-1}(h_1h_2)^{-1/2}.
\end{align*}
This proves the first part of the theorem. The expectation version follows by Lemma \ref{lemma:tail_expectation}. 
\end{proof}

\begin{lemma}
\label{lemma:exp_U1}
Suppose the conditions of Proposition \ref{prop:cv_upper} hold. For the U-statistic $U_1$ defined therein, suppose $h\gtrsim n^{-(2-\delta)}$ for some $0\leq \delta\leq 2$. Then, for any $\eta>0$, there exists some positive constant $C = C(\overline{M}_K,\underline{M}_K, \alpha, \eta)$ such that
\begin{align*}
\P\parr*{\abs*{U_1 - \theta_1}\geq C(v^{1/2}n^{-1/2} + u^{1/2}n^{-1}h^{-1/2})} \leq C\parr*{\exp(-u) + \exp(-v)} + n^{-\eta}.
\end{align*}
\end{lemma}
\begin{proof}
Denote $g_1$ as the kernel of $U_1$, and consider its truncated version of defined as
\begin{align*}
\bar{g}_1(D_i,D_j) \define \frac{1}{2h}K\parr*{\frac{\XX}{h}}\bbrace*{(f(X_i) - f(X_j)) + \widetilde{\varepsilon}_{ij}}^2\mathbbm{1}\bbrace*{|\varepsilon_i|\leq \kappa_n}\mathbbm{1}\bbrace*{|\varepsilon_j|\leq \kappa_n},
\end{align*}
where $D_i = (X_i,\varepsilon_i)$ and $\kappa_n$ is some truncation parameter satisfying $\kappa_n\uparrow\infty$ as $n\rightarrow\infty$ to be specified later. We first consider the concentration of $\bar{g}_1$ around its mean value $\bar{\theta}_1\define \E\bbrace*{\bar{g}(D_i,D_j)}$. For this, we will make use of Lemma \ref{lemma:Ubernstein} by upper bounding the 5 quantities $B_1,B_2,B_3, \nu_1^2, \nu_2^2$ therein. 

For $B_1$, denoting $\td{g}_1(D) = \E\bbrace*{\bar{g}_1(D,D_j)\mid D}$, it holds that
\begin{align*}
\td{g}_1(D_i) &= \E\bbrace*{\frac{1}{2h}K\parr*{\frac{\XX}{h}}((f(X_i) - f(X_j)) + \widetilde{\varepsilon}_{ij})^2\mathbbm{1}\bbrace*{|\varepsilon_i|\leq \kappa_n}\mathbbm{1}\bbrace*{|\varepsilon_j|\leq \kappa_n} \mid D_i}\\
&\lesssim \E\bbrace*{\frac{1}{h}K\parr*{\frac{\XX}{h}}(|X_i-X_j|^{2\alpha} + \widetilde{\varepsilon}^2_{ij})\mathbbm{1}\bbrace*{|\varepsilon_i|\leq \kappa_n}\mathbbm{1}\bbrace*{|\varepsilon_j|\leq \kappa_n}\mid D_i}\\
&\lesssim \E\bbrace*{\frac{1}{h}K\parr*{\frac{\XX}{h}}|\XX|^{2\alpha}\mid X_i} + \E\bbrace*{\frac{1}{h}K\parr*{\frac{\XX}{h}}\widetilde{\varepsilon}_{ij}^2\mathbbm{1}\bbrace*{|\varepsilon_i|\leq \kappa_n}\mid X_i,\varepsilon_i}.
\end{align*}
For the first term, we have
\begin{align*}
&\ms\E\bbrace*{\frac{1}{h}K\parr*{\frac{\XX}{h}}|\XX|^{2\alpha}\mid X_i} = \int \frac{1}{h}K\parr*{\frac{u-X_i}{h}}|u-X_i|^{2\alpha}p_X(u)du\\
&= \int K(v)|vh|^{2\alpha}p_X(X_i + vh)dv\leq \sup_{u\in\RR}p_X(u)h^{2\alpha}\int K(v)|v|^{2\alpha}dv \lesssim h^{2\alpha}.
\end{align*}
For the second term, we obtain similarly that 
\begin{align*}
\E\bbrace*{\frac{1}{h}K\parr*{\frac{\XX}{h}}\widetilde{\varepsilon}_{ij}^2\mathbbm{1}\bbrace*{|\varepsilon_i|\leq \kappa_n}\mid X_i,\varepsilon_i} &\lesssim (\varepsilon_i^2+\sigma^2)\mathbbm{1}\bbrace*{|\varepsilon_i|\leq \kappa_n} \lesssim \kappa_n^2 + \sigma^2.
\end{align*}
Putting together the pieces and using the fact that $\kappa_n\uparrow\infty$ as $n\rightarrow\infty$, we obtain that 
\begin{align*}
B_1 = \|\td{g}_1\|_\infty \lesssim \kappa_n^2.
\end{align*}
Moreover, with similar analysis, it can be readily checked that $\nu_1^2\lesssim 1$.

For $B_2$, it holds that 
\begin{align*}
\abs*{\bar{g}_1(D_i,D_j)} &\lesssim \frac{1}{h}K\parr*{\frac{\XX}{h}}\bbrace*{(f(X_i)-f(X_j))^2 + \widetilde{\varepsilon}_{ij}^2}\mathbbm{1}\bbrace*{|\varepsilon_i|\leq \kappa_n}\mathbbm{1}\bbrace*{|\varepsilon_j|\leq \kappa_n}\\
&\lesssim \frac{1}{h}K\parr*{\frac{\XX}{h}}\bbrace*{|\XX|^{2\alpha} + \widetilde{\varepsilon}_{ij}^2}\mathbbm{1}\bbrace*{|\varepsilon_i|\leq \kappa_n}\mathbbm{1}\bbrace*{|\varepsilon_j|\leq \kappa_n}\\
&\lesssim \frac{1}{h}K\parr*{\frac{\XX}{h}}\abs*{\frac{\XX}{h}}^{2\alpha}h^{2\alpha} + \frac{1}{h}K\parr*{\frac{\XX}{h}}\kappa_n^2 \lesssim \frac{1}{h}\kappa_n^2.
\end{align*}
We therefore conclude that $B_2 = \|\bar{g}_1\|_\infty \lesssim h^{-1}\kappa_n^2$.

For $B_3$, we have
\begin{align*}
B_3^2 &= n\sup_{D_i}\E\bbrace*{\bar{g}_1^2(D_i,D_j)\mid D_i}\\
&= n\sup_{D_i}\E\bbrace*{\frac{1}{h^2}K^2\parr*{\frac{\XX}{h}}(f(X_i) - f(X_j) + \td{\varepsilon}_{ij})^4\mathbbm{1}\bbrace*{\abs*{\varepsilon_i}\leq \kappa_n}\mathbbm{1}\bbrace*{\abs*{\varepsilon_j}\leq \kappa_n}\mid X_i,\varepsilon_i}\\
&\lesssim \frac{nM_K}{h}\E\bbrace*{\frac{1}{h}K\parr*{\frac{\XX}{h}}\bbrace*{(f(X_i) - f(X_j))^4 + \td{\varepsilon}_{ij}^4}\mathbbm{1}\bbrace*{\abs*{\varepsilon_i}\leq \kappa_n}\mathbbm{1}\bbrace*{\abs*{\varepsilon_j}\leq \kappa_n}\mid X_i,\varepsilon_i}\\
& \lesssim \frac{nM_K}{h}\E\bbrace*{\frac{1}{h}K\parr*{\frac{\XX}{h}}\abs*{\XX}^{4\alpha}\mid X_i,\varepsilon_i} + \frac{nM_K}{h}\E\bbrace*{\frac{1}{h}K\parr*{\frac{\XX}{h}}\td{\varepsilon}_{ij}^4\mathbbm{1}\bbrace*{\abs*{\varepsilon_i}\leq \kappa_n}\mathbbm{1}\bbrace*{\abs*{\varepsilon_j}\leq \kappa_n}\mid X_i,\varepsilon_i}.
\end{align*}
Now, using similar calculation in the analysis of $B_1$, it holds that
\begin{align*}
&\E\bbrace*{\frac{1}{h}K\parr*{\frac{\XX}{h}}\abs*{\XX}^{4\alpha}\mid X_i,\varepsilon_i} \lesssim h^{4\alpha},\\
&\E\bbrace*{\frac{1}{h}K\parr*{\frac{\XX}{h}}\td{\varepsilon}_{ij}^4\mathbbm{1}\bbrace*{\abs*{\varepsilon_i}\leq \kappa_n}\mathbbm{1}\bbrace*{\abs*{\varepsilon_j}\leq \kappa_n}\mid X_i,\varepsilon_i} \lesssim \kappa_n^4.
\end{align*}
Putting together the pieces, we conclude that $B_3 \lesssim \kappa_n^4nh^{-1}$. 

Lastly, for $\nu_2^2$, we have
\begin{align*}
\ms\nu_2^2 &= \E\bbrace*{\bar{g}_1^2(D_i,D_j)}\lesssim \E\bbrace*{\frac{1}{h^2}K^2\parr*{\frac{\XX}{h}}\parr*{f(X_i) - f(X_j) + \widetilde{\varepsilon}_{ij}}^4}\\
&\lesssim h^{-1}\E\bbrace*{\frac{1}{h}K\parr*{\frac{\XX}{h}}\abs*{f(X_i) - f(X_j)}^4} + \frac{M_K}{h}\E\bbrace*{\frac{1}{h}K\parr*{\frac{\XX}{h}}\widetilde{\varepsilon}_{ij}^4}\\
&\lesssim h^{-1}(h^{4\alpha} + \E(\varepsilon_i^4))\lesssim h^{-1}.
\end{align*}

Define $\bar{U}_1$ to be the U-statistic generated by the kernel $\bar{g}_1$, that is, $\bar{U}_1 \define {n\choose 2}^{-1}\sum_{i<j}\bar{g}_1(D_i,D_j)$, and define $\bar{\theta}_1$ to be the mean value $\E\bbrace*{\bar{g}_1(D_i,D_j)}$. Define the event $\cal{E} \define \bbrace*{\abs*{\varepsilon_i} \leq \kappa_n \text{ for all $i\in[n]$}}$. Then, we have for any $t\geq 0$,
\begin{align*}
\P\parr*{\abs*{U_1 - \theta_1}\geq t} &= \P\parr*{\abs*{\bar{U}_1 - \theta_1}\geq t \bigcap \cal{E}} + \P\parr*{\cal{E}^c}\\
&\leq \P\parr*{\abs*{\bar{U}_1 - \bar{\theta}_1}\geq t - \abs*{\theta_1 - \bar{\theta}_1}} + \P\parr*{\cal{E}^c}.
\end{align*}
For the first term, we have by Lemma \ref{lemma:Ubernstein} that, for any given $u,v> 0$, it holds that
\begin{align*}
\P\parr*{\abs*{\bar{U}_1 - \bar{\theta}_1}\geq a_1v^{1/2} + a_2v + b_1u^{1/2} + b_2u + b_3u^{3/2} + b_4u^2} \leq C(\exp(-u) + \exp(-v)),
\end{align*}
where $a_1 \lesssim n^{-1/2}, a_2 \lesssim \kappa_n^2/n, b_1 \lesssim n^{-1}h^{-1/2}, b_2\lesssim n^{-1}\kappa_n^2, b_3\lesssim \kappa_n^2n^{-3/2}h^{-1/2}, b_4\lesssim n^{-2}h^{-1}\kappa_n^{2}$. Choosing $\kappa_n = \kappa\sqrt{\log n}$ for some sufficiently large constant $\kappa$, then as long as $h = \Omega(n^{-(2-\delta)})$ for some $\delta > 0$, then the dominant terms in the above inequality are $a_1$ and $b_1$, that is, 
\begin{align*}
n^{-1/2} \vee n^{-1}h^{-1/2}.
\end{align*}
Therefore, Lemma \ref{lemma:tail_expectation} implies that
\begin{align*}
\E\parr*{\abs*{\bar{U}_1 - \bar{\theta}_1}} \leq C(n^{-1/2}\vee n^{-1}h^{-1/2}).
\end{align*}
Now we calculate the difference between $\theta_1$ and $\bar{\theta}_1$. By definition, we have
\begin{align*}
\abs*{\theta_1 - \bar{\theta}_1} &= \abs*{\E\bbrace*{\frac{1}{h}K\parr*{\frac{\XX}{h}}(f(X_i) - f(X_j) + \widetilde{\varepsilon}_{ij})^2\mathbbm{1}\bbrace*{|\varepsilon_i|\geq \kappa_n\bigcup |\varepsilon_j|\geq \kappa_n}}}\\
&\lesssim \abs*{\E\bbrace*{\frac{1}{h}K\parr*{\frac{\XX}{h}}(f(X_i) - f(X_j) + \widetilde{\varepsilon}_{ij})^2\mathbbm{1}\bbrace*{|\varepsilon_i|\geq \kappa_n}}}\\
&\lesssim \abs*{\E\bbrace*{\frac{1}{h}K\parr*{\frac{\XX}{h}}(f(X_i) - f(X_j))^2\mathbbm{1}\bbrace*{|\varepsilon_i|\geq \kappa_n}}}+\abs*{\E\bbrace*{\frac{1}{h}K\parr*{\frac{\XX}{h}}\widetilde{\varepsilon}_{ij}^2\mathbbm{1}\bbrace*{|\varepsilon_i|\geq \kappa_n}}}\\
&\lesssim h^{2\alpha}\P(|\varepsilon_i|\geq \kappa_n) + \E\bbrace*{\varepsilon_i^2\mathbbm{1}\parr*{|\varepsilon_i|\geq \kappa_n}}\\
&\lesssim h^{2\alpha}n^{-\kappa^2/(2\kappa_\varepsilon^2)} + \kappa_\varepsilon^2n^{-\kappa^2/(4\kappa_\varepsilon^2)},
\end{align*}
where the second line is by symmetry, and in the last line we use the sub-Gaussianity of $\varepsilon_i$. 	Therefore, as long as $h = \Omega(n^{-(2-\delta)})$ for some $\delta > 0$, by choosing $\kappa$ large enough (depending only on $\delta,\kappa_\varepsilon,u,v$), it holds that 
\begin{align*}
\abs*{\theta_1 - \bar{\theta}_1} = o(a_1v^{1/2} + a_2v + b_1u^{1/2} + b_2u + b_3u^{3/2} + b_4u^2).
\end{align*}
Lastly, by the sub-Gaussanity of $\varepsilon_i$, it holds that
\begin{align*}
\P(\cal{E}^c) \leq n\P(\abs{\varepsilon_i}\geq \kappa_n) \lesssim n^{-\eta}
\end{align*}
for sufficiently large $\eta$ by choosing $\kappa$ correspondingly large enough. This completes the proof.
\end{proof}

The following Poissonization lemma reduces the original problem of Proposition \ref{prop:cv_lower} into the case with a random sample size, which facilitates the calculation of $\chi^2$ distance. We introduce some notation. Consider the following experiment: for any given positive integer $n$, $f(\cdot)$, $\sigma$, distribution $p_X(\cdot)$ of $X$, and distribution $p_\varepsilon(\cdot)$ of $\varepsilon$,  
\begin{itemize}
\item generate $m\sim \text{Poi}(2n)$;
\item generate $X_1,\ldots, X_m\sim p_X$ and $\varepsilon_1,\ldots,\varepsilon_m\sim p_\varepsilon$;
\item generate $Y_i = f(X_i) + \sigma\varepsilon_i$ for each $i\in[m]$.
\end{itemize}
Denote the original experiment and the above experiment as $\P$ and $\td{\P}$, respectively, where we omit the dependence on $n$, $f(\cdot)$, $\sigma$, distribution $p_X(\cdot)$ of $X$, and distribution $p_\varepsilon(\cdot)$ of $\varepsilon$.

\begin{lemma}[Poissonization]
\label{lemma:poisson}
Let $\phi_{n,\alpha}$ be defined as in Proposition \ref{prop:cv_lower}.
For any fixed $\alpha_*>0$ and set $\cal{A}\subset[\alpha_*, \infty)$, the following inequality holds:
\begin{align*}
&\ms\inf_{\td{\sigma}^2}\sup_{\alpha\in\cal{A}}\sup_{\sigma^2\leq C_\sigma}\sup_{f\in\Lambda_\alpha(C_{\cal{F}})}\sup_{\P_{(X,\varepsilon)}\in \cal{P}_{\tiny{\text{cv}},(X,\varepsilon)}}\E_{\P}\parr*{(\td{\sigma}^2 - \sigma^2)/\phi_{n,\alpha}}^2 \\
&\geq\inf_{\td{\sigma}^2}\sup_{\alpha\in\cal{A}}\sup_{\sigma^2\leq C_\sigma}\sup_{f\in\Lambda_\alpha(C_{\cal{F}})}\sup_{\P_{(X,\varepsilon)}\in \cal{P}_{\tiny{\text{cv}},(X,\varepsilon)}}\E_{\td{\P}}\parr*{(\td{\sigma}^2 - \sigma^2)/\phi_{n,\alpha}}^2  - 4nC_\sigma^2\exp(-n/6).
\end{align*}
\end{lemma}
\begin{proof}
Define the event $\cal{E}\define \{m\geq n\}$, where $m\sim\text{Poi}(2n)$. Then, a standard tail estimate has $\P(\cal{E})\geq 1 - e^{-n/6}$. For the adaptive minimax rate under $\td{\P}$, we have
\begin{align*}
&\ms\inf_{\td{\sigma}^2}\sup_{\alpha\in\cal{A}}\sup_{\sigma^2\leq C_\sigma}\sup_{f(\cdot),\P_{(X,\varepsilon)}}\E_{\td{\P}}\parr*{(\td{\sigma}^2 - \sigma^2)/\phi_{n,\alpha}}^2\\
&= \inf_{\td{\sigma}^2}\sup_{\alpha\in\cal{A}}\sup_{\sigma^2\leq C_\sigma}\sup_{f(\cdot),\P_{(X,\varepsilon)}}\E_{\td{\P}}\bbrace*{\parr*{(\td{\sigma}^2 - \sigma^2)/\phi_{n,\alpha}}^2\parr*{\mathbbm{1}\bbrace*{\cal{E}} + \mathbbm{1}\bbrace*{\cal{E}^c}}}\\
&= \inf_{\td{\sigma}^2\leq C_\sigma}\sup_{\alpha\in\cal{A}}\sup_{\sigma^2\leq C_\sigma}\sup_{f(\cdot),\P_{(X,\varepsilon)}}\E_{\td{\P}}\bbrace*{\parr*{(\td{\sigma}^2 - \sigma^2)/\phi_{n,\alpha}}^2\parr*{\mathbbm{1}\bbrace*{\cal{E}} + \mathbbm{1}\bbrace*{\cal{E}^c}}}\\
&\leq \inf_{\td{\sigma}^2\leq C_\sigma}\sup_{\alpha\in\cal{A}}\sup_{\sigma^2\leq C_\sigma}\sup_{f(\cdot),\P_{(X,\varepsilon)}}\E_{\td{\P}}\bbrace*{\parr*{(\td{\sigma}^2 - \sigma^2)/\phi_{n,\alpha}}^2\mathbbm{1}\bbrace*{\cal{E}}} + 4nC_\sigma^2\cdot\sup_{\sigma^2\leq C_\sigma}\sup_{f(\cdot), \P_{(X,\varepsilon)}}\td{\P}(\cal{E}^c)\\
&\leq \inf_{\td{\sigma}^2}\sup_{\alpha\in\cal{A}}\sup_{\sigma^2\leq C_\sigma}\sup_{f(\cdot),\P_{(X,\varepsilon)}}\E_{\td{\P}}\bbrace*{\parr*{(\td{\sigma}^2 - \sigma^2)/\phi_{n,\alpha}}^2\mathbbm{1}\bbrace*{\cal{E}}} + 4nC_\sigma^2\exp(-n/6)\\
&\leq \inf_{\td{\sigma}^2=\td{\sigma}^2\parr*{\{(X_i,Y_i)\}_{i=1}^n}}\sup_{\alpha\in\cal{A}}\sup_{\sigma^2\leq C_\sigma}\sup_{f(\cdot),\P_{(X,\varepsilon)}}\E_{\td{\P}}\bbrace*{\parr*{(\td{\sigma}^2 - \sigma^2)/\phi_{n,\alpha}}^2\mathbbm{1}\bbrace*{\cal{E}}} + 4nC_\sigma^2\exp(-n/6)\\
&\leq \inf_{\td{\sigma}^2}\sup_{\alpha\in\cal{A}}\sup_{\sigma^2\leq C_\sigma}\sup_{f(\cdot),\P_{(X,\varepsilon)}}\E_{\P}\bbrace*{\parr*{(\td{\sigma}^2 - \sigma^2)/\phi_{n,\alpha}}^2} + 4nC_\sigma^2\exp(-n/6).
\end{align*}
This completes the proof.
\end{proof}

The following lemma will also be used in the proof of Proposition \ref{prop:cv_lower}, and is a slight variation of the constrained risk inequality derived in \cite{brown1996constrained} (see Theorem 1 therein). We first introduce some notation. Consider some measurable space equipped with a class of probability measures $\{\P_\theta\}_{\theta\in\Theta}$, where $(\Theta, d)$ is a metric space. For each $\theta\in\Theta$, let $f_\theta$ be the density of $\P_\theta$ with respect to some common dominating measure $\nu$, and denote by $\E_{\theta}$ the expectation under the measure $\P_\theta$. For any estimator $T$ of $\theta$, define its risk as
\begin{align*}
R_\theta\define R(\theta, T) \define \E_\theta(T - \theta)^2 = \int (T(x) - \theta)^2f_\theta(x)\nu(dx).
\end{align*}
Now, fix two measures $\P_{\theta_1}$ and $\P_{\theta_2}$, and let $\cal{E}$ be a measurable set.  Define
\begin{align*}
I_\cal{E} \define I(\theta_1,\theta_2, \cal{E}) = \E_{\theta_1}(q^2(X)\mathbbm{1}\{X\in\cal{E}\}),
\end{align*}
where $q(x)\define f_{\theta_2}(x)/f_{\theta_1}(x)$.

\begin{lemma}
\label{lemma:constrained_risk}
Let $\Delta \define d(\theta_1,\theta_2)$ and assume that, for certain estimator $T$, $R(\theta_1,T)\leq \varepsilon^2$ and $0 < \varepsilon < \Delta((\P_{\theta_2}(\cal{E})/\sqrt{I_\cal{E}})\wedge 1)$. Then, 
\begin{align*}
R(\theta_2,T) \geq \Delta^2\P_{\theta_2}(\cal{E})^2\parr*{1 - \frac{2\varepsilon\sqrt{I_\cal{E}}}{\P_{\theta_2}(\cal{E})\Delta}}.
\end{align*}
\end{lemma}
\begin{proof}
We follow the proof of Theorem 1 in \cite{brown1996constrained} by considering the same estimator $T$ therein which minimizes $R(\theta_2,T)$ subject to the condition $R(\theta_1,T)\leq \varepsilon^2$. Then, with $\rho$ defined therein, we have
\begin{align*}
T(x) = \frac{\rho\Delta q(x)}{1+\rho q(x)}
\end{align*}
and $R(\theta_1,T) = \varepsilon^2$, so that (see Equation (2.6) in the proof of Theorem 1 in \cite{brown1996constrained})
\begin{align*}
\varepsilon^2 = \Delta^2\int \parr*{\frac{\rho q(x)}{1+\rho q(x)}}^2 f_{\theta_1}(x)\nu(dx)
\end{align*}
under the condition $\varepsilon<\Delta$. Then, by Cauchy-Schwarz, 
\begin{align*}
\varepsilon\sqrt{I_\cal{E}} &= \Delta\parr*{\int \parr*{\frac{\rho q(x)}{1+\rho q(x)}}^2 f_{\theta_1}(x)\nu(dx)}^{1/2}\parr*{\int q^2(x)f_{\theta_1}(x)\mathbbm{1}_\cal{E}(x)\nu(dx)}^{1/2}\\
&\geq \Delta\int \frac{\rho q(x)}{1+\rho q(x)}f_{\theta_2}(x)\mathbbm{1}_\cal{E}(x)\nu(dx).
\end{align*}
Then, under the condition $\varepsilon\sqrt{I_\cal{E}}\leq \Delta\P_{\theta_2}(\cal{E})$, we have
\begin{align*}
\parr*{\Delta\P_{\theta_2}(\cal{E}) - \varepsilon\sqrt{I_\cal{E}}}^2 &\leq \Delta^2\parr*{\P_{\theta_2}(\cal{E}) - \int \frac{\rho q(x)}{1+\rho q(x)}f_{\theta_2}(x)\mathbbm{1}_\cal{E}(x)\nu(dx)}^2\\
&= \Delta^2\parr*{\int \frac{1}{1+\rho q(x)}f_{\theta_2}(x)\mathbbm{1}_\cal{E}(x)\nu(dx)}^2\\
&\leq \Delta^2\parr*{\int \frac{1}{1+\rho q(x)}f_{\theta_2}(x)\nu(dx)}^2\\
&\leq \Delta^2\int \parr*{\frac{1}{1+\rho q(x)}}^2f_{\theta_2}(x)\nu(dx)\\
&= R_{\theta_2},
\end{align*}
where the last equality is true due to Equation (2.5) in the proof of Theorem 1 in \cite{brown1996constrained}. The statement then follows from $(a-b)^2\geq a^2(1-2b/a)$ for $a,b>0$.
\end{proof}

\begin{lemma}
\label{lemma:finite}
Suppose $X_1,\ldots,X_N$ are i.i.d. Poisson variables with mean value $n/N$, where $N = N_n \geq Cn^{1+\delta}$ for some positive constant $\delta$ and absolute constant $C$. Then, there exists some positive integer $f_{\max}$ that only depends on $\delta$ such that
\begin{align*}
\P\parr*{\max_{1\leq i\leq N}X_i\leq f_{\max}} \rightarrow 1
\end{align*}  
as $n\rightarrow \infty$.
\end{lemma}
\begin{proof}
Let $\lambda \define n/N$. We will show that the above statement holds for $f_{\max} = \ceil{(1+\delta)/\delta}$. For each variable $X_i$ and positive integer $k$, we have
\begin{align*}
\P\parr*{X_i \leq k} = \sum_{\ell=0}^k \frac{\lambda^\ell}{\ell!}e^{-\lambda} = e^{-\lambda}(1+\sum_{\ell=1}^k (n/N)^\ell/\ell!).
\end{align*}
Therefore, it holds that
\begin{align*}
\P\parr*{\max_{1\leq i\leq N}X_i\leq k} &= \exp(-n)\exp\parr*{N\log\parr*{1+\sum_{\ell=1}^k (n/N)^\ell/\ell!}}\\
&= \exp(-n)\exp\parr*{N\sum_{m=1}^\infty \frac{1}{m}(-1)^{m-1}\sum_{\ell_1,\ldots,\ell_m=1}^k \frac{1}{\ell_1!\ldots\ell_m!}(n/N)^{\ell_1+\ldots+\ell_m}}.
\end{align*}
The exponent in the above display is a polynomial function of $(n/N)$, and clearly the coefficient for $(n/N)^1$ is $1$. Next, we will show next that the coefficients corresponding to $(n/N)^\ell$ for $\ell=2,\ldots, k$ are all zero. For simplicity, we will show this for $\ell = k$. By Lemma \ref{lemma:identity}, the coefficient for $(n/N)^k$ is
\begin{align*}
\frac{1}{k!}\sum_{m=1}^k (-1)^{m-1}\frac{1}{m}\sum_{\ell=1}^m \ell^k(-1)^{m-\ell}{m\choose \ell} = \frac{1}{k!}\sum_{\ell=1}^k (-1)^\ell\ell^{k-1}\sum_{m=\ell}^k {m-1\choose \ell-1} = \frac{1}{k!}\sum_{\ell=1}^k(-1)^\ell\ell^{k-1}{m\choose \ell} = 0,
\end{align*}
where the first identity is by direct calculation, the second is the Hockey-Stick identity, and the third is proved in \cite{ruiz199680}.

Next, we consider the coefficient of $(n/N)^p$ for some general $p\geq k+1$, which takes the form
\begin{align*}
&\ms\sum_{m=1}^\infty \frac{1}{m}(-1)^{m-1}\sum_{\substack{1\leq \ell_1,\ldots,\ell_m\leq k\\ \ell_1+\ldots+\ell_m=p}}\frac{1}{\ell_1!\ldots\ell_m!}= \sum_{m=1}^p \frac{1}{m}(-1)^{m-1}\sum_{\substack{1\leq \ell_1,\ldots,\ell_m\leq k\\ \ell_1+\ldots+\ell_m=p}}\frac{1}{\ell_1!\ldots\ell_m!}\leq \sum_{m=1}^p \frac{1}{m}\frac{m^p}{p!}\\
&\lesssim \frac{1}{p}\frac{p^p}{p!} \lesssim e^pp^{-3/2}.
\end{align*}
Thus, in view of the fact that $N\geq Cn^{1+\delta}$, we have
\begin{align*}
\sum_{p={k+1}}^\infty (n/N)^p\sum_{m=1}^\infty \frac{1}{m}(-1)^{m-1}\sum_{\substack{1\leq \ell_1,\ldots,\ell_m\leq k\\ \ell_1+\ldots+\ell_m=p}}\frac{1}{\ell_1!\ldots\ell_m!}\lesssim \sum_{p=k+1}^\infty \parr*{\frac{en}{N}}^pp^{-3/2}\lesssim (en/N)^{k+1}.
\end{align*}
By definition of $f_{\max}$, we have $(n/N)^{f_{\max}+1}\rightarrow 0$ as $n\rightarrow \infty$, thus $\P\parr*{\max_{1\leq i\leq N}X_i\leq f_{\max}}\rightarrow 1$ as $n\rightarrow \infty$.
\end{proof}

\begin{lemma}
\label{lemma:identity}
Fix any positive integer $k$. Then, for any positive integer $1\leq m\leq k$, the following identity holds:
\begin{align*}
\sum_{\substack{1\leq \ell_1,\ldots,\ell_m\leq k\\ \ell_1+\ldots+\ell_m=k}}\frac{1}{\ell_1!\ldots\ell_m!} = \frac{1}{k!}\sum_{\ell=1}^m (-1)^{m-\ell}\ell^k{m\choose \ell}.
\end{align*}
\end{lemma}
\begin{proof}
We will prove by induction. Denote the LHS by $C(m)$. Suppose the statement holds up to $m$. Using the identity
\begin{align*}
\sum_{\substack{0\leq \ell_1,\ldots,\ell_{m+1}\leq k\\ \ell_1+\ldots+\ell_{m+1} = k}}\frac{k!}{\ell_1!\ldots\ell_{m+1}!}(m+1)^{-k} = 1,
\end{align*}
we obtain the recursive equation
\begin{align*}
C(m+1) + {m+1\choose 1}C(m) + \ldots + {m+1\choose m}C(1) = \frac{(m+1)^k}{k!}.
\end{align*}
Therefore, plugging in the equation for $C(\ell)$, $\ell=1,\ldots,m$, we obtain that
\begin{align*}
C(m+1) &= \frac{(m+1)^k}{k!} - \sum_{\ell=1}^m \frac{{m+1\choose \ell}}{k!}\sum_{j=0}^{\ell-1}(\ell-j)^k(-1)^j{\ell\choose j}\\
&= \frac{1}{k!}\fence*{(m+1)^k - \sum_{\ell=1}^m {m+1\choose \ell}\sum_{j=1}^\ell j^k(-1)^{\ell-j}{\ell\choose j}}\\
&= \frac{1}{k!}\fence*{(m+1)^k - \sum_{j=1}^m j^k\sum_{\ell=j}^m (-1)^{\ell-j}{\ell\choose j}{m+1\choose \ell}}.
\end{align*}
Thus it suffices to show that 
\begin{align*}
\sum_{\ell=j}^m (-1)^\ell {\ell\choose j}{m+1\choose \ell} = (-1)^m{m+1\choose j}.
\end{align*}
This is indeed true since the LHS equals
\begin{align*}
&\ms{m+1\choose j}\sum_{\ell=j}^m {m+1-j\choose \ell-j}(-1)^\ell = \sum_{\ell=0}^{m-j}{m+1\choose j}{m-j+1\choose \ell}(-1)^{\ell+j}\\
&= {m+1\choose j}(-1)^j\parr*{\sum_{\ell=0}^{m-j+1}{m-j+1\choose \ell}(-1)^\ell + (-1)^{m-j}} = {m+1\choose j}(-1)^m,
\end{align*}
which completes the proof.
\end{proof}


\end{document}